\documentclass{amsart}[12pt]
\usepackage[nohug]{diagrams}
\usepackage{graphicx}
\usepackage[mathscr]{eucal}
\usepackage{amssymb, amsmath, amsthm}
\usepackage{amsfonts}
\usepackage{latexsym}
\usepackage{tabularx,array}
\usepackage{enumerate}
\usepackage[all,arc]{xy}
\usepackage{mathrsfs}
\usepackage{latexsym}
\usepackage{color}
\usepackage{mathtools}
\usepackage{leftidx}
\usepackage{bm}
\usepackage{url}
\newfont{\msam}{msam10}

\newtheorem{theorem}[]{Theorem}
\newtheorem{proposition}[]{Proposition}
\newtheorem{corollary}[]{Corollary}
\newtheorem{lemma}[]{Lemma}

\theoremstyle{definition}
\newtheorem{definition}[]{Definition}
\newtheorem{remark}[]{Remark}
\newtheorem{defn}[theorem]{Definition}

\newtheorem{conj}[]{Conjecture}

\let\nc\newcommand

\def\bthm{\begin{theorem}}
\def\ethm{\end{theorem}}
\def\blemma{\begin{lemma}}
\def\elemma{\end{lemma}}
\def\bproof{\begin{proof}}
\def\eproof{\end{proof}}
\def\bprop{\begin{proposition}}
\def\eprop{\end{proposition}}
\def\bcor{\begin{corollary}}
\def\ecor{\end{corollary}}
\def\bconj{\begin{conj}}
\def\econj{\end{conj}}
\nc{\la}{\label}
\def\F{\mathcal{F}}

\def\Z{\mathbb{Z}}
\def\N{\mathbb{N}}

\def\M{\mathcal{M}}

\def\L {\boldsymbol{L}}

\def\Top{\mathrm{Top}}
\def\Set{\mathrm{Set}}
\def\sSet{\mathrm{sSet}}
\def\Cat{\mathrm{Cat}}
\def\Ch{\mathrm{Ch}}
\def\Mod{\mathrm{Mod}}
\def\sGr{\mathrm{sGr}}
\def\Gr{\mathrm{Gr}}

\def\scAlg{\mathrm{sComm}}

\def\D{\mathcal{D}}
\def\C{\mathcal{C}}

\def\M{\mathcal{M}}
\def\cN{\mathcal{N}}
\def\cB{\mathcal{B}}

\def\Ho{{\mathtt{Ho}}}

\nc{\CS}{{\tt{CS}}}
\nc{\CR}{{\tt{CR}}}
\nc{\SR}{{\tt{SR}}}
\nc{\ocolim}{{\rm ocolim}}
\nc{\Ob}{{\rm Ob}}
\nc{\Hom}{{\rm{Hom}}}
\nc{\Homcont}{{\mathcal{H}om}}
\nc{\HOM}{\underline{\rm{Hom}}}
\nc{\DER}{\underline{\rm{Der}}}
\nc{\END}{\underline{\rm{End}}}
\nc{\bSym}{\mathbf{Sym}}
\nc{\Ext}{{\rm{Ext}}}
\nc{\Map}{{\rm{Map}}}
\nc{\Rep}{{\rm{Rep}}}
\nc{\DRep}{{\rm{DRep}}}
\nc{\ODRep}{{\mathcal O}{\rm{DRep}}}
\nc{\NCRep}{\widetilde{\rm{Rep}}}
\nc{\RAct}{{\rm{RAct}}}
\nc{\bs}{\backslash}
\nc{\ob}{{\tt{Obs}}}
\nc{\CE}{\mathcal{C}}
\nc{\TP}{{T\!P}}
\nc{\un}{\underline{n}}
\nc{\um}{\underline{m}}
\nc{\rn}{\langle n \rangle}
\nc{\nn}{{{\natural} {\natural}}}
\nc{\n}{{{\natural}}}
\nc{\A}{\mathbb A}
\nc{\B}{{\mathrm{B}}}
\nc{\Ba}{\overline{\mathrm{B}}}
\nc{\bC}{\overline{C}}
\nc{\bOmega}{\boldsymbol{\Omega}}
\nc{\bB}{\boldsymbol{B}}
\nc{\EXT}{\underline{\rm{Ext}}}
\nc{\TOR}{\underline{\rm{Tor}}}
\nc{\hocolim}{\mathrm{hocolim}}
\def\H{\mathrm H}

\def\HC{\mathrm{HC}}
\def\HS{\mathrm{HS}}
\def\HB{\mathrm{HB}}
\def\HO{\mathrm{HO}}
\def\HR{\mathrm{HR}}
\def\rHC{\overline{\mathrm{HC}}}

\def\F{\mathcal F}

\nc{\End}{{\rm{End}}}
\nc{\GL}{{\rm{GL}}}
\nc{\gl}{{\mathfrak{gl}}}
\nc{\rgl}{\overline{{\mathfrak{gl}}}}
\nc{\g}{{\mathfrak{g}}}
\nc{\h}{{\mathfrak{h}}}
\nc{\PGL}{{\rm{PGL}}}
\nc{\SL}{{\rm{SL}}}
\nc{\sll}{\mathfrak{sl}}
\nc{\cn}{ \mbox{\rm c\^{o}ne} }
\nc{\PSL}{{\rm{PSL}}}
\nc{\ad}{{\rm{ad}}}
\nc{\Ad}{{\rm{Ad}}}
\nc{\dlim}{\varinjlim}
\nc{\plim}{\varprojlim}
\nc{\colim}{{{\rm colim}}}

\newcommand{\HH}{{\rm{HH}}}

\newcommand{\Tor}{{\rm{Tor}}}
\newcommand{\Spec}{{\rm{Spec}}}

\newcommand{\Sym}{\Lambda}
\newcommand{\Aut}{{\rm{Aut}}}

\newcommand{\id}{{\rm{Id}}}

\newcommand{\Tr}{{\rm{Tr}}}

\newcommand{\into}{\,\hookrightarrow\,}

\newcommand{\sto}{\xrightarrow{\sim}}

\def\bs{\backslash}

\def\ffgr{\mathfrak{G}}

\def\sset{\mathtt{sSet}}

\def\lgr{\mathbb{G}}

\def\cL{\mathcal{L}}

\newcommand{\rar}{\xrightarrow{}}

\nc{\env}{\mathrm{End}(V)}
\nc{\cC}{\mathcal{C}}
\numberwithin{equation}{section}
\numberwithin{theorem}{section}
\numberwithin{lemma}{section}
\numberwithin{proposition}{section}
\numberwithin{definition}{section}
\numberwithin{corollary}{section}
\numberwithin{example}{section}
\numberwithin{remark}{section}

\def\cC{\mathcal{C}}

\newcommand{\rH}{\overline{\mathrm{H}}}
\def\cO{\mathcal O}
\def\G{\mathbb{G}}
\def\fM{\mathfrak{M}}
\def\cP{\mathcal{P}}

\newcommand{\Lie}{\mathrm{Lie}}
\def\bdf{\begin{defn}}
\def\edf{\end{defn}}

\def\brm{\begin{remark}}
\def\erm{\end{remark}}

\theoremstyle{definition}
\def\bdf{\begin{definition}}
\def\edf{\end{definition}}

\def\cB{\mathcal B}
\newcommand{\SP}{\mathrm{SP}}

\newcommand{\bS}{{\mathbb S}}

\newcommand{\bF}{{\mathbb F}}

\newcommand{\cA}{{\mathcal A}}
\newcommand{\cF}{{\mathcal F}}

\newcommand{\cH}{{\mathcal H}}

\newcommand{\uC}{\underline{\mathcal C}}
\newcommand*{\Nwarrow}{\rotatebox[origin=c]{45}{\(\Leftarrow\)}}
\DeclareMathOperator*{\Moplus}{\text{\raisebox{0.25ex}{\scalebox{0.75}{$\bigoplus$}}}}
\DeclareMathOperator*{\Mvee}{\text{\raisebox{0.25ex}{\scalebox{0.8}{$\bigvee$}}}}
\def\arbreBA{\vcenter{\xymatrix@R=2pt@C=2pt{
&&&&\\
&&&*{}\ar@{-}[ul] & \\
&&*{}\ar@{-}[uurr] \ar@{-}[uull] \ar@{-}[d]     &&\\
&&&&
}}}

\def\arbreAB{\vcenter{\xymatrix@R=2pt@C=2pt{
&&&&\\
&*{}\ar@{-}[ur] &&& \\
&&*{}\ar@{-}[uurr] \ar@{-}[uull] \ar@{-}[d]     &&\\
&&&&
}}}

\def\arbreABC{\vcenter{\xymatrix@R=1pt@C=1pt{
&&&&&&\\
&*{}\ar@{-}[ur] &&&&& \\
&&*{}\ar@{-}[uurr] &&&&\\
&&&*{}\ar@{-}[uuurrr] \ar@{-}[uuulll] \ar@{-}[d] &&&\\
&&&&&&
}}}

\def\arbreBAC{\vcenter{\xymatrix@R=1pt@C=1pt{
&&&&&&\\
&&&*{}\ar@{-}[ul] &&& \\
&&*{}\ar@{-}[uurr] &&&&\\
&&&*{}\ar@{-}[uuurrr] \ar@{-}[uuulll] \ar@{-}[d] &&&\\
&&&&&&
}}}

\def\arbreACB{\vcenter{\xymatrix@R=1pt@C=1pt{
&&&&&&\\
&*{}\ar@{-}[ur] &&&&& \\
&&&&*{}\ar@{-}[uull] &&\\
&&&*{}\ar@{-}[uuurrr] \ar@{-}[uuulll] \ar@{-}[d] &&&\\
&&&&&&
}}}

\def\arbreBCA{\vcenter{\xymatrix@R=1pt@C=1pt{
&&&&&&\\
&&&&&*{}\ar@{-}[ul] & \\
&&*{}\ar@{-}[uurr] &&&&\\
&&&*{}\ar@{-}[uuurrr] \ar@{-}[uuulll] \ar@{-}[d] &&&\\
&&&&&&
}}}

\def\arbreCAB{\vcenter{\xymatrix@R=1pt@C=1pt{
&&&&&&\\
&&&*{}\ar@{-}[ur] &&& \\
&&&&*{}\ar@{-}[uull] &&\\
&&&*{}\ar@{-}[uuurrr] \ar@{-}[uuulll] \ar@{-}[d] &&&\\
&&&&&&
}}}

\def\arbreCBA{\vcenter{\xymatrix@R=1pt@C=1pt{
&&&&&&\\
&&&&&*{}\ar@{-}[ul] & \\
&&&&*{}\ar@{-}[uull] &&\\
&&&*{}\ar@{-}[uuurrr] \ar@{-}[uuulll] \ar@{-}[d] &&&\\
&&&&&&
}}}

\def\arbreACA{\vcenter{\xymatrix@R=1pt@C=1pt{
&&&&&&\\
&*{}\ar@{-}[ur] &&&&*{}\ar@{-}[ul] & \\
&&&&&&\\
&&&*{}\ar@{-}[uuurrr] \ar@{-}[uuulll] \ar@{-}[d] &&&\\
&&&&&&
}}}

\begin{document}

\title{Derived Character Maps of Group Representations}
%
\author{Yuri Berest}
\address{Department of Mathematics,
Cornell University, Ithaca, NY 14853-4201, USA}
\email{berest@math.cornell.edu}
\author{Ajay C. Ramadoss}
\address{Department of Mathematics,
Indiana University,
Bloomington, IN 47405, USA}
\email{ajcramad@indiana.edu}
\begin{abstract}
In this paper, we construct and study derived character maps of finite-dimensional representations of $\infty$-groups. As models for $\infty$-groups we take  homotopy simplicial groups, i.e. homotopy simplicial ${\ffgr}^{\rm op}$-algebras over the algebraic theory of
groups (in the sense of \cite{Ba02}). We define cyclic, symmetric and representation homology for `group algebras' $k[\Gamma]$ over such groups and construct canonical trace maps  relating these homology theories. In the case of one-dimensional representations, we show that our trace maps are of topological origin: they are
induced by natural maps of (iterated) loop spaces that are well studied in homotopy theory. Using this topological interpretation, we deduce some algebraic results about representation homology: in particular, we prove that the symmetric homology  of group algebras and one-dimensional representation homology are naturally isomorphic, provided the base ring $k$ is a field of characteristic zero.
We also study the behavior of the derived character maps
of $n$-dimensional representations in the stable limit as $n\to \infty$, in which case we show that they `converge' to become isomorphisms.
\end{abstract}
\maketitle
\section{Introduction}
If $ \Gamma $ is a finite group and $k$ is a field of characteristic zero, every finite-dimensional $k$-linear representation $ \varrho: \Gamma \to \GL_n(k) $  is semi-simple and determined (up to equivalence) by its character: the trace function $ \Tr_n: \langle g \rangle \mapsto \Tr_n[\varrho(g)] \,$ defined on the set $\langle\Gamma\rangle$  of conjugacy classes of elements of $ \Gamma $. Moreover, for each  $ n \ge 0 $, there are only finitely many equivalence classes of such representations. These well familiar facts from representation theory of finite groups generalize to arbitrary groups by means of algebraic geometry. For any discrete group $ \Gamma $, the set of all $n$-dimensional representations of $ \Gamma $ can be naturally given the structure of an affine algebraic variety (more precisely, an affine $k$-scheme) $ \Rep_n(\Gamma) $ called the {\it  representation variety} of $\Gamma $. 
The equivalence classes of $n$-dimensional representations of $\Gamma$ are classified by the orbits of the general linear group $ \GL_n $ that acts algebraically on $ \Rep_n(\Gamma)$ by conjugation. The classes of semi-simple representations correspond to the closed orbits\footnote{At least when $\Gamma$ is finitely generated.} and are parametrized by the affine quotient scheme
\begin{equation*}
\Rep_n(\Gamma)/\!/\GL_n(k) := \Spec\,\cO[\Rep_n(\Gamma)]^{\GL_n} \end{equation*}
called the {\it character variety} of $\Gamma$.
Now, the characters of  representations assemble into a linear map
\begin{equation}
\la{Tr0}
\Tr_n(\Gamma):\ k\langle \Gamma \rangle \to \cO[\Rep_n(\Gamma)]^{\GL_n}
\end{equation}
which is defined on the $k$-vector space $k\langle \Gamma \rangle$ spanned by the conjugacy classes of elements of $ \Gamma $. A well-known theorem of
C. Procesi \cite{Pr} asserts that the characters of $\Gamma $, i.e. the images of the map \eqref{Tr0}, generate $ \cO[\Rep_n(\Gamma)]^{\GL_n} $ as a commutative $k$-algebra, and thus, by Nullstellensatz, detect the semi-simple representations of $ \Gamma $ when $k$ is algebraically closed. In general, the equivariant geometry of $ \Rep_n(\Gamma) $ is closely related to representation theory of $\Gamma$, the geometric structure of $\GL_n$-orbits in $ \Rep_n(\Gamma) $ determining the algebraic structure of representations. Since the late 1980s this relation has been extensively studied  and exploited in many areas of mathematics, most notably in geometric group theory and low-dimensional topology (see, for example, \cite{LM85, Sik}).

Derived algebraic geometry allows one to extend --- and in some sense to complete --- this beautiful connection between representation theory and geometry. The classical representation scheme $ \Rep_n(\Gamma) $ admits a natural derived extension $ \DRep_n(\Gamma) $ called the {\it derived representation scheme}\footnote{The first construction of this kind --- the derived moduli space
$ \boldsymbol{{\mathrm R}}{\rm Loc}_G(X)$ of $G$-local systems over a pointed connected space $X$ --- was introduced by Kapranov \cite{K01}. In recent years, several other constructions have been developed in various frameworks of derived algebraic geometry (most notably, 
in the work of To\"en, Vezzosi, Pridham and Pantev (see, e.g., \cite{TV08, Prid1, Prid2, Prid3, PTVV, TP21}). A comparison of  these constructions can be found in  \cite[Appendix]{BRYIII}. The relation to our present work is briefly discussed in the end of Section~\ref{S3.2}.} 
of $ \Gamma $. In general, $\, \DRep_n(\Gamma) $ is represented by a simplicial commutative $k$-algebra $ \cO[\DRep_n(\Gamma)]$ whose homotopy groups $\,\pi_{i} \cO[\DRep_n(\Gamma)]$, $\,i \ge 0$,  are (non-abelian) homological invariants of $ \Gamma $ (or rather its classifying space $ B \Gamma $). 
Following \cite{BRYI, BRYII}, we set
\begin{equation}
\la{HR}
\HR_\ast(\Gamma,\GL_n(k)) := \pi_{\ast} \cO[\DRep_n(\Gamma)]
\end{equation}
and call \eqref{HR} the {\it $n$-dimensional representation homology} of  $\Gamma $.  
By definition, $\,\HR_\ast(\Gamma,\GL_n(k))\,$ is a graded commutative $k$-algebra, whose degree $0$ part is canonically isomorphic to the coordinate ring of  $\Rep_n(\Gamma)$:
\begin{equation}
\la{HR0}
\HR_0(\Gamma,\GL_n(k)) \cong \cO[\Rep_n(\Gamma)]\ .
\end{equation}

Apart from groups, representation homology can be also defined for various types of algebras (e.g., associative and Lie algebras, see \cite{BKR,BR,BFPRW, BFPRW2}) as well as for topological spaces (see \cite{BRYI, BRYII, BRYIII}). What is surprising perhaps is that, in the case of groups and spaces, there is a simple construction of representation homology that does not require the use of derived algebraic geometry nor a heavy machinery of homotopical algebra. This construction (discovered in \cite{BRYI}) plays a key role in the present paper: we will use it to {\it define} representation homology for {\it homotopy simplicial groups}, which are simple models for
$\infty$-groups, a more general and (from the homotopy-theoretic point of view)  more natural concept than that of a discrete or simplicial group (see \cite{Ba02}).

Now, returning to the character map \eqref{Tr0}, we observe that its domain --- the $k$-span $k \langle \Gamma \rangle $ of the conjugacy classes of elements of $\Gamma $ --- is isomorphic to the $0$-th cyclic homology of the group algebra $k[\Gamma]$:
\begin{equation}
\la{HC0}
\HC_0(k[\Gamma]) \cong k \langle \Gamma \rangle
\end{equation}
With isomorphisms \eqref{HR0} and \eqref{HC0}, we can therefore rewrite \eqref{Tr0} in the form
\begin{equation}
\la{Tr00}
\Tr_n(\Gamma):\ \HC_0(k[\Gamma])\, \to \, \HR_0(\Gamma,\GL_n(k))^{\GL_n}
\end{equation}
This suggests that there might exist a natural extension of the map \eqref{Tr0} to higher cyclic homology  with values in higher representation homology of $\Gamma$:
\begin{equation}
\la{Tr*}
\Tr_n(\Gamma)_\ast:\ \HC_\ast(k[\Gamma])\, \to \, \HR_\ast(\Gamma,\GL_n(k))^{\GL_n}
\end{equation}
We call \eqref{Tr*} the {\it derived character maps of $n$-dimensional representations of\, $\Gamma$}, and our first goal in this paper is to define and study these maps for an arbitrary homotopy simplicial group $\Gamma$.

In the case of associative algebras, the derived character maps were originally constructed 
in \cite{BKR}, using non-abelian homological algebra. This construction was extended to Lie algebras in \cite{BFPRW},
where it was shown --- among other things --- that derived character maps of Lie algebra representations are Koszul dual to the classical Loday-Quillen-Tsygan maps \cite{LQ84, T83}.
The case of groups that we study in this paper is special for several reasons. 
First, as mentioned above, in this case the representation homology  admits a simple construction that is similar to Connes' well-known construction of cyclic homology in terms of 
Tor-functors on the cyclic category. We will show that behind this `similarity' there is actually a connection: a simple formula for the derived character maps \eqref{Tr*} relating Tor-functors via classical homological algebra (see Section~\ref{S3.4} and, in particular, Definition~\ref{Gchar}).

Second, the cyclic homology of group algebras has a beautiful topological interpretation
that goes back to the work of Goodwillie, Burghelea,  Fiedorowicz and others
(see \cite[Chap. 7]{L}). Specifically, there is a natural isomorphism
\begin{equation}
\la{HCiso}
\HC_*(k[\Gamma])\,\cong\, \H_*(ES^1 \times_{S^1}\! \cL(B\Gamma);\, k)\, ,
\end{equation}
where the right-hand side is the $ S^1$-equivariant homology of the free loop space 
$ \cL(B\Gamma) := \Map(S^1, B\Gamma) $ of the classifying space of $\Gamma$. In fact,
\eqref{HCiso} is just one on the list of several classical isomorphisms 
relating algebraic homology theories associated with so-called crossed simplicial groups \cite{LF} to  (stable) homotopy theory:
\begin{eqnarray}
\la{HHH}
\HH_*(k[\Gamma]) & \cong & \H_*(\cL(B\Gamma); \,k)\,, \nonumber\\*[1ex]
\HC_*(k[\Gamma]) & \cong & \H_*(ES^1 \times_{S^1}\! \cL(B\Gamma);\, k)\,, \nonumber\\*[1ex]
\HS_*(k[\Gamma]) & \cong & \H_*(\Omega\, \Omega^{\infty} \Sigma^{\infty}\!(B\Gamma);\, k) \,,\\*[1ex]
\HB_*(k[\Gamma]) & \cong & \H_*(\Omega^{2} \Sigma (B\Gamma);\, k) \,,\nonumber\\*[1ex]
\HO_*(k[\Gamma]) & \cong & \H_*(E(\Z/2)_+ \wedge_{\Z/2} \Omega\, \Omega^{\infty} \Sigma^{\infty}(B\Gamma);\, k)\,, \nonumber
\end{eqnarray}
where $\Omega$, $ \Sigma $ and $ \Omega^{\infty} \Sigma^{\infty} $ denote the based loop, the (reduced) suspension, and the stable homotopy functors, respectively.
The first two of the above isomorphisms (for Hochschild and cyclic homology) are well known:
they were originally established in \cite{Go} and \cite{BF}, and their proofs appear in Loday's textbook \cite{L} (see also \cite{L2} for a nice self-contained exposition). The last three (for the symmetric $\HS_*$, braided $\HB_*$ and hyperoctahedral $\HO_*$ homologies) are less known: they were discovered by Fiedorowicz \cite{F} in the early 1990s, but detailed proofs were published only recently (see \cite{Au1} and \cite{Gr}).

The second (and perhaps, the main) goal of this paper is to extend the above list of isomorphisms by adding to it representation homology. To be precise, for any commutative ring $k$, let $\HR_\ast(k[\Gamma]) :=  \HR_{\ast}(\Gamma, \G_m(k)) $ denote the one-dimensional representation homology of $ \Gamma $. We prove (see Theorem~\ref{omspinf} and Lemma~\ref{HXGa}):
\begin{theorem}
\la{T1}
For any homotopy simplicial group $ \Gamma $, there is a natural isomorphism
\begin{equation}
\la{HR1}
\HR_\ast(k[\Gamma])\,\cong\, \H_\ast(\Omega\,\SP^{\infty}\! (B\Gamma);\,k)
\end{equation}
where $\,\SP^\infty\! (B\Gamma)\,$ denotes the classical Dold-Thom space of the classifying 
space of $ \Gamma\,$.
\end{theorem}

Apart from the Hochschild and cyclic theories, most interesting on the list \eqref{HHH} is the {\it symmetric homology} theory $ \HS_\ast $ introduced in \cite{F} and studied in \cite{Au1, Au2}.
Roughly speaking, $ \HS_\ast $  is defined\footnote{See Sections~\ref{S3.3} and \ref{S4.2} for precise definitions of $ \HC_\ast(k[\Gamma])$ and $ \HS_\ast(k[\Gamma]) $ in the context of homotopy simplicial groups.} in the same way as $ \HC_\ast $, with 
Connes' cyclic category $\Delta C$ replaced by the symmetric category $\Delta S$, where the 
family of the symmetric groups $ \{S^{\rm op}_{n+1}\}_{n \ge 0} $ is used instead of
the cyclic groups $ \{C_{n+1}\}_{n \ge 0} $. Now, the natural inclusions of groups $ C_{n+1} \into S_{n+1} $ extend to a functor  $\iota: \Delta C^{\rm op} \into \Delta S$, which, in turn, induces a natural map $\, \HC_*(k[\Gamma]) \to \HS_*(k[\Gamma]) \,$. It turns out that, with identifications \eqref{HHH}, this last map is induced (on homology) by a map of topological spaces
\begin{equation}
\la{CSmap}
\CS_{B \Gamma}:\ ES^1 \times_{S^1}\! \cL(B \Gamma)\, \to \, \Omega\, \Omega^{\infty} \Sigma^{\infty}\!(B \Gamma)
\end{equation}
The map \eqref{CSmap} is actually defined as a natural transformation $ \CS_{X} $ on the (homotopy) category of all pointed spaces; it was originally constructed in the paper \cite{CC}, and its relation to symmetric homology was noticed in \cite{F}. We will refer to \eqref{CSmap} as the {\it Carlsson-Cohen map} for $ B\Gamma$. 

We can now state our second observation that provides a topological interpretation of the derived character maps \eqref{Tr*} for one-dimensional representations. To shorten notation we will write the maps \eqref{Tr*} for $n=1$ as
\begin{equation}
\la{Tr1}
\Tr(\Gamma)_{\ast}:\ \HC_\ast(k[\Gamma])\, \to \, \HR_\ast(k[\Gamma])
\end{equation}
The next theorem encapsulates the main results of Section~\ref{S4.3} (see Proposition~\ref{ftsym} and Corollary~\ref{topch}), Section~\ref{S5.2} (see Proposition~\ref{csx}) and Section~\ref{S5.3} (see Proposition~\ref{srx}).
\begin{theorem}
\la{CRMap}
With isomorphisms \eqref{HHH} and \eqref{HR1}, the derived character maps \eqref{Tr1} are induced on homology by a natural map of topological spaces 
\begin{equation}
\la{CRmap}
\CR_{B \Gamma}:\ ES^1 \times_{S^1}\! \cL(B \Gamma)\, \to \,  \Omega\,\SP^{\infty}\!(B\Gamma)
\end{equation}
The map \eqref{CRmap} factors $($as a homotopy natural transformation$)$ through the Carlsson-Cohen map \eqref{CSmap}:
\begin{equation}
\la{CSR}
ES^1 \times_{S^1}\! \cL(B \Gamma) \xrightarrow{\CS_{B \Gamma}}  \Omega\, \Omega^{\infty} \Sigma^{\infty}\!(B \Gamma) \xrightarrow{\SR_{B \Gamma}}   \Omega\,\SP^{\infty}\!(B\Gamma)
\end{equation}
where the induced map $\SR $ is the $($looped once$)$ canonical natural transformation $ \Omega^\infty \Sigma^\infty \to \SP^{\infty}\!$ relating stable homotopy to $($reduced$)$ singular
homology of pointed spaces.
\end{theorem}
Theorem~\ref{CRMap} shows that, for any homotopy simplicial group $\Gamma $, the derived character map \eqref{Tr1} factors through symmetric homology, and the induced map
\begin{equation}
\la{HCSR}
\SR_{B\Gamma, \ast}:\ \HS_\ast(k[\Gamma]) \, \to \, \HR_\ast(k[\Gamma])
\end{equation}
is determined by a map of spaces that is well known in topology.
Using topological results, we then conclude (see Corollary~\ref{Cor22} and Remark~\ref{Rem5.1}):
\begin{corollary}
\la{Cor21}
If $k$ is a field of characteristic $0$, the map \eqref{HCSR} is an isomorphism,
at least when $ B \Gamma $ is a simply connected space.
\end{corollary}
The results stated above are all concerned with derived characters of one-dimensional representations. For higher dimensional representations ($n >1$), the maps \eqref{Tr*} are more complicated: in particular, they do not seem to factor through $ \HS_\ast(k[\Gamma])$, and in general, the relation between symmetric homology and representation homology remains mysterious. However, when $ n \to \infty $, things become more tractable. Assuming that $k$ is a field of characteristic $0$, we 
can naturally pass to the projective limit:
$$
\HR_\ast(\Gamma,\,\GL_\infty(k))^{\GL_{\infty}} 
:=\, \varprojlim_{n}\, \HR_\ast(\Gamma,\,\GL_n(k))^{\GL_n}
$$
and construct the {\it stable} character maps
\begin{equation} \la{stable} 
\Tr_\infty(\Gamma)_{\ast}:\ 
\rHC_\ast(k[\Gamma]) \,\to\, \HR_\ast(\Gamma,\,\GL_\infty(k))^{\GL_{\infty}} \,,
\end{equation}
where $ \rHC $ stands for the reduced cyclic homology. In this case, we have the following result, the proof
of which is parallel to \cite{BR} and outlined in the last section of the paper
(see Theorem~\ref{SRHs}).
\begin{theorem} 
\la{SRHsi}
Let $ \Gamma $ be a homotopy simplicial group such that
$ B\Gamma $ is a simply connected space of finite $($rational$)$ type. Then the stable character maps \eqref{stable} induce an algebra isomorphism
\begin{equation} 
\la{sttr}
\Sym \Tr_{\infty}(\Gamma)_\ast:\ \Sym_k[\,\rHC_\ast(k[\Gamma])] \,\sto \, \HR_\ast(\Gamma,\GL_\infty)^{\GL_\infty}\,,
\end{equation}
where $\Sym_k[\,\rHC_\ast(k[\Gamma])]$ is the graded symmetric algebra generated by the reduced cyclic homology of $k[\Gamma]$.
\end{theorem}
We close this Introduction by mentioning one application of stable character maps in derived Poisson
geometry. If $\Gamma $ is a simplicial group model of a simply-connected closed manifold $X$ of dimension $d$ (so that $X \simeq B\Gamma $), then, by \eqref{HHH}, we can identify $ \rHC_\ast(k[\Gamma]) $ with the reduced
$S^1$-equivariant homology  $\bar{\H}^{S^1}_{*}(\mathcal{L}(X); k) $ of the free loop space of $X$.
Thanks to the work of Chas and Sullivan, the latter is known to carry the so-called {\it string topology}\,
Lie bracket, making the symmetric algebra $ \Sym_k[\,\rHC_\ast(k[\Gamma])] \cong \Sym_k[\,\bar{\H}^{S^1}_{*}(\mathcal{L}(X); k)\,]
$ a graded Poisson algebra. On the other hand, the representation homology ring $\HR_\ast(\Gamma,\GL_\infty)^{\GL_\infty}$ acquires a $(2-d)$-shifted graded Poisson structure from the Poincar\'{e} duality pairing on (the cohomology of) $X$. As an application of Theorem~\ref{SRHsi}, 
we show that under the isomorphism \eqref{sttr}, these two Poisson structures agree: i.e., the 
map \eqref{sttr} is an isomorphism of graded Poisson algebras (see Corollary \ref{manifold}).

The paper is organized as follows. In Section \ref{S2}, we review basic facts from abstract homotopy 
theory concerning homotopy colimits. The new result proved in this section is Proposition \ref{ShLemma}, which we refer to as `Shapiro Lemma for model categories'. This proposition provides a key step for proofs of main theorems in Section~\ref{S4} and may be of independent interest. In Section \ref{S3}, after reviewing basic theory of homotopy simplicial groups (Section~\ref{S3.1}), we define representation homology (Section~\ref{S3.2}), and cyclic homology (Section~\ref{S3.3}) for such groups and construct the 
derived character maps relating the two (Section~\ref{S3.4}). In Section \ref{S4}, we prove 
Theorem \ref{T1} (Section~\ref{S4.1}) and then, after defining symmetric homology for homotopy simplicial groups (Section~\ref{S4.2}), we prove part of Theorem \ref{CRMap} (see Proposition \ref{ftsym} and Corollary \ref{topch} in Section~\ref{S4.3}). The proof of Theorem \ref{CRMap} is completed in Section \ref{S5}, where we study the maps \eqref{CRmap} and \eqref{CSR} in topological terms, using Goodwillie homotopy calculus and classical operads (see Proposition \ref{csx} and Proposition \ref{srx}). 
Finally, in Section \ref{S6}, we describe the stabilization procedure for the derived character maps as $n \to \infty$ and sketch the proofs of Theorem \ref{SRHsi} and Corollary \ref{manifold}. Each of the six sections begins with a short introduction that provides more details about its contents.

\subsection*{Acknowledgments} The work of the first author was partially supported by NSF grant DMS 1702372 and the Simons Collaboration Grant 712995. The second author was partially supported by NSF grant DMS 1702323.

\section{Shapiro Lemma for Model Categories}\la{S2}
In this section, we establish one general result in abstract homotopy theory related to homotopy colimits that will provide a key step in the proof of our main Theorem~\ref{T1}. We call this result (Proposition~\ref{ShLemma})
`Shapiro Lemma for model categories' as it appears to be a non-abelian generalization (in the context of model categories) of the classical Shapiro Lemma in homological algebra. We begin with a brief overview of the theory of homotopy
colimits. The standard reference for this material is the last two chapters of Hirschhorn's book \cite{Hir} 
but many results that we mention are classical and go back to  Bousfield-Kan \cite{BK72} and Quillen \cite{Q2}. Our exposition is inspired by Cisinski's beautiful paper \cite{Cis} that treats homotopy colimits axiomatically by the analogy with derived direct image functors in algebraic geometry (unlike \cite{Cis}, however, we do not use the language of Grothendieck derivators). With exception of Proposition~\ref{ShLemma}, which (to the best of our knowledge) is new, all results in this section are known. 

\subsection{Notation and conventions}
Throughout this section, $ \M $ will denote a fixed model category which we assume to be cofibrantly generated
and having all small limits and colimits. Unless stated
otherwise, $\cA$, $ \cB$, $ \cC$, $\ldots$ will denote small categories that we will use to index diagrams in $\M$. For a small category $\cA$, the category of $\cA$-diagrams in $\M$ (i.e. all functors $ \cA \to \M $) will be denoted $\M^{\cA} $. As usual, $ {\rm Cat} $ will stand for the category of all small categories with morphisms being arbitrary functors.

\subsection{Homotopy colimits}
\la{S2.2} 
For any small category $\cA$, the category $\M^{\cA} $ has a projective (aka Bousfield-Kan) model structure inherited from $\M$: the weak equivalences and fibrations are defined in this model structure objectwise, while the cofibrations are determined by the Lifting Axiom of model categories (specifically, as morphisms having
the left lifting property with respect to fibrations which are also weak equivalences in $ \M^{\cA}$). 
Since $\M$ is cofibrantly generated, such a model structure on $\M^{\cA}$ always exists and is cofibrantly generated (see \cite[Theorem 11.6.1]{Hir}).

Any functor $f: \cA \to \cB$ (a morphism in ${\rm Cat}$) defines the {\it pullback functor} on the diagram categories $\,f^\ast: \M^{\cB} \to \M^{\cA}\,$, which is obtained by restricting diagrams $ \cB \to \M $ along $f$. This pullback functor preserves objectwise weak equivalences and fibrations and --- since $\M$ has small 
colimits -- admits a left adjoint:
\begin{equation} 
\la{fsh} 
f_{!}: \M^{\cA} \rightleftarrows \M^{\cB}: f^\ast  
\end{equation}
defined on a diagram $ X: \cA \to \M $ as the left Kan extension $\,f_{!}(X) := {\rm Lan}_f(X)\,$ of  
$ X $ along $f$. Thus, the functors \eqref{fsh} form a Quillen pair between the model categories $\M^{\cA}$ and $\M^{\cB}$. Then, by Quillen's Adjunction Theorem (see \cite[8.5.8]{Hir}), they admit total (left and right) 
derived functors
\begin{equation} 
\la{Dfsh}
\L f_!: \Ho(\M^{\cA}) \rightleftarrows \Ho(\M^{\cB}): f^\ast
\end{equation}
that form an adjunction between the homotopy categories of diagrams induced by \eqref{fsh}.

The derived pushforward functor $\L f_!$ is called the {\it homotopy left Kan extension} along $f$.
It is a generalization of the classical {\it homotopy colimit functor} $\,\hocolim_{\cA}: \Ho(\M^{\cA}) \to \Ho(\M)\,$ that corresponds to the trivial map
$\, \cA \to \ast \,$, where ``$ \ast $'' denotes the one-point category (the terminal object in ${\rm Cat}$). 
In this last case, we will use the classical notation writing $ {\rm hocolim}_{\cA}(X) $ instead of 
$ \L(\cA \to \ast)_!(X) $ for $ X: \cA \to \M $.
We summarize the main properties of this construction in the following theorem.
\begin{theorem}[see \cite{Cis}]
\la{CisT}
Let $\M$ be a model category with all small limits and colimits.

\vspace{1ex}

$(1)$ $2$-{\rm Functoriality:}\, The pullback functors  $ f^\ast $ fit together
to give a strict, weakly product-preserving 2-functor $\, {\rm Cat}^{\rm op} \to {\rm CAT}\,$ that takes a small category $ \cA \in {\rm Cat} $ to the homotopy category $\Ho(\M^{\cA}) $. By adjunction, this implies, in particular, the existence of natural weak equivalences
\begin{equation} 
\la{funct} 
\L (fg)_! \simeq \L f_! \, \L g_!\  
\end{equation}
for any composable morphisms $f$ and $g$ in ${\rm Cat}$.

\vspace{1ex}

$(2)$ {\rm Reflexivity:}\, For any $ \cA \in {\rm Cat}$, the functor $ i^*: \Ho(\M^{\cA}) \to 
\Ho(\M^{{\cA}^{\delta}}) $ corresponding to the inclusion of the underlying discrete subcategory $ \cA^{\delta} \subset \cA $ is conservative, i.e. reflects the weak equivalences in $\M^{{\cA}^{\delta}}$.

\vspace{1ex}

$(3)$ {\rm Base change:}\, For any $f:\cA \to \cB$ and any object $ b \in \cB$, the $2$-commutativity of
the fibre square
$$
\begin{diagram}[small]
f\downarrow b & \rTo^{\pi} & \cA\\
\dTo^{p} & \Nwarrow & \dTo_f\\
\ast & \rTo^b & \cB\\
\end{diagram}$$
induces a change-of-base natural transformation that is a
natural weak equivalence:
$$ \L p_! \, \pi^\ast \stackrel{\sim}{\to} b^\ast \, \L f_! $$
For a diagram $X: \cA \to \M$, this simply says that 
\begin{equation} \la{basech} 
\L f_!X(b) \simeq {\rm hocolim}_{f \downarrow b}(\pi^\ast X) \ .
\end{equation}
\end{theorem}

\vspace{1ex}

\begin{remark}
\la{Rem1}
In terminology of \cite{Cis} ({\it cf.} Def.~1.6, pp. 205-206), the properties $(1)$-$(3)$ of Theorem~\ref{CisT} can be summarized by saying that the 2-functor $ \Ho(\M^{-}): {\rm Cat}^{\rm op} \to {\rm CAT} $ is a weak left derivator ({\it un d\'erivateur faible \`a gauche}) associated to the model category $\M$.
\end{remark}

\vspace{1ex}

The properties of homotopy colimits listed in Theorem~\ref{CisT} are essentially formal. The next result --- called the Cofinality Theorem --- gives a deeper property of homotopy-theoretic nature that is very useful in computations. To state this result we recall that a functor $\,f:\cA \to \cB\,$ is  {\it right homotopy cofinal} if its comma-category $ b \downarrow f $ under each object $ b \in \cB $ is (weakly) contractible, i.e. $\,B(b \downarrow f) \simeq {\rm pt}\,$. As an example, we point out that every right adjoint functor is right homotopy cofinal: 
indeed, if $\,f:\cA \to \cB\,$ admits a left adjoint, say $ g: \cB \to \cA $, then each comma-category $ b \downarrow f $ has an initial object (namely, $(b, \eta_b)$, where $ \eta_b: b \to fg(b) $ is
the unit of the adjunction evaluated at $b \in \cB$), hence $ b \downarrow f $ is contractible for any $b \in \cB $.
\begin{theorem}[Cofinality Theorem]
\la{Cofinality}
If $ f: \cA \to \cB $ is right homotopy cofinal, then the natural map
$$
{\rm hocolim}_{\cA}(f^\ast X) \stackrel{\sim}{\to} {\rm hocolim}_{\cB}(X)  
$$
is a weak equivalence for any diagram $X:\,\cB \to \M$.
\end{theorem}
For the proof of Theorem~\ref{Cofinality} we refer to \cite[Theorem 19.6.7]{Hir}. 
As an application, we prove one simple lemma that we will need for our computations. Given a functor $\,f:\cA \to \cB\,$, we recall that its {\it fibre category} $f^{-1}(b)$ over an object $ b \in \cB $ is the subcategory
of $ \cA $ consisting of all objects $ a \in \cA $ such that $ f(a) = b $ and all morphisms $ \varphi \in \Hom_{\cA}(a,a')$ such that $ f(\varphi) = {\rm Id}_b $. Note that the fibre inclusion functor $i: f^{-1}(b) \into \cA $ factors through the comma-category $f \downarrow b $ over $b$:
\begin{equation}
\la{ijpi}
\begin{diagram}[small]
f^{-1}(b) & \rInto^{i} & \cA\\
 \dTo^{j} & \ruTo_{\pi} &\\
 f \downarrow b & &\\
\end{diagram}
\end{equation}
defining the `comparison' functor
\begin{equation}
\la{prec}
j:\ f^{-1}(b) \to f \downarrow b\,\qquad a \mapsto (a, f(a)=b \stackrel{{\rm Id}}{\to} b) 
\end{equation}
Recall that a functor $f:\cA \to \cB$ is {\it precofibred} if \eqref{prec} has a left adjoint for every object $b \in \cB $  (see \cite[\S 1]{Q}).
%
%
\begin{lemma} 
\la{precof}
If $f:\cA \to \cB$ is precofibred, then 
$$ 
(\L f_!X)(b) \,\simeq\, {\rm hocolim}_{f^{-1}(b)}\,(i^\ast X)
$$
for any diagram $X: \cA \to \M$.
\end{lemma}
\begin{proof}
By assumption, the inclusion functor $\,j:f^{-1}(b) \to f \downarrow b\,$ is right adjoint, hence right homotopy cofinal. By the base change formula \eqref{basech} and Cofinality Theorem~\ref{Cofinality}, we conclude
\begin{eqnarray*}
(\L f_!X)(b) & \simeq & {\rm hocolim}_{f \downarrow b}(\pi^\ast X)\\
             & \simeq & {\rm hocolim}_{f^{-1}(b)}(j^\ast \pi^\ast X)\\
             &  = &  {\rm hocolim}_{f^{-1}(b)}((\pi j)^\ast X)\\
             & = & {\rm hocolim}_{f^{-1}(b)}(i^\ast X)
\end{eqnarray*}
where the last identification follows from \eqref{ijpi}.
\end{proof}
In practice, precofibred functors arise from the so-called Grothendieck construction
(see \cite{T}). Given a functor $\,F: \cC \to {\rm Cat}\,$ (i.e., a strict diagram of small catgories), its {\it Grothendieck construction} is defined to be the small category $\,\cC\! \smallint\! F$ with 
$\,{\rm Ob}(\cC\! \smallint\! F) := \{(c, x): c \in \cC,\, x \in F(c)\}\,$ and morphism sets 
\begin{equation}
\la{homcf}
\Hom_{\cC\! \int\! F}((c,x), \,(c',x')) := \{ (\varphi, f): \varphi \in {\rm Hom}_{\cC}(c,c'),\, f \in {\rm Hom}_{F(c')}(F(\varphi)x, x')\}\ .
\end{equation}
The composition in $\cC\! \int\! F$ is given by $\,(\varphi, f) \circ (\varphi', f') = (\varphi \, \varphi',\ f \, F(\varphi)f')\,$. The category $\cC\! \int\! F$ comes equipped with a natural (forgetful) functor 
$$
p: \, \cC\! \smallint\! F \to \cC \ ,\qquad (c,x) \mapsto c 
$$ 
which is precofibred (in fact, cofibred) over $\cC$.  Notice that $ p^{-1}(c) = F(c) $ for any object $ c\in \cC $. Hence, by Lemma \ref{precof}, for any functor  $\,X:\,\cC\! \int\! F \to \M\,$, 
\begin{equation} 
\la{ehocolim} 
(\L p_!X)(c) \simeq {\rm hocolim}_{F(c)} [X(c)]\ ,
\end{equation}
where $X(c) := i_c^{\ast}X$ is the restriction of $X$ to $F(c)$ via the inclusion functor 
$$
i_c:\, F(c) \to \cC\! \smallint\! F\,,\qquad x \mapsto (c,x)\ ,\quad (x \stackrel{f}{\to} x') \mapsto ({\rm Id}_c, f)\ . $$
Note that, by 2-functoriality of homotopy Kan extensions (see \eqref{funct}),  \eqref{ehocolim} implies the weak equivalence
\begin{equation}
\la{Thomas}
{\rm hocolim}_{\cC\! \smallint\! F}(X)\, \simeq\, {\rm hocolim}_{c \in \cC}({\rm hocolim}_{F(c)} X(c)) 
\end{equation} 
which is known as {\it Thomason's formula} for homotopy colimits over $\cC\! \int\! F$
(see \cite[Theorem 26.8]{CS}).

An important special case arises when we apply the Grothendieck construction to a
set-valued functor $\,F: \cC \to \Set \,$, regarding sets as discrete categories
(i.e. by embedding $\, \Set \into \Cat $). In this case, the category $ \cC\! \int\! F $ is usually denoted $ \cC_F $ and called
the {\it category of elements of} $F$ as its object set $ {\rm Ob}(\cC_F) $ can be identified with $\,\coprod_{c \in \cC} F(c) \,$ (we will still write the objects of
$ \cC_F $ as pairs $ (c,x)$, where $ c \in \cC$ and $ x \in F(c)$). The Hom-sets in $ \cC_F $ are given by $\,\Hom_{\cC_F}((c,x), \,(c',x')) = \{\varphi \in {\rm Hom}_{\cC}(c,c')\,:\, F(\varphi)x = x'\}\,$ ({\it cf.} \eqref{homcf}). If we take $ \M = s\Set $ to be the category of simplicial sets (equipped with standard Quillen model structure) and apply Thomason's formula \eqref{Thomas} to the trivial diagram $ X: \cC_F \to \ast $ in $ \M $, then for any functor $\,F: \cC \to \Set \,$, we get
\begin{equation}
\la{BK}
\hocolim_{\cC}(F)\, \cong \, N_*(\cC_F)
\end{equation} 
where $ N_\ast(\cC_F) $ denotes the simplicial nerve  of the category
$ \cC_F $. Formula \eqref{BK} is known as the {\it Bousfield-Kan construction} for homotopy colimits in $ s\Set $ (see \cite{BK72}).

\subsection{Homotopy coends} 
Homotopy coends are special kinds of homotopy colimits defined for {\it bifunctors}, 
i.e. the diagrams of the form $  \cC^{\rm op} \times \cC \to \M$. There is a broader range of techniques for manipulating with such homotopy colimits, which makes them more accessible for computations. The homotopy coends are defined in terms of the so-called {\it factorization category} $\F(\cC)$ introduced by Quillen \cite{Q2}. 
It can be  described as the category of elements of the bifunctor $ {\rm Hom}:\, \cC^{\rm op} \times \cC \to {\rm Set}$ of the given 
category $\cC$:
\begin{equation} \la{twar} \F(\cC):= (\cC^{\rm op} \times \cC)\!\smallint {\rm Hom}\ . \end{equation}
We will be actually dealing with the opposite category $\F(\cC)^{\rm op}$ which can be explicitly described as follows: the objects of $\F(\cC)^{\rm op}$ are the morphisms $\{\varphi:c \to d\}$ in $\cC$, and the ${\rm Hom}$-sets are commutative squares
\begin{equation} \la{twarhom} \begin{diagram}[small] d & \lTo^{\beta} & d'\\
                                                     \uTo^{\varphi} & & \uTo_{\varphi'}\\
                                                     c & \rTo^{\alpha} & c'\\
                                                     \end{diagram}\end{equation}
i.e., ${\rm Hom}_{\F(\cC)^{\rm op}}(\varphi, \varphi')$ consists of the pairs of morphisms $(\alpha,\beta)$ in $\cC$ such that $\varphi= \beta \varphi'\alpha$, with compositions defined in the obvious way.
Note that $\, \F(\cC)^{\rm op} \not\simeq \F(\cC^{\rm op})\,$ in general.
Now, there are two natural functors
\begin{equation} \la{twarsf} s^{\rm op}:\ \F(\cC)^{\rm op} \to \cC\,,\qquad (c \stackrel{\varphi}{\to} d) \mapsto c \end{equation}
\begin{equation} \la{twartf} t^{\rm op}:\ \F(\cC)^{\rm op} \to \cC^{\rm op}\,,\qquad (c \stackrel{\varphi}{\to} d) \mapsto d \end{equation}
called the (opposite) source and target functors, respectively. We have 
\begin{lemma}[Quillen] \la{rhcof}
The functors \eqref{twarsf} and \eqref{twartf} are both right homotopy cofinal. 
\end{lemma}
\begin{proof}
Since $\F(\cC)$ is defined by Grothendieck construction \eqref{twar}, the canonical (forgetful) functor 
$$ s \times t: \F(\cC) \to \cC^{\rm op} \times \cC $$
is precofibred. It follows ({\it cf.} \cite[Example, p. 94]{Q2}) that both $s:\F(\cC) \to \cC^{\rm op}$ and $t:\F(\cC) \to \cC$ are precofibred. Hence the inclusions $s^{-1}(c) \hookrightarrow s\downarrow c$ and $t^{-1}(d) \hookrightarrow t \downarrow d$ induce weak equivalences of classifying spaces 
\begin{equation}\la{clsp}  B(s^{-1}(c)) \simeq B(s \downarrow c)\quad , \quad  B(t^{-1}(d)) \simeq B(t \downarrow d)\ .\end{equation}
On the other hand, by inspection, $s^{-1}(c)= c \downarrow \cC$ and $t^{-1}(d)= (\cC \downarrow d)^{\rm op}$ are the slice and coslice categories respectively. Since both $c \downarrow \cC$ and $(\cC \downarrow d)^{\rm op}$ have initial objects, they are contractible for all $c,d \in \cC$. To complete the proof it remains to note that $(c \downarrow s^{\rm op}) = (s \downarrow c)^{\rm op}$ and $(d \downarrow t^{\rm op}) = (t \downarrow d)^{\rm op}$, where $s^{\rm op}$ and $t^{\rm op}$ are the functors \eqref{twarsf} and \eqref{twartf}. Hence 
$$
B(c \downarrow s^{\rm op}) \,=\, B(s \downarrow c)^{\rm op} \simeq B(s \downarrow c) \simeq  B(s^{-1}(c)) \simeq {\rm pt}
$$
and similarly $B(d \downarrow t^{\rm op}) \simeq {\rm pt} $. This shows that $s^{\rm op}$ and $t^{\rm op}$ are right homotopy cofinal.
\end{proof}
In view of Lemma \ref{rhcof}, the Cofinality Theorem gives two natural weak equivalences 
\begin{equation} \la{sustar} s^{\ast}: {\rm hocolim}_{\F(\cC)^{\rm op}}(s^\ast Y) \stackrel{\sim}{\to} {\rm hocolim}_{\cC}(Y)
\end{equation}
\begin{equation} \la{tustar} t^{\ast}: {\rm hocolim}_{\F(\cC)^{\rm op}}(t^\ast X) \stackrel{\sim}{\to} {\rm hocolim}_{\cC^{\rm op}}(X) 
\end{equation}
for any diagrams $X:\cC^{\rm op} \to \M$ and $Y:\cC \to \M$. These equivalences can be used to express arbitrary
homotopy colimits over $ \cC$ and $ \cC^{\rm op} $ as homotopy coends which we introduce next. Set
$$ \pi^{\rm op}:=  t^{\rm op} \times s^{\rm op}: \F(\cC)^{\rm op} \to \cC^{\rm op} \times \cC$$
and for a bifunctor $D:\cC^{\rm op} \times \cC \to \M$, define its {\it homotopy coend} by
\begin{equation} \la{hcoend} \int^{c \in \cC}_{\L} D(c,c) \,:=\, {\rm hocolim}_{\F(\cC)^{\rm op}}(\pi^{\ast}D)\,,\end{equation}
where $\pi^{\ast}D := D \circ \pi^{\rm op}: \F(\cC)^{\rm op} \to \M$.  This is indeed the (left) derived functor of the classical coend functor which is usually denoted 
\begin{equation} \la{coend} \int^{c \in \cC} D(c,c) \,:=\, {\rm colim}_{\F(\cC)^{\rm op}}(\pi^{\ast}D)\ .\end{equation}
The notation \eqref{hcoend} is very convenient as it suggests the analogy with (definite) integrals in Calculus. For example, for a bifunctor $D:(\cA \times \cB)^{\rm e} \to M$ defined on a product of two small categories $(\cA \times \cB)^{\rm e} := \cA^{\rm op} \times \cB^{\rm op} \times \cA \times \cB$ there is a natural weak equivalence
$$\int^{(a,b) \in \cA \times \cB}_{\L} D(a,b;a,b) \simeq \int^{a \in \cA}_{\L} \int^{b \in \cB}_{\L} D(a,b;a,b) $$
which is analogous to the classical Fubini Theorem in Calculus (and thus called the Fubini Theorem for homotopy coends). Another useful formula that we will need is 
\begin{equation} \la{lfhcoend} \int_{\L}^{c \in \cC} \L F [D(c,c)] \,\simeq\, \L F\left[\int_{\L}^{c \in \cC} D(c,c)\right] \end{equation}
where $F$ is a left Quillen functor between model categories. This formula is a consequence of a more general
(well-known) result that the derived functors of left Quillen functors preserve homotopy colimits (for a short proof, see e.g. \cite[Proposition 3.15]{Wer}). 

We are now in a position to state the main result of this section.
\begin{proposition}[Shapiro Lemma for model categories] 
\la{ShLemma}
Let $\M$ be a model category, $\cC$ a small category, and $F:\cC \to {\rm Set}$ any set-valued functor on $\cC$. Then, for any contravariant diagram $X:\,\cC^{\rm op} \to \M$, 
\begin{equation}
\la{sheq}
{\rm hocolim}_{\cC_F^{\rm op}}(p^\ast X)\, \simeq\, \int_{\L}^{c \in \cC} X(c) \otimes F(c) 
\end{equation}
where $\cC_F$ is the category of elements of $F$, and ``$\,\otimes$'' denotes the natural $($tensor$)$ action\footnote{That is, $\otimes$ is the bifunctor $\M \times {\rm Set} \to \M$ defined by $A \otimes S = \coprod_S A$, where $\coprod_S A$ is the coproduct of copies of $A$ indexed by the elements of $S$.} of $\,{\rm Set}$ on $\M$.
\end{proposition}
For the proof of Proposition \ref{ShLemma}, we need the following observation.
\begin{lemma} \la{nprcof} 
For any set-valued functor $F:\cC \to {\rm Set}$, the functor $ \F(p)^{\rm op}: \F(\cC_F)^{\rm op} \to \F(\cC)^{\rm op}$ induced by the canonical projection $p: \cC_F \to \cC$ is precofibred.
\end{lemma}
\begin{proof}
The proof is by direct verification.
To simplify the notation we set $ f := \F(p)^{\rm op} $ and describe first the fibre 
category $f^{-1}(\varphi)$ for $(\varphi: c \to d) \in \F(\cC)^{\rm op}$. The objects of $f^{-1}(\varphi)$ are the morphisms in $\cC_F$ of the form $(c,x) \stackrel{\varphi}{\longrightarrow} (d,y) $ such that $y=F(\varphi)(x)$. We will write the object $(c,x) \stackrel{\varphi}{\longrightarrow} (d,F(\varphi)(x))$ of $ \F(\cC_F)^{\rm op}$ as $(\varphi, x)$. Thus, 
$${\rm Ob}(f^{-1}(\varphi))\,=\,\{ (\varphi, x): x \in F(c)\}\ . $$
Further, the morphisms $(\varphi,x) \to (\varphi,y)$ in $f^{-1}(\varphi)$ are precisely the morphisms in $\F(\cC_F)^{\rm op}$, i.e., commutative diagrams of the form
$$ \begin{diagram}[small]
(d, F(\varphi)(x)) & \lTo^{\beta} & (d, F(\varphi)(y))\\
  \uTo^{\varphi} & & \uTo_{\varphi}\\
 (c,x) & \rTo^{\alpha}& (c, y)\\ 
\end{diagram}
$$
mapped to the identity by $f$. This last condition implies that $\alpha={\rm Id}_c$ and $\beta={\rm Id}_d$. Hence,
$${\rm Hom}_{f^{-1}(\varphi)}\big((\varphi, x), (\varphi, y)\big) = \begin{cases}
                                                                         \{{\rm Id}\} \quad & \text{if } x=y\\
                                                                         \varnothing \quad  & \text{ otherwise}
                                                                       \end{cases}\ .$$
Hence, $f^{-1}(\varphi) \cong F(c)$, where the set $F(c)$ is viewed as a discrete category. 

Next, for $(\varphi:c \to d) \in \F(\cC)^{\rm op}$, we have
$${\rm Ob}(f \downarrow \varphi) = \{ [\psi: k \to l, z, \alpha,\beta]: (\psi,z) \in \F(\cC_F)^{\rm op}, (\alpha,\beta) \in {\rm Hom}_{\F(\cC){\rm op}}(\psi, \varphi) \}\ .$$
The morphisms $[\psi, z,\alpha,\beta] \to [\psi', z', \alpha',\beta']$ in $f \downarrow \varphi$ are the commutative diagrams in $\cC_F$ of the form 
$$ 
\begin{diagram}[small]
 (l, F(\psi)(z)) & \lTo^{\delta} & (l', F(\psi')(z'))\\
    \uTo^{\psi} & & \uTo_{\psi'}\\
 (k,z) & \rTo^{\gamma}& (k',z')\\    
\end{diagram}
$$
such that the diagram 
$$ 
\begin{diagram}[small]
l & & \lTo^{\delta} & & l'\\
& \luTo^{\beta} & & \ruTo^{\beta'} & \\
\uTo^{\psi} & &  d & & \uTo_{\psi'} \\
 & & \uTo^{\varphi} & & \\
k &  \rTo^{\gamma}& \HonV & & k'\\
& \rdTo_{\alpha} & \vLine & \ldTo_{\alpha'} &\\
& & c & & \\
\end{diagram}
$$
commutes in $\cC$. Hence, a morphism $[\psi, z, \alpha,\beta] \to [\varphi, x, {\rm Id}_c, {\rm Id}_d]$ in $f \downarrow \varphi$ is a commutative diagram of the form 
$$ 
\begin{diagram}[small]
(l,F(\psi)(z)) & \lTo^{\beta} & (d,F(\varphi)(x))\\
 \uTo^{\psi} & & \uTo_{\varphi}\\
(k, z) & \rTo^{\alpha} & (c,x)\\ 
\end{diagram}
$$
Such a diagram exists if and only if $x=F(\alpha)(z)$, in which case it is unique. Hence, 
$${\rm Hom}_{f \downarrow \varphi}([\psi,z,\alpha,\beta], [\varphi, x,{\rm Id}_c,{\rm Id}_d]) \,=\, 
  \begin{cases}
  \{(\alpha,\beta)\} & \text{ if } x=F(\alpha)(z)\\
   \varnothing  & \text{ otherwise} \ .
  \end{cases}$$
Here, $(\alpha,\beta)$ is viewed as a morphism $(\psi, z) \to (\varphi, x)$ in $\F(\cC_F)^{\rm op}$ rather then $\F(\cC)^{\rm op}$. 

Next, consider the assignment 
$$\Phi: f \downarrow \varphi \to f^{-1}(\varphi)\,,\qquad [\psi,z,\alpha,\beta] \mapsto (\varphi,F(\alpha)(z))\ .  $$
If $(\gamma,\delta):(\psi,z) \to (\psi',z')$ is a morphism in $f \downarrow \varphi$, then $z'=F(\gamma)(z)$ and $\alpha' \circ \gamma = \alpha$. Hence, letting $\Phi$ map $(\gamma,\delta)$ to the identity on $ (\varphi,F(\alpha)(z))$ makes $\Phi$ a {\it functor}. We then note that 
$${\rm Hom}_{f^{-1}(\varphi)}(\Phi([\psi,z,\alpha,\beta]), (\varphi,x)) \,=\, {\rm Hom}_{f^{-1}(\varphi)}((\varphi,F(\alpha)(z)),(\varphi,x)) = \begin{cases}
                                                                 \{{\rm Id}\} & \text{ if } x=F(\alpha)(z)\\
                                                                  \varnothing  & \text{ otherwise}
                                                              \end{cases} \ .$$
Hence, there is a natural bijection
$${\rm Hom}_{f^{-1}(\varphi)}(\Phi([\psi,z,\alpha,\beta]), (\varphi,x))  \cong {\rm Hom}_{f \downarrow \varphi}([\psi,z,\alpha,\beta], [\varphi,x,{\rm Id}_c,{\rm Id}_d]) \,,$$
showing that $\Phi$ is left adjoint to the canonical inclusion 
$$f^{-1}(\varphi) \hookrightarrow f \downarrow \varphi\,,\qquad (\varphi,x) \mapsto [\varphi,x,{\rm Id}_c,{\rm Id}_d ]\ . $$
This shows that $f$ is precofibred, as desired.
\end{proof}
\begin{proof}[Proof of Proposition \ref{ShLemma}]
By formula \eqref{tustar} (applied to the category $\cC_F$), there is a natural weak equivalence
$$ t^{\ast}: {\rm hocolim}_{\F(\cC_F)^{\rm op}}(t^\ast p^{\ast} X) \,\xrightarrow{\sim} \,{\rm hocolim}_{\cC_F^{\rm op}}(p^{\ast} X) $$
where 
$$ 
t^{\ast}p^{\ast}X: \  
\F(\cC_F)^{\rm op}\, \xrightarrow{t^{\rm op}}\, \cC_F^{\rm op}\, 
\xrightarrow{p^{\rm op}} \, \cC^{\rm op} \,\xrightarrow{X}\,  \M\ .
$$ 
On the other hand, by definition \eqref{hcoend}, 
$$\int_{\L}^{c \in \cC} X(c) \otimes F(c)\, =\, {\rm hocolim}_{\F(\cC)^{\rm op}}[\pi^{\ast}(X \otimes F)] $$
where 
$$ \pi^{\ast}(X \otimes F):\  \F(\cC)^{\rm op} \,\xrightarrow{\pi^{\rm op}}\, 
\cC^{\rm op} \times \cC  \,\xrightarrow{X \times F} \, \M \times {\rm Set} \,\xrightarrow{\otimes} \, \M\ .
$$ 
To prove the desired proposition we thus need to show that
\begin{equation} \la{comphocolim1} {\rm hocolim}_{\F(\cC_F)^{\rm op}}(t^\ast p^{\ast} X) \,\simeq\, {\rm hocolim}_{\F(\cC)^{\rm op}}[\pi^{\ast}(X \otimes F)] \end{equation}
By Theorem \ref{CisT}(1) (see \eqref{funct}), it suffices to show that there is an weak equivalence  of $\F(\cC)^{\rm op}$-diagrams
\begin{equation} 
\la{lfsh} 
\L f_!(t^\ast p^{\ast} X)\, \simeq\, \pi^{\ast}(X \otimes F)  \, ,
\end{equation}
where $ f: \, \F(\cC_F)^{\rm op} \to \F(\cC)^{\rm op} $ is the functor induced by the canonical projection $p: \cC_F \to C\,$.
Thanks to Lemma \ref{nprcof}, we can use Lemma \ref{precof} to evaluate the homotopy Kan extension in  \eqref{lfsh} in terms of homotopy colimits over fibre categeories.
Specifically, for any $\varphi \in \F(\cC)^{\rm op}$, we have 
$$ \L f_!(t^\ast p^{\ast} X)(\varphi)\, \simeq\, {\rm hocolim}_{f^{-1}(\varphi)}(i^{\ast}t^\ast p^{\ast} X)$$
where $i:f^{-1}(\varphi) \hookrightarrow \F(\cC_F)^{\rm op}$. In the proof of Lemma \ref{nprcof}, we have described the fibre category $f^{-1}(\varphi)$: namely, $\,f^{-1}(\varphi)$ is isomorphic to 
the discrete category $ F(c) $ for any $(\varphi: c \to d) \in \F(\cC)^{\rm op}$. 
Now, since both sides of \eqref{sheq} are invariant
under weak equivalences in $ \M^{\cC^{\rm op}}\!$, we may assume that $X$ is an objectwise cofibrant diagram in $\M$. Then, since $\,
i^\ast t^\ast p^{\ast} X = i^\ast f^{\ast} t^{\ast} X = (fi)^{\ast} t^{\ast} X = t^{\ast}X(\varphi) = X(d)
\,$, we have
$$ 
{\rm hocolim}_{f^{-1}(\varphi)}(i^{\ast}t^\ast p^{\ast} X) \simeq \coprod_{F(c)} X(d)\, 
$$ 
which is precisely the value of $\,\pi^{\ast}(X \otimes F)\,$ at $\varphi$. Thus, $\L f_!(t^\ast p^{\ast} X)\varphi \simeq \pi^{\ast}(X \otimes F))\varphi $ in $\M$ for all $\varphi \in \F(\cC)^{\rm op}$. By Theorem \ref{CisT}(2), this implies \eqref{lfsh}. Summing up, we have constructed the pullback-pushforward diagram
$$ 
\begin{diagram}[small]
    & & {\rm hocolim}_{\F(\cC_F)^{\rm op}}(t^{\ast} p^{\ast} X) & & \\
   & \ldTo^{t^{\ast}} &    &  \rdTo^{\L f_!}  & \\
 {\rm hocolim}_{\cC_F^{\rm op}}(p^{\ast}X) & & & &  {\rm hocolim}_{\F(\cC)^{\rm op}}[\pi^{\ast}(X \otimes F)]\\
\end{diagram}
$$
each arrow in which is a weak equivalence. This shows that the objects in both sides of \eqref{sheq} 
are weakly equivalent in $\M$ as claimed by the proposition.
\end{proof}

In the special case, if we take $ \M = \Ch(\Mod_k) $ to be the category of chain complexes of $k$-modules equipped with standard projective model structure (see \cite[2.3.11]{Hov}), Proposition~\ref{ShLemma} implies the following classical result in homological algebra.
\begin{lemma}[Shapiro Lemma]
\la{Shap} Let $k$ be a commutative ring, $\cC$ a small category, and 
$ \Mod_k(\cC^{\rm op}) $ the $($abelian$)$ category of $\,\cC^{\rm op}$-diagrams of $k$-modules. Then, for any $ X \in \Mod_k(\cC) $ and any set-valued functor $ F:\,\cC \to \Set $, there is a natural isomorphism
$${\rm Tor}^{\cC_F}_\ast(p^{\ast}X, k) \,\cong \, {\rm Tor}^{\cC}_\ast(X, k[F]) $$
where $\,k[F]:\ \cC\, \xrightarrow{F}\, {\rm Set}\, \xrightarrow{k[\mbox{--}\,]}\, {\rm Mod}_k$
is the $k$-linear functor generated by $F$.
\end{lemma}
Lemma~\ref{Shap} appears in \cite[Appendix C.12]{L}, where it is proved 
in the special case $ X = k $ the trivial module (the constant $\,\cC^{\rm op}$-diagram valued at $k$); in the general form, it is stated, for example, in \cite{Dja19}.


\section{Representation and Cyclic Homology of Homotopy Simplicial Groups}
\la{S3}
In this section, we define representation homology of groups with coefficients in a commutative Hopf algebra $\cH$, following the approach of \cite{BRYII, BRYI}. Taking $\cH=\cO(G)$, where $G$ is an 
affine algebraic group, we then construct the derived character maps for $G$-representations of $\Gamma$. In the case when $ G = \GL_n $, these maps specialize to the character maps \eqref{Tr*} announced in the Introduction. We will work with {\it homotopy} simplicial groups (in the sense of Badzioch \cite{Ba02}), which is a more general and flexible setting than that of the usual (strict) simplicial groups used in \cite{BRYII,BRYI}. 
In Section~\ref{S3.1}, we define the classifying spaces for such groups, and in Section~\ref{S3.3},
the cyclic bar construction and cyclic homology, both of which may be of independent interest. 
We begin by reviewing the main results of \cite{Ba02} specializing to the algebraic theory of (discrete) groups.
\subsection{Homotopy simplicial groups}
\la{S3.1}
Let $\ffgr$ denote the small category whose objects $\langle n \rangle$ are the finitely generated free groups $\mathbb{F}_n = \bF\langle x_1, x_2, \ldots, x_n \rangle$, one for each $ n \ge 0 $ (with convention that $ \langle 0 \rangle $ is the trivial group), and the morphisms are arbitrary group homomorphisms. Every discrete group $\Gamma$ defines a contravariant functor $\underline{\Gamma}: \ffgr^{\rm op} \to {\rm Set}$, $\langle n \rangle \mapsto \Gamma^n$, which is simply the restriction of the Yoneda functor ${\rm Hom}(\mbox{--},\Gamma):\Gr^{\rm op} \to {\rm Set}$ to $\ffgr \subset \Gr$. More generally, every simplicial group $\Gamma \in \sGr$ (i.e. a simplicial object in $\Gr$) defines a functor 
\begin{equation} \la{s3e1} \underline{\Gamma}: \ffgr^{\rm op} \to \sSet\,,\qquad \langle n \rangle \mapsto \Gamma^n \,,\end{equation}
where $\Gamma^n$ denotes the product of $n$ copies of the underlying simplicial set of $\Gamma$ in the category $\sSet$. The functors \eqref{s3e1} can be characterized by the property of being product-preserving. To make it precise, observe that the category $\ffgr$ carries a (strict) monoidal structure $\amalg: \ffgr \times \ffgr \to \ffgr$ given by the coproduct (free product) of free groups: $\langle n \rangle \amalg \langle m \rangle = \langle n+m \rangle$. The opposite category $\ffgr^{\rm op}$ is thus equipped with the dual monoidal structure which we simply denote by $\Pi: \ffgr^{\rm op} \times \ffgr^{\rm op} \to \ffgr^{\rm op}$. Every object $\langle n \rangle^{\rm o} \in \ffgr^{\rm op}$ comes equipped with $n$ natural projections:
\begin{equation} \la{s3e2} p_{n,k}: \ \langle n \rangle^{\rm o} \to \langle 1 \rangle^{\rm o}\,,\qquad 1 \leqslant k \leqslant n\,, \end{equation}
that correspond to the canonical inclusions 
$ i_{n,k}:\, \langle 1 \rangle \hookrightarrow \langle n \rangle$, $\,x_1 \mapsto x_k $, in $\ffgr$. We say that a functor $\F: \ffgr^{\rm op} \to \sSet$ is {\it product-preserving} if the maps induced by \eqref{s3e2}
\begin{equation} 
\la{s3e3} 
\F(p_n):= \prod_{k=1}^n \F(p_{n,k}):\ \F \langle n \rangle \,\to\, (\F\langle 1 \rangle)^{n} 
\end{equation} 
are isomorphisms in $\sSet$ for all $n \geqslant 0$. It is easy to show that assigning to a simplicial group $\Gamma \in \sGr$ the functor \eqref{s3e1} defines an equivalence of categories 
\begin{equation} \la{s3e4} \sGr \xrightarrow{\sim} \sSet^{\ffgr^{\rm op}}_{\otimes} \end{equation}
where $\sSet^{\ffgr^{\rm op}}_{\otimes}$ denotes the full subcategory of product-preserving functors in the diagram category $\sSet^{\ffgr^{\rm op}}$. We will use \eqref{s3e4} to identify $\sGr=\sSet^{\ffgr^{\rm op}}_{\otimes}$, thus regarding the simplicial groups as functors of the form \eqref{s3e1}. Now, the homotopy simplicial groups are obtained by replacing the assumption that the maps \eqref{s3e3} are isomorphisms in $\sSet $ with that of being weak equivalences, which is a more natural condition from the point of view of homotopy theory. Precisely, 
\begin{definition}[Badzioch \cite{Ba02}] \la{hsgr}
A {\it homotopy simplicial group} is a functor $\F:\,\ffgr^{\rm op} \to \sSet$ that is weakly product-preserving in the sense that the maps \eqref{s3e3} are weak equivalences in $\sSet$ for all $n \geqslant 0$ (with convention that $\F\langle 0 \rangle \simeq {\rm pt}$).
\end{definition}

The category of homotopy simplicial groups (i.e. the full subcategory of all weakly product-preserving functors in $\sSet^{\ffgr^{\rm op}}$) does not carry any model structure as it is not closed under colimits. Instead, as suggested in \cite{Ba02}, one can put a new model structure on the diagram category $\sSet^{\ffgr^{\rm op}}$ in which the homotopy simplicial groups are
exhibited as fibrant objects ({\it cf.} \cite[Proposition 5.5]{Ba02}). We call this model structure the {\it Badzioch model structure} and denote it by $\sGr^{h}$. To be precise, $\sGr^h$ is defined by localizing (i.e. taking the left Boufield localization of) the standard projective model structure on $\sSet^{\ffgr^{\rm op}}$ with respect to the collection of maps $S=\{i_n: \amalg_{k=1}^n {\rm Hom}_{\ffgr}(\mbox{--},\langle 1\rangle) \to {\rm Hom}_{\ffgr}(\mbox{--},\langle n\rangle)\}_{n \ge 0}$ induced by the inclusions  $i_{n,k}: \,\langle 1\rangle \to \langle n \rangle$:
$$\sGr^h\,:=\,\mathcal{L}_S(\sSet^{\ffgr^{\rm op}})\ .  $$ 
By definition, the underlying category of $ \sGr^h $ is that of $\sSet^{\ffgr^{\rm op}}$ but the class of
weak equivalences in $ \sGr^h $ --- called {\it $S$-local weak equivalences} --- is larger: it includes the set $S$ in addition to all (objectwise) weak equivalences of diagrams in  $\sSet^{\ffgr^{\rm op}}\!\!$.
There is a canonical localization functor $\cL_S: \sSet^{\ffgr^{\rm op}} \to \sGr^h$ that takes a diagram $\Gamma \in \sSet^{\ffgr^{\rm op}}$ to its functorial fibrant replacement in the model structure $\sGr^h$. In this way, one can make any diagram in $\sSet^{\ffgr^{\rm op}}$ a homotopy simplicial group. On the other hand, the model category of (strict) simplicial groups $\sGr$ is related to $\sGr^h$ by a Quillen adjunction:
\begin{equation} 
\la{s3e5} 
K:\,\sGr^h \rightleftarrows \sGr\,:J 
\end{equation}
which is obtained by localizing (at $S$) the Quillen adjunction $\, K: \sSet^{\ffgr^{\rm op}} \rightleftarrows \sGr :J \,$ between $\sGr$ and the model category of all diagrams $ \sSet^{\ffgr^{\rm op}}$. In particular, the right adjoint functor in \eqref{s3e5} is given by the inclusion $\,J(\Gamma) = \underline{\Gamma}\,$ (see \eqref{s3e1}), while the left adjoint --- called the Badzioch rigidification functor --- 
is described explicitly in Lemma~\ref{LKformula} below.
Now, the main result of \cite{Ba02} reads: 
\begin{theorem}[Badzioch] 
\la{Bad}
The adjunction \eqref{s3e5} is a Quillen equivalence. 
\end{theorem}

\begin{remark}
\la{Rem2}
Theorem \ref{Bad} was proved in \cite{Ba02} (see {\it loc. cit.}, Theorem 6.4) for an arbitrary one-sorted algebraic theory (not only the theory of groups). It was extended  to all multi-sorted theories in \cite{Berg06}, and further to limit theories and to diagrams in model categories other than $\sSet$ in \cite{Ros15}. 
\end{remark}

\vspace{1ex}

Next, recall that there is a classical adjunction --- called the {\it Kan loop group construction} \cite{Kan1} --- that relates the model category $\sGr$ of (strict) simplicial groups to that of (reduced) simplicial sets:
\begin{equation} \la{s3e7} 
\lgr\,:\, \sSet_0 \rightleftarrows \sGr\,:\,\bar{W}
\end{equation}

The left adjoint $\lgr$ is called the {\it Kan loop group functor}, and the right adjoint $\bar{W}$ is the {\it classifying complex functor} on simplicial groups. The properties of these functors are well known and discussed in detail, for example, in \cite[Chapter V]{GJ} (see also \cite[Section 2.2]{BRYI}). Here, we mention only two important facts: first, the pair \eqref{s3e7} is a Quillen equivalence, both $\lgr$ and $\bar{W}$ being homotopy invariant functors (see \cite[V.6.4]{GJ}). Second, for any reduced simplicial set $X$, there is a weak homotopy equivalence (see \cite[V.5.11]{GJ})
\begin{equation} \la{s3e8} |\lgr(X)| \simeq \Omega|X| \end{equation}
where $\Omega|X|$ is the (Moore) based loop space of the geometric realization of $X$. The equivalence \eqref{s3e8} clarifies the topological meaning of the Kan loop group functor $\lgr$
(and justifies its name). Combining now Badzioch's Theorem \ref{Bad} with Kan's construction, we get natural equivalences of homotopy categories
\begin{equation} \la{s3e9} 
\Ho(\sGr^h) \,\xrightarrow{\L K} \, \Ho(\sGr) \,\xrightarrow{\bar{W}} \,
\Ho(\sSet_0) \,\xrightarrow{|\,\mbox{--}\,|} \, \Ho({\rm Top}_{0,\ast}) 
\end{equation}
induced by the above indicated functors. This leads us to the following definition.
\begin{definition} \la{CSpaceHSG}
For a homotopy simplicial group $\Gamma \in \sGr^h$, we define its {\it classifying space} $B\Gamma$ by 
\begin{equation} \la{s3e10} B\Gamma\,:=\, |\bar{W} \L K(\Gamma)| \end{equation}
where $\L K:\Ho(\sGr^h) \to \Ho(\sGr)$ is the derived rigidification functor (see \eqref{s3e6}).
\end{definition}
\noindent
Note that if $\Gamma $ is a (strict) simplicial group, i.e. $ \Gamma = J(\Gamma)\,$,
then $B\Gamma \cong |\bar{W}\Gamma|$, since $\L K \circ J  \cong {\rm Id}$. Thus the above definition is a natural extension of Kan's definition of classifying spaces for simplicial groups (which is, in turn, an extension of the classical definition of $B\Gamma$ for ordinary discrete groups). 

We conclude this section by giving a simple formula for the Badzioch rigidification functor that did not seem to appear explicitly in \cite{Ba02}. 
%
\begin{lemma}
\la{LKformula}
The functor $ K: \sGr^h \to \sGr $ in \eqref{s3e5} is given by the coend
\begin{equation} 
\la{Kf}
K(\Gamma) = \int^{\langle n \rangle \in \ffgr} \Gamma\langle n \rangle \otimes \mathbb{F}\langle n \rangle \,,
\end{equation}
where $\mathbb{F}:\ffgr \hookrightarrow \Gr \hookrightarrow \sGr,\,\langle n \rangle \mapsto \mathbb{F}_n$, is the canonical inclusion functor, and $\,\otimes: \sSet \times \sGr \to \sGr\,$ is the standard simplicial tensor action on the category of simplicial groups.
\end{lemma}
It follows from Lemma~\ref{LKformula} that the derived functor $\L K$ can be written as the homotopy coend
\begin{equation} 
\la{s3e6}
\L K(\Gamma) = \int^{\langle n\rangle \in \ffgr}_{\L} \Gamma\langle n \rangle \otimes \mathbb{F}\langle n \rangle\ . 
\end{equation}
For the proof of Lemma~\ref{LKformula} and formula \eqref{s3e6} (in the general setting of \cite{Ba02}) we refer to our forthcoming paper \cite{BR22b}.

\subsection{Representation homology}
\la{S3.2}
Let $k$ be a commutative ring. 
Recall that, for a  small category $\cC$, we denote by ${\rm Mod}_k(\cC)$ and ${\rm Mod}_k(\cC^{\rm op})$ the categories of all covariant and contravariant functors from $\cC$ to ${\rm Mod}_k$, respectively. It is well known that these are abelian categories with sufficiently many  projective and injective objects. Recall also (see, e.g., \cite[Appendix C.10]{L}) that there is a natural bi-additive functor
\begin{equation*} 
\la{s3e11} \mbox{--} \otimes_{\cC,k} \mbox{--}\,:\, {\rm Mod}_k(\cC^{\rm op}) \times {\rm Mod}_k(\cC) \to {\rm Mod}_k 
\end{equation*}
called the {\it functor tensor product}.
Explicitly, for  $ \M: \cC \to \Mod_k $ and $ \cN:\cC^{\rm op}\to \Mod_k $, it is defined by
\begin{equation} 
\la{s3e11}
\cN \otimes_{\cC,k} \M \, := \, \big[\Moplus_{c \in \cC} \cN(c) \otimes_k \M(c) \big]/R
\end{equation}
where $R$ is the $k$-submodule spanned by elements of the form $ \cN(\varphi)x \otimes y -x \otimes \M(\varphi)y $ for all $x \in \cN(c') $, $y \in \M(c)$ and $\varphi \in {\rm Hom}_{\cC}(c,c')$. The functor \eqref{s3e11} is right exact (with respect to each argument), preserves sums, and is left balanced. Its classical (left) derived functors with respect to each argument are canonically isomorphic and their common value is denoted by ${\rm Tor}^{\cC}_{\ast}(\cN,\M)$. More generally, we can extend the bifunctor \eqref{s3e11} to chain complexes of $\cC$-modules, i.e. the categories ${\rm Ch}({\rm Mod}_k(\cC^{\rm op}))$ and ${\rm Ch}({\rm Mod}_k(\cC))$, and define 
\begin{equation} \la{s3e12} {\rm Tor}^{\cC}_\ast(\cN,\M)\,:=\, {\rm H}_{\ast}(\cN \otimes^{\L}_{\cC,k} \M) \end{equation}
for any $\cN \in {\rm Ch}({\rm Mod}_k(\cC^{\rm op}))$ and $\M \in {\rm Ch}({\rm Mod}_k(\cC))$. Note that $\cN \otimes^{\L}_{\cC,k} \M$ is an object in the (unbounded) derived category $\mathcal{D}(k)= \mathcal{D}({\rm Mod}_k)$ of $k$-modules, and \eqref{s3e12} is just the usual hyper-Tor functor on chain complexes. Next, observe that there is a natural functor 
\begin{equation} \la{s3e13} \sSet^{\cC^{\rm op}} \,\xrightarrow{\,k[\,\mbox{--}\,]\,} \, 
{\rm sMod}_k({\cC}^{\rm op}) \stackrel{N}{\into} \, {\rm Ch}({\rm Mod}_k(\cC^{\rm op}))
\end{equation}
transforming the $\cC^{\rm op}$-diagrams in $ \sSet $ (simplicial presheaves on $\cC$) to 
chain complexes over ${\rm Mod}_k(\cC^{\rm op})$.  Here $N$ stands for 
the classical Dold-Kan normalization functor that identifies simplicial objects in ${\rm Mod}_k(\cC^{\rm op})$ with non-negatively graded chain complexes in ${\rm Ch}({\rm Mod}_k(\cC^{\rm op}))$. Abusing notation, we will write the functor \eqref{s3e13} simply as $\,k[\,\mbox{--}\,]\,$. 

We are now in a position to define representation homology of homotopy simplicial groups with coefficients in commutative Hopf algebras. We recall the well-known fact (see, e.g., \cite[Proposition~14.1.6]{R20}) that every such algebra $ \cH $ defines a covariant functor (a left $\ffgr$-module) by the rule
\begin{equation} 
\la{halg}
\underline{\cH}:\ffgr \to {\rm Mod}_k\ ,\quad  \langle n \rangle \mapsto \cH^{\otimes n}\ .
\end{equation}
In particular if $G$ is an affine algebraic group (e.g., $G={\rm GL}_n(k)$) with coordinate ring $\cH = \cO(G)$, then \eqref{halg} can be written in the form $\langle n \rangle \mapsto \cO[{\rm Rep}_G(\langle n \rangle)]$ which makes the functoriality clear. 

\begin{definition} \la{RHDef}
The {\it representation homology} of a homotopy simplicial group $\Gamma \in \sGr^h$ with coefficients in $\cH$ is defined by 
$$
{\rm HR}_{\ast}(\Gamma, \cH)\,:=\,{\rm Tor}^{\ffgr}_{\ast}(k[\Gamma],\underline{\cH}) \ .
$$
In the special case when $G$ is an affine algebraic group over $k$ and $ \cH = \cO(G) $, we simply write
${\rm HR}_{\ast}(\Gamma, G) $ instead of ${\rm HR}_{\ast}(\Gamma, \cO(G))\,$.
\end{definition}

The next lemma shows that the above definition agrees with the Badzioch model structure on $\sGr^h$.
\begin{lemma} \la{RHBadM}
If two homotopy simplicial groups $\Gamma$ and $\Gamma'$ are weakly equivalent in $\sGr^h$, then 
\begin{equation}
\la{equihr}
\HR_*(\Gamma, \cH)\,\cong\, \HR_*(\Gamma', \cH)
\end{equation} 
for any commutative Hopf algebra $\cH$. 
\end{lemma}
\begin{proof}
By \cite[Proposition 5.6]{Ba02}, if two homotopy simplicial groups $\Gamma$ and $\,\Gamma'$ are $S$-locally weakly equivalent, then their underlying diagrams are, in fact, weakly equivalent in $\sSet^{\ffgr^{\rm op}}$.  It therefore suffices to show that \eqref{equihr} holds for any objectwise weak equivalent diagrams $\Gamma, \Gamma':\, \ffgr^{\rm op} \to \sSet$. To this end, observe that the linearization functor 
\begin{equation} \la{s3e14} 
k[\,\mbox{--}\,] \,:\, \sSet^{\ffgr^{\rm op}} \to {\rm sMod}_k(\ffgr^{\rm op})
\end{equation}
is left Quillen with respect to the projective model structures (its right adjoint is the forgetful functor). Since the weak equivalences in $\sSet^{\ffgr^{\rm op}}$ are defined objectwise and the model structure on $\sSet$ is cofibrant, being left Quillen, the functor \eqref{s3e14} is actually homotopy invariant: i.e., it maps weakly equivalent objects in $\sSet^{\ffgr^{\rm op}}$ to weakly equivalent objects in ${\rm sMod}_k(\ffgr^{\rm op})$, which, in turn, are transformed by the normalization functor $N$ to quasi-isomorphic complexes in ${\rm Ch}({\rm Mod}_k(\ffgr^{\rm op}))$. Thus if $\Gamma \simeq \Gamma'$ in $ \sSet^{\ffgr^{\rm op}}$, then $\,k[\Gamma] \otimes^{\L}_{\ffgr,k} \underline{\cH} \,\simeq \, k[\Gamma'] \otimes^{\L}_{\ffgr,k} \underline{\cH}\,$ in $ \D(k)$, which implies \eqref{equihr}. 
\end{proof}
\begin{remark}
\la{Rem3}
The result of Lemma \ref{RHBadM} does {\it not} extend to arbitrary objects in $\sGr^h$, since the functor
\eqref{s3e14} does not preserve $S$-local weak equivalences in general. The latter can be seen easily by
looking at \eqref{s3e14} evaluated at representable simplicial presheaves in $\sSet^{\ffgr^{\rm op}}$.
\end{remark}

\vspace{1ex}

An important consequence of Lemma \ref{RHBadM} is that the representation homology of a homotopy simplicial group $ \Gamma $  depends only on the homotopy type of its classifying space $B\Gamma$ (Definition \ref{CSpaceHSG}). In fact, we have
\begin{equation} \la{s3e15} {\rm HR}_{\ast}(\Gamma, \cH) \,\cong\, {\rm HR}_\ast(B\Gamma, \cH)\end{equation}
where the `${\rm HR}$' in the right-hand side stands for representation homology of topological spaces
as defined in \cite{BRYI}, using a (non-abelian) derived representation functor (see {\it loc. cit.},
Definition~3.1). Indeed, by Badzioch's results ({\it cf.} \cite[Theorem 3.1]{Ba02}), every homotopy simplicial group $\Gamma$ is weakly equivalent to a strict one, say $\Gamma'$; hence 
\begin{equation} \la{s3e16} 
B\Gamma \,\simeq\, B\Gamma' \,\simeq\, \bar{W}\Gamma'\ .\end{equation}
On the other hand, by \cite[Theorem 4.2]{BRYI}, $\,{\rm HR}_{\ast}(\Gamma',\cH) \cong {\rm HR}_{\ast}(\bar{W}\Gamma',\cH)\,$, which together with \eqref{s3e16} and the isomorphism \eqref{equihr} of Lemma~\ref{RHBadM} implies \eqref{s3e15}.

We conclude this section by briefly explaining how our approach (Definition~\ref{RHDef}) relates to derived algebraic geometry (DAG). For a model of DAG, we will take the simplicial presheaf model developed in \cite{TV08}. Given a homotopy simplicial group $ \Gamma \in \sGr^h $ and an affine algebraic group (scheme) $G$ over $k$ with coordinate algebra $\cH = \cO(G)$, we introduce the 
{\it derived representation scheme of $\Gamma $ in $G$}: 
\begin{equation}\la{DRep}
\DRep_{G}(\Gamma) := \boldsymbol{\mathrm{R}}\Spec\,(k[\Gamma] \otimes^{\L}_{\ffgr} \cO(G))\ .
\end{equation}
Here `$\, \boldsymbol{\mathrm{R}}\Spec \,$'  denotes the To\"en-Vezzosi derived Yoneda functor  that assigns to a (homotopy) simplicial commutative algebra $A$ --- a derived ring in terminology of \cite{TV08} --- the simplicial presheaf (prestack)  
$$
\boldsymbol{\mathrm{R}}\Spec(A):\, \mathrm{dAff}_k^{\rm op} := \scAlg_k \ \to \ \sSet\ ,\quad
B \,\mapsto\, {\mathrm{Map}}(QA,\,B)\ ,
$$
where $QA$ is a cofibrant replacement of $A$ and `$ \mathrm{Map} $' is the simplicial mapping space (function complex) in $ \scAlg_k $. The prestack  $ \boldsymbol{\mathrm{R}}\Spec(A) $ satisfies the descent condition for \'etale hypercoverings and hence defines a derived stack (which is a derived affine scheme in the sense of \cite{TV08}). On the other hand, for any pointed space (simplicial set) $ X $, we can define the {\it pointed} mapping stack $ \boldsymbol{\mathrm{Map}}_*(X, BG)$ to be the homotopy fibre of the canonical map in the (homotopy) category of derived stacks:
\begin{equation}\la{map*}
\boldsymbol{\mathrm{Map}}_*(X, BG) := {\rm hofib}\,[\boldsymbol{\mathrm{Map}}(X, BG) \,\to\, BG]\ ,
\end{equation}
where $ \boldsymbol{\mathrm{Map}}(X, BG)$ stands for the (unpointed) derived mapping stack  defined in \cite[2.2.6.2]{TV08}. This last mapping stack is a basic object of derived algebraic geometry that plays an important role in applications (see, e.g., \cite{PTVV}). Now, its relation to representation homology is clarified by the following
\begin{proposition}[see \cite{BRYIII}]
\la{PDAG}
There is a $($weak$)$ equivalence of derived stacks
$$
\DRep_{G}(\Gamma)\,\simeq\, \boldsymbol{\mathrm{Map}}_*(B \Gamma, BG) 
$$
\end{proposition}
For a detailed proof of Proposition~\ref{PDAG} and more explanations we refer to \cite[Appendix~A.1]{BRYIII}.

\subsection{Cyclic homology}
\la{S3.3}
We now define cyclic homology for homotopy simplicial groups. To this end, we will associate to each $\Gamma \in \sGr^h$ a cyclic module $k[B^{\rm cyc}\Gamma]$ that generalizes the classical cyclic bar construction $ C_{\ast}(k[\Gamma])$  when $\Gamma$ is an ordinary discrete group. 

Let $\Delta$ denote the standard (co)simplicial category whose objects are finite ordered sets $[n]=\{0<1<\cdots<n\}$ and morphisms are (nonstrictly) order preserving maps. The category $\Delta$ is generated by two families of maps $d^i_n:[n-1] \to [n]$ ($0 \leqslant i \leqslant n$, $n \geqslant 1$) and $s^j_n:[n+1] \to [n]$ ($0 \leqslant j \leqslant n$, $n \geqslant 0$), called the (co)face and (co)degeneracy maps respectively. These maps satisfy the standard (co)simplicial relations listed, for example, in \cite[Appendix B.3]{L}. Connes' {\it cyclic category} $\Delta C $ is a natural extension of $\Delta$ that has the same objects and is generated by the morphisms of $ \Delta $ and the cyclic maps $ \tau_n:[n] \to [n]$, $n \geqslant 0$, satisfying $ \tau_n^{n+1}={\rm Id}$ (see \cite[6.11]{L}). In fact, $\Delta C $ can be characterized by the two properties:

\begin{enumerate}
\item[(Cyc1)] For each $n \geqslant 0$, ${\rm Aut}_{\Delta C}([n]) \cong C_{n+1}\,$, where $ C_{n+1} = \Z/(n+1)\Z \,$,
\item[(Cyc2)] Any morphism $f:[n] \to [m]$ in $\Delta C $ can be factored uniquely as $f=g \circ \varphi$, where $g \in {\rm Hom}_{\Delta}([n],[m])$ and $\varphi \in {\rm Aut}_{\Delta C}([n])\,$,
\end{enumerate}

\noindent
These properties show that $\Delta C $ is a {\it crossed simplicial category}
associated to the family of cyclic groups $\{C_{n+1}\}_{n \geqslant 0}$ (see \cite[6.3.0]{L}). 

Now, if $\Gamma$ is an ordinary discrete group, there is a natural functor 
  \begin{equation} \la{s3e18} B^{\rm cyc}_\ast\Gamma\,:\, \Delta C^{\rm op} \to \Set\end{equation}
  called the cyclic bar construction of $\Gamma$ that has the property that $k[B^{\rm cyc}_\ast\Gamma] \cong C_{\ast}(k[\Gamma])$, where $C_\ast(k[\Gamma])$ is the standard cyclic module associated to $k[\Gamma]$ as an associative $k$-algebra. Explicitly, the functor \eqref{s3e18} is defined by (see \cite[7.3.10]{L})
  \begin{eqnarray*}
  d_i(g_0,\ldots,g_n) & = & \begin{cases}
                             (g_0, \ldots, g_{i-1},g_ig_{i+1}, \ldots, g_n) & 0 \leqslant i < n\\
                              (g_ng_0,g_1,\ldots, g_{n-1}) & i=n
                             \end{cases}\\
  s_j(g_0,\ldots,g_n) &= &  (g_0,\ldots,g_j,1,g_{j+1},\ldots, g_n)\\
  t_n(g_0,\ldots, g_n) & = & (g_n, g_0, g_1,\ldots, g_{n-1}) 
  \end{eqnarray*}
  where $(g_0,\ldots,g_n) \in \Gamma^{n+1}$. Clearly, $\Gamma \mapsto B^{\rm cyc}_\ast\Gamma$ gives a functor $B^{\rm cyc}_{\ast}:\Gr \to \Set^{\Delta C^{\rm op}}$. If we identify $\Gr=\Set^{\ffgr^{\rm op}}_{\otimes}$ as in \eqref{s3e1}, then it turns out that $ B^{\rm cyc}_{\ast} $ coincides with the pull-back functor for a certain natural map $\Psi_{\rm cyc}: \Delta C \to \ffgr$ in ${\rm Cat}$. Specifically,  
  \begin{equation} \la{s3e19} 
  \Psi_{\rm cyc}: \Delta C \to \ffgr \end{equation}
  is defined on objects by 
  $$\Psi_{\rm cyc}([n]) \,:= \,\langle n+1 \rangle \,=\, \mathbb{F}\langle x_0,\ldots,x_n\rangle $$
  and on morphisms by the following formulas 
  \begin{equation} \la{s3e20}
\left. \begin{aligned}
\Psi_{\rm cyc}(d^i_n) &: \langle n \rangle \rar \langle n+1 \rangle\,, \qquad (x_0, x_1, \ldots, x_{n-1}) \mapsto \begin{cases} (x_0,\ldots, x_{i-1}, x_i x_{i+1}, \ldots, x_n)\ , &  0 \leq i < n\\
(x_nx_0, x_1,\ldots, x_{n-1})\ , & i=n \end{cases} \\
\Psi_{\rm cyc}(s^j_n) &: \langle n+2\rangle \rar \langle n+1 \rangle \,,  \qquad (x_0, \ldots, x_{n+1}) \mapsto (x_0, \ldots, x_j,1,x_{j+1},\ldots, x_n)\ ,\\
\Psi_{\rm cyc}(\tau_n)&: \langle n+1 \rangle \rar \langle n+1 \rangle\,,  \qquad (x_0,x_1, \ldots,x_n) \mapsto (x_n,x_0, x_1, \ldots, x_{n-1})\ .\\
\end{aligned} \right.
\end{equation}
where $ (x_0, x_1, \ldots, x_n) $ is an ordered sequence of generators of the free group $ \mathbb{F}\langle x_0,\ldots,x_n\rangle $.
\begin{lemma} \la{Bcycpb}
 For any discrete group  $\Gamma$ there is a natural isomorphism of cyclic sets 
 $$B^{\rm cyc}\Gamma \,\cong\, \Psi^{\ast}_{\rm cyc}(\underline{\Gamma}) $$
 where $\underline{\Gamma}:\ffgr^{\rm op} \to \Set$ is the functor corresponding to $\Gamma$ under the identification \eqref{s3e1}. 
 \end{lemma}
 \begin{proof}
 Straightforward.
 \end{proof}
 \begin{remark}
 The functor \eqref{s3e19} was defined in \cite{BRZ21} on a slightly larger -- the so-called epicyclic -- category $\Delta \Psi$, which is an extension of  $ \Delta C $  describing the Adams operations on cyclic modules. 
 \end{remark}

 \vspace{1ex}
 
Lemma \ref{Bcycpb} motivates the following definition.
 \begin{definition} \la{CBHSG}
 For a homotopy simplicial group $\Gamma \in \sGr^h$, we define its {\it cyclic bar construction} by
 \begin{equation} \la{s3e21} B^{\rm cyc}\Gamma\,:=\, \Psi^{\ast}_{\rm cyc}(\Gamma):\ \Delta C^{\rm op} \to \sSet\end{equation}
 and its {\it cyclic homology} by
 \begin{equation} \la{s3e22} {\rm HC}_\ast(k[\Gamma]) 
 \,:=\, \Tor^{\Delta C^{\rm op}}_{\ast}(k,\,k[B^{\rm cyc}\,\Gamma])
 \,\cong\, \Tor^{\Delta C}_{\ast}(k[B^{\rm cyc}\,\Gamma],\,k) 
\end{equation} 
 \end{definition}
 The same argument as in (the proof of) Lemma \ref{RHBadM} shows that ${\rm HC}_\ast(k[\Gamma])$ depends only on the homotopy type of $\Gamma$ in the Badzioch model category $\sGr^h$, and hence, on the homotopy type of its classifying space $B\Gamma$. In view of Lemma \ref{Bcycpb}, the above definition of ${\rm HC}_{\ast}(k[\Gamma])$ for
 $ \Gamma $ an ordinary discrete group
 coincides with the classical (Connes') definition of cyclic homology of group algebras (see \cite[6.2.8]{L}).
 
 \subsection{Derived character maps}
 \la{S3.4}
 Next, we will construct a family of natural transformations relating the cyclic homology to representation homology of a homotopy simplicial group $\Gamma$. In the special case when $ \cH = \cO(\GL_n) $, this family contains a distinguished element determined by the usual trace $\Tr_n$ that gives the derived character map \eqref{Tr*} announced in the Introduction. With our current definitions of representation and cyclic homology the construction is actually very simple.
 It is based on two lemmas. The first one is a standard result of homological algebra that simply exhibits the naturality of derived tensor products \eqref{s3e12}.
 \begin{lemma}
 \la{LTor}
 Let $f: \cA \to \cB $ be a functor between small categories. For any complexes $ \cN \in \Ch(\Mod_k\, \cB^{\rm op}) $ and $ \M \in \Ch(\Mod_k\, \cB) $, there is a natural map $\,f^*\cN \otimes^{\L}_{\cA, k} f^*\M \to \cN \otimes^{\L}_{\cB, k} \M\,$ in the derived category $ \D(k)$ of $k$-modules that induces
 $$
 f^*:\ \Tor^{\cA}_{\ast}(f^*\cN,\,f^*\M)\,\to\,
 \Tor^{\cB}_{\ast}(\cN,\,\M)
 $$
 \end{lemma}
To apply this lemma in our situation we recall that
every commutative Hopf $k$-algebra $\cH$ defines the covariant functor $\underline{\cH}:\ffgr \to \Mod_k$ by formula \eqref{halg}. Restricting this functor via
the morphism \eqref{s3e19} gives rise to a cocyclic $k$-module that we denote
$$
B_{\rm cyc} \cH 
:= \Psi_{\rm cyc}^{\ast}(\underline{\cH}):\ \Delta C \to {\rm Mod}_k
$$
On the other hand, by Definition~\ref{CBHSG}, $\,\Psi_{\rm cyc}^*(k[\Gamma]) = k[B^{\rm cyc}(\Gamma)]\,$ for any homotopy simplicial group $ \Gamma $. Thus, by Lemma~\ref{LTor}, the functor $ \Psi_{\rm cyc} $ induces a canonical map
\begin{equation}
\la{Psi*}
\Psi^*_{\rm cyc}:\ \Tor^{\Delta C}_{\ast}(k[B^{\rm cyc}\,\Gamma],\,B_{\rm cyc} \cH)\ \to\ 
\Tor^{\ffgr}_{*}(k[\Gamma],\,\underline{\cH})
\end{equation}
The target of this map is precisely 
 $\HR_*(\Gamma, \cH)$ (see Definition~\ref{RHDef}), while the domain differs
from $ \HC_*(k[\Gamma]) $ in the second argument
of `Tor' ({\it cf.} Definition~\ref{CBHSG}). 
To connect the two Tor's we will use the following lemma which we state in the language of affine algebraic groups.
\begin{lemma}
\la{Lcocyc}
Let $G$ be an affine algebraic group defined over $k$, and let $ \cO(G) $ be its coordinate algebra. There is a natural isomorphism
\begin{equation}
\la{corres}
\Hom_{\Mod_k(\Delta C)}(k,\ B_{\rm cyc}[\cO(G)])\,\cong\, \cO(G)^{G}
\end{equation}
where $ \cO(G)^G $ denotes the invariant subalgebra of $  \cO(G) $
under the adjoint $G$-action.
\end{lemma}
\begin{proof}
For $ m \ge 0 $, denote by $ \pi_m: [0] \to [m]$  the composition of maps $\,d^0_m \, d^0_{m-1} \, \ldots \, d^0_1\,$ in $ \Delta C$. It follows from \eqref{s3e20} that $\Psi_{\rm cyc}(\pi_m): \langle 1 \rangle \to \langle m+1 \rangle$ is the homomorphism of groups taking the generator $x$ of $ \bF\langle x\rangle$  to the product of generators $x_0 x_1 \ldots x_m$ in $\bF\langle x_0, \ldots , x_m\rangle$. The corresponding map $ [B_{\rm cyc} \cO(G)](\pi_m): \cO(G) \to \cO(G)^{\otimes (m+1)} $ can thus be identified with the the $m$-fold coproduct in $ \cO(G)$:
\begin{equation}
\la{Dnf}
\Delta_G^{(m)}:\ 
\cO(G) \to \cO(G^{m+1})\,,\quad P \mapsto \big[(g_0, g_1, \ldots,g_m) \mapsto P(g_0\,g_1 \, \ldots \, g_m) \big] \ .
\end{equation}
Now, it is easy to check that, for a fixed $ P \in \cO(G)^G$, the maps $ \Delta_G^{(m)}(P):\, k \to \cO(G^{m+1})$ taking $1 \in k $ to $ \Delta_G^{(m)}(P) $ assemble to
a morphism of cocylic modules $\,\Delta_G(P):\ k \,\to\, B_{\rm cyc}[\cO(G)]\,$, the commutativity with cyclic operators $ \tau_m $ being a consequence of the $G$-invariance of $P$.
We claim that the assignment $\, P \mapsto \Delta_G(P)\,$ defines a $k$-linear isomorphism
\begin{equation}
\la{Dnf1}
\Delta_G:\ \cO(G)^G \,\xrightarrow{\sim}\, \Hom_{\Mod_k(\Delta C)}(k,\ B_{\rm cyc}[\cO(G)])\ .
\end{equation}

The inverse of \eqref{Dnf1} can be constructed as follows. Let $\varphi \in \Hom_{\Mod_k(\Delta C)}(k,\ B_{\rm cyc}[\cO(G)])$. Note that, for all $[m] \in \Delta C $, its components $\varphi_{[m]}:\,k \to \cO(G)^{\otimes (m+1)} \cong \cO(G^{m+1})$ are $k$-linear maps. Define $ T\varphi := \varphi_{[0]}(1) \in \cO(G)$, where $ 1 \in k$.  Since $\varphi$ is a natural transformation, 
$$\varphi_{[m]}(1) = \{[B_{\rm cyc} \cO(G)](\pi_m)\}(\varphi_{[0]}(1)) = \{[B_{\rm cyc} \cO(G)](\pi_m)\}(T\varphi) = \Delta^m(T \varphi)\ ,
$$
where $ \Delta^m $ is defined in \eqref{Dnf}.
Similarly, applying  $B_{\rm cyc}[\cO(G)]$ to the cyclic operators $ \tau_m $ in $ \Delta C$, we have
$$\varphi_{[m]}(1) = \{[B_{\rm cyc} \cO(G)](\tau_m)\}(\varphi_{[m]}(1))\,, $$
from which it follows that $T\varphi (g_0 \ldots g_m)= T\varphi(g_m g_0\ldots g_{m-1})$ for all $g_0,\ldots,g_m \in G$. This is equivalent to the assertion that $T \varphi \in \cO(G)^G$. 
Thus $T$ defines a $k$-linear map 
$$
\Hom_{\Mod_k(\Delta C)}(k,\ B_{\rm cyc}[\cO(G)]) \to \cO(G)^G\,,\qquad \varphi \mapsto T\varphi\ .
$$
It is clear from its construction that the above map is the inverse of \eqref{Dnf1}.
\end{proof}
We can now make the following definition.
\begin{definition}
\la{Gchar}
Let $ \Gamma \in \sGr^h $ be a homotopy simplicial group. For an affine algebraic group $G$ and an $\Ad\,G$-invariant polynomial $ P \in \cO(G)^G$, we define the {\it derived $G$-character map of $\,\Gamma\,$
associated to $P$} by
\begin{equation}
\la{chiGP}
\chi_{G,P}(\Gamma)_{\ast}:\ \HC_*(k[\Gamma]) \,\xrightarrow{(\Delta_G P)_*}\,
\Tor^{\Delta C}_{\ast}(k[B^{\rm cyc}\,\Gamma],\,B_{\rm cyc}[\cO(G)])\,\xrightarrow{\Psi_{\rm cyc}^*}\,
\HR_\ast(\Gamma, G)\ ,
\end{equation}
where $ (\Delta_G P)_* $ is a linear map induced  by the map of cocyclic modules $ \Delta_G P:\, k \to B_{\rm cyc}[\cO(G)]$ (see \eqref{Dnf} and \eqref{Dnf1}), and $ \Psi_{\rm cyc}^* $ is the map \eqref{Psi*} defined for $ \cH = \cO(G) $. 
\end{definition}
Explicitly, if we choose a projective resolution $\, Q \xrightarrow{\sim} k[\Gamma] \,$ of $k[\Gamma]$ in the (abelian) category $ \Mod_k(\ffgr^{\rm op})$, applying  the functor $ \Psi_{\rm cyc}^* $ gives a projective resolution $\, \Psi_{\rm cyc}^*Q \xrightarrow{\sim} k[B^{\rm cyc}\, \Gamma] \,$ of the cyclic module $ k[B^{\rm cyc}\,\Gamma]$ in $ \Mod_k(\Delta C^{\rm op})$. The map \eqref{chiGP} is then induced by a map of chain complexes
\begin{equation}
\la{chiGP1}
\chi_{G,P}(\Gamma)_{\ast}: \ (\Psi_{\rm cyc}^*Q) \otimes_{\Delta C} k\,\to\,
Q \otimes_{\ffgr} \cO(G)
\end{equation}
which, in turn, is induced by the following map (see \eqref{s3e11})
\begin{equation}
\la{chiGP2}   
\Moplus_{[m] \in \Delta C}\,Q\langle m+1\rangle\ \to\ 
\Moplus_{\langle n \rangle \in \ffgr}\,Q\langle n\rangle \otimes \cO(G)^{\otimes n}\ ,\quad
v_{m+1} \,\mapsto\, v_{m+1} \otimes \Delta_G^{(m)}(P) \ , 
\end{equation}
where $ v_{m+1} \in Q\langle m+1\rangle $ and $\Delta_G^{(m)}(P) \in \cO(G)^{\otimes (m+1)} $ is defined by \eqref{Dnf}.

In the special case when $ G = \GL_n(k) $ and $ P = \Tr_n \in \cO(\GL_n) $ is the usual trace function on $ (n \times n)$-matrices, we denote the map \eqref{chiGP} by 
\begin{equation}
\la{Trrr}
\Tr_n(\Gamma)_\ast:\ \HC_*(k[\Gamma])\,\to\,\HR_\ast(\Gamma, \GL_n(k))\,,
\end{equation}
and call it the {\it derived character map of $n$-dimensional representations of} $\,\Gamma$. 
%
%
%
%
In the rest of the paper, we will study the maps $\Tr_n(\Gamma)_\ast$ in two extreme cases: $\, n = 1 \,$ and $\, n = \infty \,$. In the first case, we will give a topological realization of $ \Tr(\Gamma)_\ast := \Tr_1(\Gamma)_* $ by showing that this map is induced on homology by a natural map of topological spaces; in the second case, we will show that $\Tr_{\infty}(\Gamma)_\ast := \varprojlim\,\Tr_n(\Gamma)_\ast\, $ extends to an isomorphism between the graded symmetric algebra generated by $ \rHC_*(k[\Gamma]) $ and
the $\GL_\infty$-invariant subalgebra of the stable representation homology $ \HR_*(\Gamma, \GL_{\infty}(k))$. We close this section with a general remark linking the above 
construction to earlier work.
\begin{remark}
\la{RemAlg}
If $ \Gamma $ is an ordinary discrete or (strict) simplicial group, then $ k[\Gamma] $
is naturally a simplicial associative $k$-algebra. By (a monoidal version of) the classical 
Dold-Kan correspondence (see \cite{SS03}), we can therefore view $ k[\Gamma] $ as a
differential-graded (DG) associative $k$-algebra. For such algebras (defined over 
a field $k$ of characteristic $0$), the derived character maps of $n$-dimensional 
representations were constructed in \cite{BKR}. One can show that these maps agree 
with \eqref{Trrr} in the case of group algebras, although the comparison is not entirely 
trivial as the methods used in \cite{BKR} and the present paper are quite different. 
We will address this question in our forthcoming paper \cite{BR22b} in a greater generality.
\end{remark}

\section{Topological Realization of Derived Character Maps}
\la{S4}
In this and next sections, we will prove our main results (Theorem~\ref{T1} and Theorem~\ref{CRMap}) stated in the Introduction. Here we will construct
the required spaces and maps simplicially: in terms of homotopy colimits of small diagrams of simplicial sets
and associated natural maps. Then, in the next section, we will reproduce these maps in topological terms, using Goodwillie homotopy calculus and topological operads. The connection between the two approaches seems
instructive and deserves a further investigation.

\subsection{The space $ X_{\Gamma}$}
\la{S4.1}
Recall that $\ffgr$ denotes the skeleton of the category of finitely generated free groups. There is a natural {\it abelianization functor} 
\begin{equation} \la{s4e1} \underline{\Z}\,:\,\ffgr \to \Set\,,\qquad \langle n \rangle \mapsto \Z^n\ ,
\end{equation}
that takes the free group $\langle n \rangle = \bF_n $ to (the underlying set of) its abelianization $\langle n \rangle_{\rm ab} = \Z^n$. As in Section \ref{S2},  we can form the category of elements of \eqref{s4e1}, using the Grothendieck construction:
\begin{equation} \la{s4e2} \ffgr_{\Z}\,:=\, \ffgr\! \smallint \underline{\Z} \end{equation}
The objects of $ \ffgr_{\Z} $ are given explicitly by 
$${\rm Ob}(\ffgr_{\Z}) \,=\,\{ (\langle n \rangle ;k_1,\ldots, k_n)): \langle n \rangle \in \ffgr \,,\, (k_1,\ldots,k_n) \in \Z^n \} $$
and the morphism sets are
$${\rm Hom}_{\ffgr_\Z}((\langle n \rangle; k_1,\ldots,k_n), (\langle m \rangle; l_1,\ldots,l_m))\,=\,\{\varphi \in {\rm Hom}_{\ffgr}(\langle n \rangle, \langle m \rangle): \ \varphi_{\rm ab}(k_1,\ldots,k_n)= (l_1,\ldots,l_m) \} $$
Note that the abelianized map $\,\varphi_{\rm ab}:\Z^n \to \Z^m\,$ above is represented by an integral $(m \times n)$-matrix, $\varphi_{\rm ab} \in {\mathbb M}_{m \times n}(\Z)$, and its action on $n$-tuples of integers is simply given by matrix multiplication. The category \eqref{s4e2} comes together with the canonical (forgetful) functor
\begin{equation} \la{s4e3} p: \ffgr_{\Z} \to \ffgr\,,\qquad  (\langle n \rangle; k_1,\ldots,k_n) \mapsto \langle n \rangle\ .\end{equation}
Given a homotopy simplicial group $\Gamma: \ffgr^{\rm op} \to \sSet$, we now define 
\begin{equation} \la{s4e4} X_{\Gamma}\,:=\, |\,\hocolim_{\ffgr_{\Z}^{\rm op}}(p^{\ast}\Gamma)\,|\,, \end{equation}
where $p^{\ast}$ is the pullback functor $\sSet^{\ffgr^{\rm op}} \to \sSet^{\ffgr_\Z^{\rm op}}$ associated to \eqref{s4e3}. The relation of the space \eqref{s4e4} to representation homology becomes clear from the following observation.
\begin{lemma} \la{HXGa}
For any $\Gamma$ and any commutative ring $k$, there is a natural isomorphism 
\begin{equation} \la{s4e5} 
{\rm H}_\ast(X_{\Gamma}, k) \,\cong\, \HR_\ast(\Gamma, \mathbb{G}_m(k)) 
\end{equation}
\end{lemma}
\begin{proof}
We have the sequence of natural isomorphisms
\begin{equation*}
H_\ast(X_{\Gamma}, k) \, \cong \, {\rm Tor}^{\ffgr_\Z^{\rm op}}_\ast(k,\, k[p^{\ast}\Gamma])
                             \,\cong \, {\rm Tor}^{\ffgr_\Z}_\ast(k[p^{\ast}\Gamma],\,k)
                        \,    =  \, {\rm Tor}^{\ffgr_\Z}_\ast(p^{\ast}k[\Gamma],\,k)
                            \, \cong\,  {\rm Tor}^{\ffgr}_\ast(k[\Gamma],\,k[\Z])\ ,
\end{equation*}
where the first two are standard (see, e.g., \cite[Appendix C.10]{L}) and the last one follows from the classical Shapiro Lemma (see Lemma \ref{Shap}). To complete the proof it remains to note that $ k[\Z] $ can be identified with $\cO[{\mathbb G}_m(k)]  $ as a commutative Hopf algebra. 
\end{proof}
As in the Introduction, we shorten notation for one-dimensional representation homology, writing 
\begin{equation}\la{HRnot}
\HR_\ast(k[\Gamma])\,:=\,\HR_\ast(\Gamma, \mathbb{G}_m(k))\ .
\end{equation}
Our next goal is to  identify the homotopy type of the space $X_{\Gamma}$ in terms of the classifying space of $\Gamma$. The following theorem is one of the main results of the present paper.
\begin{theorem} \la{omspinf}
For any homotopy simplicial group $\Gamma$, there is a weak equivalence in ${\rm Top}_\ast \!:$
\begin{equation} 
\la{s4e6} 
X_{\Gamma} \,\simeq\, \Omega\,{\rm SP}^{\infty}(B\Gamma)
\end{equation}
where $B\Gamma$ is the classifying space of $\, \Gamma $ $($see Definition \ref{CSpaceHSG}$)$.
\end{theorem}
Before proving this theorem, we recall a few basic facts about the Dold-Thom space and related constructions
(see, e.g., \cite[Chap. 4.K]{Hat}). For any pointed connected CW complex $X$, the {\it Dold-Thom space} ${\rm SP}^{\infty}(X)$ is defined as the infinite symmetric product: namely, 
\begin{equation} \la{s4e7} {\rm SP}^{\infty}(X)\,=\,\varinjlim_n\, {\rm SP}^n(X) \end{equation}
where ${\rm SP}^n(X) := X^n\!/S_n$ with $S_n$ acting on $ X^n $ the natural way (by permuting the factors). The maps ${\rm SP}^n(X) \to {\rm SP}^{n+1}(X)$ along which the inductive limit \eqref{s4e7} is taken are induced by the natural inclusion 
$$ 
X^n \hookrightarrow X^{n+1}\,,\qquad (x_1,\ldots,x_n) \mapsto (x_1,\ldots, x_n,\ast) \ , 
$$
where `$\ast$' stands for the basepoint of $X$. The Dold-Thom Theorem asserts that, for all
$ i \ge 1 $, there are isomorphisms of abelian groups 
$$
\pi_i[{\rm SP}^{\infty}(X)] \,\cong\, H_i(X,\Z)\,,
$$
that are natural in pointed connected CW complexes $X$. In fact, this classical theorem provides a topological realization for the Hurewicz homomorphisms in the sense that the natural map of spaces 
\begin{equation} \la{s4e8} X={\rm SP}^1(X) \hookrightarrow {\rm SP}^{\infty}(X)\end{equation}
induces the homomorphisms of groups: $\,\pi_i(X) \to \H_i(X,\Z)\,$ for all $i \geqslant 1$. 

Now, let $\cF X$ denote the homotopy fibre of the inclusion map \eqref{s4e8} so that we have a homotopy fibration sequence
\begin{equation} \la{s4e9} 
\cF X \to X \to {\rm SP}^{\infty}(X) \ .
\end{equation}
There is an alternative way to obtain this fibration sequence, using Kan's simplicial group model $\lgr(X)$ of the space\footnote{Abusing notation, we will identify a pointed connected space $X$ with a reduced simplicial set that represents $X$, i.e.
$ X =|X|$.} $X$. Namely (see, e.g., \cite[Section 7]{Bau}), \eqref{s4e9} arises from the short exact sequence of simplicial groups 
\begin{equation} \la{s4e10} 1 \to \lgr_2(X) \to \lgr{(X)} \to \mathbb{A}(X) \to 1\end{equation}
by applying the classifying space functor $\, B = | \bar{W}(\mbox{--}) |\,$. Here $\, \lgr_2(X) := [\lgr(X), \lgr(X)] \,$ denotes the commutator subgroup of the Kan loop group $\lgr(X)$ and $\mathbb{A}(X)$ its abelianization:
\begin{equation} \la{s4e11}  \mathbb{A}(X)  \,:=\,  (\lgr{X})_{\rm ab}\,:=\, \lgr(X)/\lgr_2(X) \end{equation}

Thus, we have $\,{\rm SP}^{\infty}(X) \simeq B\mathbb{A}(X) = |\bar{W} \mathbb{A}(X)| \,$,
which, by Kan's Theorem (see \eqref{s3e8}), implies 
\begin{equation} 
\la{s4e11} 
\Omega\,{\rm SP}^{\infty}(X) \,\simeq\, \Omega\,|\bar{W} \mathbb{A}(X)| \,\simeq \, |\lgr\bar{W} \mathbb{A}(X)| \,\simeq\, |\mathbb{A}(X)|\ .\end{equation}
Note that for any reduced simplicial set $X$, $\, {\mathbb A}(X) \cong \tilde{\Z}[X] \,$ is just
the reduced free simplicial abelian group generated by $X$.
After these preliminary remarks we can proceed with
\begin{proof}[Proof of Theorem \ref{omspinf}] As a first step we apply Proposition~\ref{ShLemma} to express
the homotopy colimit \eqref{s4e4} as a homotopy coend:
\begin{equation}\la{eee1}
\hocolim_{\ffgr_{\Z}^{\rm op}}(p^{\ast}\Gamma)\,\simeq\, \int_{\L}^{\langle n \rangle \in \ffgr}\!
\Gamma \langle n \rangle \times \underline{\Z}^n
\end{equation}
Next, observe that the bifunctor 
\begin{equation}
\la{eee2}
\Gamma \times \underline{\Z}:\ \ffgr^{\rm op} \times \ffgr \,\to\, \sSet\ ,\quad
(\langle n \rangle, \langle m \rangle)\,\mapsto\, \Gamma \langle n \rangle \times \Z^m\ ,
\end{equation}
that appears in the homotopy coend \eqref{eee1} can be factored as
$$
\ffgr^{\rm op} \times \ffgr \,\xrightarrow{\,\Gamma\, \otimes\, \bF\,}\, \sGr \,\xrightarrow{\,
(\,\mbox{--}\,)_{\rm ab}\,}\, {\rm sAb} \,\xrightarrow{{\rm forget}}\, \sSet
$$
where the first arrow is precisely the bifunctor $\,\Gamma\, \otimes\, \bF\,$ that appears in formula
\eqref{Kf} of Lemma~\ref{LKformula}, expressing the rigidification functor $K$. This last 
bifunctor takes an object $ (\langle n \rangle, \langle m \rangle) \in \ffgr^{\rm op} \times \ffgr $ to
the simplicial group $\,\amalg_{\Gamma \langle n \rangle}\,\bF_m\,$, which is given, in each simplicial degree, 
by a free product of copies of the free group $ \bF_m $ indexed by the components of the simplicial
set $\Gamma \langle n \rangle $. Hence $\,\Gamma\, \otimes\, \bF\,$ is an objectwise cofibrant diagram
in $ \sGr $, and therefore 
\begin{equation}
\la{eee3}
\L(\Gamma\, \otimes\, \bF)_{\rm ab}\,\simeq\, (\Gamma\, \otimes\, \bF)_{\rm ab} \,\cong\, \Gamma \times \underline{\Z}\ ,
\end{equation}
where $ \L(\,\mbox{--}\,)_{\rm ab} $ stands for the (left) derived functor of the abelianization functor $\,(\,\mbox{--}\,)_{\rm ab}:\, \sGr \to {\rm sAb}$. Since the abelianization functor is left Quillen, its derived functor commutes with homotopy coends (see \eqref{lfhcoend}). Hence, combining \eqref{s3e6} with \eqref{eee3}, we get
\begin{equation}
\la{eee4}
\L[\L K(\Gamma)]_{\rm ab}\,\simeq\, \int^{\langle n\rangle \in \ffgr}_{\L} \L(\Gamma\langle n \rangle \otimes \mathbb{F}\langle n \rangle)_{\rm ab} 
\,\simeq\, \int^{\langle n\rangle \in \ffgr}_{\L}(\Gamma\langle n\rangle\, \otimes\, \bF\langle n\rangle)_{\rm ab}\,\simeq\, 
\int^{\langle n\rangle \in \ffgr}_{\L}\Gamma\langle n\rangle \times \underline{\Z}^n
\end{equation}
On the other hand, $\,\L[\L K(\Gamma)]_{\rm ab} \simeq [\lgr{\bar{W}}\L K(\Gamma)]_{\rm ab} = \mathbb{A}(\bar{W}\L K(\Gamma))$, hence, by \eqref{s4e11}, we have
\begin{equation}
\la{eee5}
|\,\L[\L K(\Gamma)]_{\rm ab}\,|\,\simeq\,|\,\mathbb{A}(\bar{W}\L K(\Gamma))\,|\,\simeq\,
\Omega\,{\rm SP}^{\infty}(B \Gamma) 
\end{equation}
Combining now \eqref{eee1}, \eqref{eee4} and \eqref{eee5}, we get the desired equivalence $ X_{\Gamma} \simeq \Omega\,{\rm SP}^{\infty}(B \Gamma)\,$.
\end{proof}
Note that Theorem \ref{omspinf} combined with Lemma \ref{HXGa} implies Theorem \ref{T1} stated in the Introduction.

\subsection{Symmetric homology}
\la{S4.2}

In Section \ref{S3.3}, we defined cyclic homology of homotopy simplicial groups by associating to each $\Gamma \in \sGr^h$ a cyclic bar construction $B^{\rm cyc}\Gamma :\Delta C^{\rm op}\to \sSet$ (see Definition \ref{CBHSG}). In this section, we introduce an analogue of this construction for symmetric groups. Recall that the {\it symmetric  crossed simplicial category} $\Delta S$ is defined to be an extension of $\Delta$ that has the same objects as $\Delta$ (and $\Delta C$) with morphisms characterized by the two properties ({\it cf.} \cite[6.1.4]{L}):

\vspace{1ex}

\begin{enumerate}
\item[(Sym1)] For each $n \geqslant 0$, ${\rm Aut}_{\Delta S}([n]) \cong S_{n+1}^{\rm op}\,$, where $S_{n+1}$ is the $(n+1)$-th symmetric group.
\item[(Sym2)] Any morphism $f:[n] \to [m]$ in $\Delta S$ can be factored uniquely as the composite $f=g \circ \sigma $ with $g \in {\rm Hom}_{\Delta}([n],[m])$ and $\sigma \in {\rm Aut}_{\Delta S}([n]) \cong S_{n+1}^{\rm op}$.
\end{enumerate}

\vspace{1ex}

\noindent
There is an inclusion functor (a morphism in ${\rm Cat}$):
\begin{equation} 
\la{s4.2e1} 
\iota: \ \Delta C^{\rm op} \,\stackrel{\sim}{\to}\, \Delta C \,\hookrightarrow\, \Delta S\ , 
\end{equation}
where the first arrow is an isomorphism of categories (called Connes' duality) and the second one is induced by the natural inclusion of groups $C_{n+1} \hookrightarrow S_{n+1}$({\it cf.} \cite[6.1.11]{L}). Explicitly, the functor \eqref{s4.2e1} is given on objects by $ \iota([n]) = [n] $ and
on generators by the following formulas
\begin{equation} \la{s4.2e2}
\left. \begin{aligned}
\iota(d^n_i) & = \begin{cases}
                    s^i_{n-1} \ ,&  0 \leq i < n\\
                    s^0_{n-1} \circ (n,0,1,\ldots,n-1)\ , &  i=n
                 \end{cases}\\
\iota(s^n_j) & =  d^{j+1}_{n+1}\\
\iota(t_n) & = (n,0,1,\ldots,n-1)
\end{aligned} \right.
\end{equation}
where $d^n_i:[n] \to [n-1]$, $\,s^n_j:[n] \to [n+1]$ and $t_n:[n] \to [n]$ denote the generators of $\Delta C^{\rm op}$ dual (opposite) to the generators $d^i_n$, $s^j_n$ and $\tau_n$ of $\Delta C$, respectively.
\begin{lemma} \la{dsextn}
The functor $\,\Psi_{\rm cyc}^{\rm op}: \Delta C^{\rm op} \to \ffgr^{\rm op}$ defined by \eqref{s3e19}, \eqref{s3e20} extends through $\iota $, giving a commutative diagram of small categories
\begin{equation} \la{s4.2e3}
\begin{diagram}[small]
\Delta C^{\rm op} & \rTo^{\Psi_{\rm cyc}^{\rm op}\ } & \ffgr^{\rm op}\\
 \dInto^{\iota} & \ruTo_{\Psi_{\rm sym}} &\\
 \Delta S & &\\
\end{diagram}
\end{equation}
\end{lemma}
\begin{proof}
In order to construct the functor $\Psi_{\rm sym}$ it is convenient to use the following notation for morphisms in $\Delta S$ ({\it cf.} \cite[Section 1.1]{Au1}). Any morphism $\,f:\, 
[n] \xrightarrow{\sigma} [n] \xrightarrow{g} [m]\,$ in $ \Delta S $ can be written uniquely as a `tensor product' of $m+1$ noncommutative monomials $\,X_0,\,X_1,\,\ldots,X_m\, $ in $n+1$ formal variables $\{x_0,x_1,\ldots,x_n\}$:
\begin{equation} \la{s4.2e4} f= X_0 \otimes X_1 \otimes \ldots \otimes X_m \end{equation}
where each $X_i$ is the product $\,x_{i_1}x_{i_2} \ldots x_{i_r}\,$ of $r=|f^{-1}(i)|$ variables whose indices $ i_k $ occur in the fibre $f^{-1}(i)$ and that are ordered in the same way as numbers in $\{\sigma(0),\ldots,\sigma(n)\}$, i.e. $\sigma(i_1) < \sigma(i_2) < \ldots < \sigma(i_r)$. For example, if $ f: [4] \to [3]\,$ is given by the composition 
$ g \circ \sigma $ in $\Delta S $, where $g \in \Hom_{\Delta}([4], [3]) $ is defined by $g(0)=g(1)=0$, $g(2)=g(3)=1$ and $g(4)=3$ and 
$ \sigma \in \Aut_{\Delta S}([4]) = S_5^{\rm op} $ is the permutation
$$ \sigma = \left(\! \begin{array}{ccccc}
                   0  & 1 &2 & 3& 4 \\
                   1 &  0 & 4 & 2 & 3\\
                       \end{array} \! \right) $$
then $f$ is represented by $\, x_1x_0 \otimes x_3x_4 \otimes 1 \otimes x_2 \,$.
The composition of morphisms $f_1 \circ f_2$ is defined by a natural substitution rule: for example, if $ f_1: [3] \to [3] $ and $f_2: [4] \to [3] $ in $ \Delta S $ are represented by
\begin{equation*}
f_1 = 1 \otimes x_0 \otimes 1 \otimes x_3x_2x_1 \ ,\quad
f_2 = x_2x_1 \otimes x_4 \otimes 1 \otimes x_0x_3 \ ,
\end{equation*}
then $ f_1 \circ f_2:\,[4] \to [3] $ can be computed as 
\begin{equation*}
\left.
\begin{aligned}
f_1 \circ f_2 &= (1 \otimes X_0 \otimes 1 \otimes X_3X_2X_1)
\circ(\underbrace{x_2 x_1}_{X_0} \otimes \underbrace{x_4}_{X_1} \otimes \underbrace{1}_{X_2} \otimes \underbrace{x_0 x_3}_{X_3})\\
              & = 1 \otimes x_2x_1 \otimes 1 \otimes (x_0x_3) \cdot 1 \cdot (x_4)\, = \, 1 \otimes x_2x_1 \otimes 1 \otimes x_0x_3x_4
\end{aligned}
\right.
\end{equation*}
With this notation, we define the functor 
\begin{equation} \la{s4.2e5} \Psi_{\rm sym}:\ \Delta S \to \ffgr^{\rm op}\end{equation}
on objects by 
$$ \Psi_{\rm sym}([n]) = \langle n+1 \rangle $$
and on morphisms by the following formula: if $ f \in \Hom_{\Delta S}([n], [m])$ is represeented by
$$
f = (x_{i_1}\ldots x_{i_r}) \otimes \cdots \otimes (x_{k_1} \ldots x_{k_s})\ ,
$$
then 
\begin{equation} \la{s4.2e6} \Psi_{\rm sym}(f): \langle m+1 \rangle \to \langle n+1 \rangle \,,\qquad X_0 \mapsto x_{i_1} \ldots x_{i_r}\,,\ \ldots\ ,\,X_m \mapsto x_{k_1} \ldots x_{k_s}\ ,
\end{equation}
where $\langle m+1 \rangle = \mathbb{F}\langle X_0,\ldots,X_m \rangle$ and $\langle n+1 \rangle = \mathbb{F}\langle x_0,\ldots,x_n\rangle$. Note that the maps \eqref{s4.2e2}
can be rewritten in this tensor notation as
\begin{equation*} 
\left. \begin{aligned}
\iota(d^n_i) & = \begin{cases}
                   x_0 \otimes \ldots \otimes x_{i-1} \otimes x_ix_{i+1} \otimes x_{i+2} \otimes \ldots \otimes x_n   \ ,&  0 \leq i < n\\
                   x_nx_0 \otimes x_1 \otimes \ldots \otimes x_{n-1}\ , &  i=n
                 \end{cases}\\
\iota(s^n_j) & =  x_0 \otimes \ldots \otimes x_j \otimes 1 \otimes x_{j+1} \otimes \ldots \otimes x_n\\
\iota(\tau_n) & = x_n \otimes x_0 \otimes x_1 \otimes \ldots \otimes x_{n-1}
\end{aligned} \right.
\end{equation*}
The commutativity of \eqref{s4.2e3} can now be checked by a trivial calculation that we leave to the reader.
\end{proof}

Having in hand the functor $\Psi_{\rm sym}:\ \Delta S \to \ffgr^{\rm op}\,$, we can now define a symmetric
bar construction in the same way as we defined the cyclic bar construction in Definition \ref{CBHSG}.
\begin{definition}
\la{defHS}
For a homotopy simplicial group $\Gamma \in \sGr^h$, its {\it symmetric bar construction} is the functor 
\begin{equation} \la{s4.2e7} B_{\rm sym}\Gamma\,:=\,\Psi^{\ast}_{\rm sym}\Gamma : \Delta S \to \sSet \end{equation} 
ans its {\it symmetric homology} is defined by
\begin{equation} \la{s4.2e8} {\rm HS}_{\ast}(k[\Gamma])\,:=\,{\rm Tor}^{\Delta S}_\ast(k,k[B_{\rm sym}\Gamma])\ . \end{equation} 
\end{definition}

\vspace{1ex}

\begin{remark}
The same argument as (in the proof of) Lemma \ref{RHBadM} shows that ${\rm HS}_\ast(k[\Gamma])$ depends only on the homotopy type of $\Gamma$ in $\sGr^h$ and hence on the homotopy type of the space $B\Gamma$.
\end{remark}


\begin{remark}
For $\Gamma$ an ordinary discrete group, the definition \eqref{s4.2e7} agrees with Fiedorowicz's original definition of the symmetric bar construction (see \cite{F} and also \cite{Au1}). In this case, formula \eqref{s4.2e8} defines the symmetric homology of the group algebra $k[\Gamma]$. Note that, unlike $B^{\rm cyc}\Gamma$ (see \eqref{s3e21}), the functor $B_{\rm sym}\Gamma:\Delta S \to \sSet$ is covariant on $\Delta S$ (which we emphasize by writing ``sym" as a subscript). 
\end{remark}


\begin{remark} \la{catmon}
To study symmetric homology it is often convenient to work with the {\it augmented} symmetric category $\Delta S_+$ which is defined by adding to $\Delta S$ the initial object $[-1]$ and morphisms $[-1] \to [n]$, one for each $n \geqslant -1$ (see \cite{Au1}). It is easy to see that the map $\Psi_{\rm sym}$ defined in Lemma \ref{dsextn} extends to $\Delta S_+ $:
\begin{equation} \la{s4.2e9} \Psi_{{\rm sym},+}:\ \Delta S_+ \to \ffgr^{\rm op}\end{equation}
by letting $\Psi_{{\rm sym},+}([-1]) := \langle 0 \rangle$. Now, the category $ \Delta S_+ $ is isomorphic 
to the category of so-called {\it finite associative sets}, $\cF({\rm as})$, introduced in \cite{PR02} (see also \cite[Section 15.4]{R20} for a detailed discussion) .
The latter is known to be a permutative category (PROP) that describes the associative unital algebras (see \cite{P2} and also \cite{R20}). It opposite category $\cF({\rm as})^{\rm op} $ describes the coassociative counital coalgebras. If we identify $\,\Delta S_{+} = \cF({\rm as})\,$, 
the restriction functor $ \Psi_{{\rm sym},+}^{\ast}: \, {\rm Mod}_k(\ffgr) \to {\rm Mod}_k[\cF(\rm as)^{\rm op}] $ associated to the opposite of \eqref{s4.2e9} takes commutative Hopf algebras viewed as functors  \eqref{halg} on
$ \ffgr $ to the underlying coassociative coalgebras viewed as functors on $ \cF(\rm as)^{\rm op} $. In other words, the morphism $\Psi_{{\rm sym},+}^{\rm op} $ is isomorphic to a morphism  of PROPs: $\, \cF(\rm as)^{\rm op} \to \ffgr \,$ that ``forgets" the algebra structure on commutative Hopf algebras. 
\end{remark}

\subsection{Symmetric homology vs representation homology}
\la{S4.3}
Recall that in Section \ref{S3.4}, we constructed the derived character map $\Tr(\Gamma)_\ast $ relating the cylic homology  of $ \Gamma $ to its (one-dimensional) representation homology:
\begin{equation}\la{s4.3e3} 
\Tr(\Gamma)_\ast:\ \HC_\ast(k[\Gamma]) \to \HR_\ast(k[\Gamma]) \end{equation}
On the other hand, as a consequence of Lemma \ref{dsextn}, we have a restriction map
\begin{equation}
\la{s4.3e2} 
\iota^\ast:\ \HC_\ast(k[\Gamma]) = {\rm Tor}^{\Delta C^{\rm op}}_\ast(k,\,k[B^{\rm cyc}\Gamma])\, \to\, \Tor^{\Delta S}_{\ast}(k,\,k[B_{\rm sym}\Gamma]) = {\rm HS}_{\ast}(k[\Gamma])\ . 
\end{equation}
induced by the isomorphism of cyclic spaces
\begin{equation}
\la{s4.3e1} 
B^{\rm cyc}\Gamma \,\cong\, \iota^{\ast} B_{\rm sym}\Gamma 
\end{equation}
The next proposition shows that the derived character map \eqref{s4.3e3} factors through \eqref{s4.3e2}, thus relating representation homology to symmetric homology.

\begin{proposition} 
\la{ftsym}
For any homotopy simplicial group $\Gamma \in \sGr^h \!$, there is a natural map
\begin{equation}\la{s4.3e4} \tilde{\Psi}_{\rm sym}^{\ast}: \ {\rm HS}_\ast(k[\Gamma])\, \to\, \HR_{\ast}(k[\Gamma]) \end{equation}
 such that 
 \begin{equation}
 \la{s4.3e4p} 
 \begin{diagram}[small]
 \HC_*(k[\Gamma])& & \rTo^{\Tr(\Gamma)_\ast} & & \HR_*(k[\Gamma])\\
& \rdTo_{\iota^*} &                         & \ruTo_{\tilde{\Psi}_{\rm sym}^{\ast}} &\\
&                  & \HS_*(k[\Gamma])        & &
 \end{diagram}
 \end{equation}
\end{proposition}
\begin{proof}
As our notation suggests, the map \eqref{s4.3e4} is actually induced by a morphism $\tilde{\Psi}_{\rm sym}$ in $\Cat$. 
%
%
We construct $\tilde{\Psi}_{\rm sym}$ by lifting the functor $\Psi_{\rm sym}$ of Lemma \ref{dsextn} to the
(opposite) category of elements of the abelianization functor \eqref{s4e1}:
\begin{equation}\la{s4.5e5}
\begin{diagram}[small]
 & & & & \ffgr_{\Z}^{\rm op}\\
 & & & \ruTo^{\tilde{\Psi}_{\rm sym}} & \dTo_{p^{\rm op}}\\
\Delta C^{\rm op} & \rTo^{\iota}& \Delta S &  \rTo^{\,\Psi_{\rm sym}\ } & \ffgr^{\rm op}
\end{diagram}
\end{equation}
The existence of such a lifting is a consequence of the following observation. Consider the composition of functors 
\begin{equation} \la{s4.3e6} \begin{diagram}[small] \Delta S^{\rm op} & \rTo^{\Psi_{\rm sym}^{\rm op}} & \ffgr & \rTo^{(\mbox{--})_{\rm ab}} & {\rm Ab} \end{diagram} \end{equation}
that takes an object $[n] \in \Delta S$ to the abelian group $\Z^{n+1}$. If we represent a morphism $f:[n] \to [m]$ in $\Delta S$ using the tensor notation \eqref{s4.2e4},
then $ \Psi_{\rm sym}^{\rm op}(f)_{\rm ab}:\Z^{m+1} \to \Z^{n+1}$, the value of \eqref{s4.3e6} on $f$, is represented by an integral $(n+1) \times (m+1)$-matrix whose rows are indexed by $ 0 \leqslant i \leqslant n$ and columns by $0 \leqslant j \leqslant m$, and the $j$-th column consists entirely of $0$'s and $1$'s, with the $1$'s occurring in positions indicated by the elements of $f^{-1}(j)$. For example, if $f:[4] \to [3]$ in $\Delta S$ is represented by the product $x_1x_0 \otimes x_3 x_4 \otimes 1 \otimes x_2$, then 
$ \Psi_{\rm sym}^{\rm op}(f)_{\rm ab}:\,\Z^{4} \to \Z^{5}$ is given by
$$ \Psi_{\rm sym}^{\rm op}(f)_{\rm ab} = \left( 
\begin{array}{cccc}
1 & 0 & 0 & 0\\
1 & 0 & 0 & 0\\
0 & 0 & 0 & 1\\
0 & 1 & 0 & 0\\
0 & 1 & 0 & 0\\
\end{array}
\! \! \right)\ .$$
Observe that for any morphism $f$ in $\Delta S$ the matrix $\Psi_{\rm sym}^{\rm op}(f)_{\rm ab} $ thus defined has exactly one nonzero entry in each row and that entry is $1$. Hence $\Psi_{\rm sym}^{\rm op}(f)_{\rm ab}$ maps the vector $(1,1,\ldots,1)^t \in \Z^{m+1}$ to the vector $(1,1,\ldots,1)^t \in \Z^{n+1}$. This shows that there is a well-defined functor 
\begin{equation} \la{s4.3e7} 
\tilde{\Psi}_{\rm sym}:\ \Delta S \to \ffgr_{\Z}^{\rm op}\ ,
\quad [n] \mapsto (\langle n+1 \rangle; 1,1,\ldots, 1)\ ,
\end{equation}
that makes the diagram \eqref{s4.5e5} commutative.
It follows from \eqref{s4.5e5} that 
$$
k[B_{\rm sym}\Gamma] = \Psi_{\rm sym}^{\ast}(k[\Gamma]) = {\tilde{\Psi}}_{\rm sym}^{\ast}(k[p^{\ast}\Gamma])
$$ 
Hence, by Lemma \ref{LTor}, the functor \eqref{s4.3e7} induces a natural map
\begin{equation} 
\la{s4.3e9} 
{\rm HS}_{\ast}(k[\Gamma]) = \Tor^{\Delta S}_{\ast}(k, k[B_{\rm sym} \Gamma]) 
\xrightarrow{{\tilde{\Psi}}_{\rm sym}^{\ast}}  \Tor^{\ffgr^{\rm op}_{\Z}}_{\ast}(k,\,k[p^*\Gamma])\, .
\end{equation}
We claim that if the target of the map \eqref{s4.3e9} is identified with the representation homology 
of $k[\Gamma]$ via the Shapiro Isomorphism (see Lemma~\ref{Shap}), then the required factorization property \eqref{s4.3e4p} holds. To verify this we fix a projective resolution 
$ Q \xrightarrow{\sim} k[\Gamma] $ of $ k[\Gamma] $ in $ \Mod_k(\ffgr^{\rm op}) $.
Then $ p^*(Q) \xrightarrow{\sim} p^*k[\Gamma] = k[p^*\Gamma] $ gives a projective resolution of $ k[p^*\Gamma] $ in $ \Mod_k(\ffgr_{\Z}^{\rm op}) $, and the Shapiro Isomorphism 
$$
\Tor^{\ffgr_{\Z}}_{\ast}(k[p^*\Gamma], \,k) \xrightarrow{\sim} 
\Tor^{\ffgr}_{\ast}(k[\Gamma], \,p_!(k))
$$
is induced by the composition
$$
p^*(Q) \otimes_{\ffgr_{\Z}} k 
\xrightarrow{\id \otimes \varepsilon_k} 
p^*(Q) \otimes_{\ffgr_{\Z}} p^* p_!(k)  
\xrightarrow{p^*}  
Q \otimes_{\ffgr} p_!(k)
$$
where the first map is given by the adjunction unit  $\,\varepsilon: \id \to p^* p_! \,$ and the second is
the restriction map via $ p $. Explicitly, using the definition \eqref{s3e11} of functor tensor products,  we can 
represent the above composite map as
\begin{equation}
\la{Shch}
\Moplus_{(\langle n\rangle;\, k_1, \ldots, k_n) \in \ffgr_{\Z}} 
Q\langle n\rangle   \ \to \ \Moplus_{\langle n\rangle \in \ffgr}  
Q\langle n\rangle \otimes k[\Z^{n}]\ ,\quad 
(v_n)_{(\langle n\rangle;\, k_1, \ldots, k_n) \in \ffgr_{\Z}} \mapsto 
(v_n \otimes (k_1, \ldots, k_n))_{\langle n\rangle \in \ffgr}
\end{equation}
where $ v_n \in Q \langle n \rangle $ and the subscripts denote the indices of the corresponding components 
of direct sums. Now, using the same resolution $ Q $, 
we can write explicitly  the composition of maps \eqref{s4.3e2} and \eqref{s4.3e9}: 
$$
\Psi^{\ast}_{\rm cyc}(Q) \otimes_{\Delta C} k 
\,\xrightarrow{\iota^*}\,
\Psi^{\ast}_{\rm sym}(Q) \otimes_{\Delta S^{\rm op}} k
\,\xrightarrow{{\tilde{\Psi}}_{\rm sym}^{\ast}} \, 
p^*(Q) \otimes_{\ffgr_{\Z}} k
$$
At the level of chain complexes, this last composition is induced by the map
\begin{equation}
\la{mapcom}
\Moplus_{[m] \in \Delta C} Q\langle m+1\rangle \ \to\
\Moplus_{[m] \in \Delta S^{\rm op}} Q\langle m+1\rangle\ \to\
\Moplus_{(\langle n \rangle; k_1, \ldots k_n) \in \ffgr_{\Z}} Q\langle n \rangle 
\end{equation}
that takes
$\,
(v_{m+1})_{[m] \in \Delta C}\, \mapsto\, (v_{m+1})_{[m] \in \Delta S^{\rm op}} 
\, \mapsto\, (v_{m+1})_{(\langle m+1 \rangle;\, 1, 1, \ldots, 1) \in \ffgr_{\Z}}
\,$.
Combining \eqref{Shch} and \eqref{mapcom}, we see that the resulting map 
$$
\Moplus_{[m] \in \Delta C} Q\langle m+1\rangle \ \to \ \Moplus_{\langle n\rangle \in \ffgr}  
Q\langle n\rangle \otimes k[\Z^{n}]\ ,\quad 
(v_{m+1})_{[m] \in \Delta C}\ \mapsto\ 
(v_{m+1} \otimes (1, 1, \ldots, 1))_{\langle m+1 \rangle \in \ffgr}
$$
coincides exactly with the map \eqref{chiGP2} representing the derived character  
$ \chi_{\GL_1, \Tr_1}(\Gamma)_\ast = \Tr(\Gamma)_{\ast} $. This finishes the proof of
the proposition.
\end{proof}

\begin{remark}\la{Remlift}
The proof of Proposition~\ref{ftsym} shows that,  
apart from \eqref{s4.3e7}, any functor of the form
\begin{equation}
\la{liftm} 
\tilde{\Psi}^{(m)}_{\rm sym}:\ \Delta S \to \ffgr_{\Z}^{\rm op}\ ,
\quad [n] \mapsto (\langle n+1 \rangle;\, m, m,\ldots, m)\ ,
\end{equation}
where $ m \in \Z $ is a fixed integer, satisfies the lifting property \eqref{s4.5e5}. It is easy to see that there are no other solutions to this lifting problem. Among \eqref{liftm} the functor $\tilde{\Psi}^{(0)}_{\rm sym}$ corresponding to $m=0$ is the only one that factors through $\ffgr^{\rm op}$: $\,\tilde{\Psi}^{(0)}_{\rm sym} = s^{\rm op} \circ {\Psi}_{\rm sym} $, where $ s: \ffgr \into \ffgr_{\Z} $ is the `zero' section of $p$.
\end{remark}

Next, we observe that the linear maps factoring  $ \Tr(\Gamma)_* $ in \eqref{s4.3e4p} arise 
(on homology) from the natural maps of topological spaces  induced by the 
functors \eqref{s4.2e1} and \eqref{s4.3e7} ({\it cf.} Lemma \ref{HXGa}):
\begin{equation} \la{s4.3e10} 
|{\rm hocolim}_{\Delta C^{\rm op}}\, (B^{\rm cyc}\Gamma)| \,\xrightarrow{\iota^{\ast}} \, 
|{\rm hocolim}_{\Delta S} \,(B_{\rm sym}\Gamma)| \,\xrightarrow{\tilde{\Psi}_{\rm sym}^{\ast}}\, 
|{\rm hocolim}_{\ffgr_{\Z}^{\rm op}} \,(p^{\ast}\Gamma)| 
\end{equation}
(Here, abusing notation, we denote these topological maps by the same symbols as the corresponding 
linear maps.) By Theorem \ref{omspinf}, we know that
\begin{equation} \la{s4.3e11} |{\rm hocolim}_{\ffgr_{\Z}^{\rm op}} (p^{\ast}\Gamma)| \,\simeq\, 
\Omega\,{\rm SP}^{\infty}(B\Gamma)
\end{equation}
On the other hand, by theorems of Goodwillie (see \cite[Theorem 7.2.4]{L}) and Fiedorowicz \cite{F} (see \cite[Section 5.3]{Au1}), 
\begin{eqnarray}
\la{s4.3e12} |{\rm hocolim}_{\Delta C^{\rm op}} (B^{\rm cyc}\Gamma)| & \simeq & ES^1 \times_{S^1} \cL(B\Gamma)\,,\\*[1ex]
\la{s4.3e13}  |{\rm hocolim}_{\Delta S} (B_{\rm sym}\Gamma)| & \simeq & \Omega \Omega^{\infty} \Sigma^{\infty} (B\Gamma)\,,
\end{eqnarray}
where $\cL(B\Gamma) := {\rm Map}(S^1,B\Gamma)$  and $\,\Omega^{\infty}\Sigma^{\infty}(B\Gamma) := \hocolim_{n\to \infty}\,\Omega^n\Sigma^n(B\Gamma)\,$ denote the free loop space and the infinite loop space of $B\Gamma$, respectively. Combining \eqref{s4.3e10} with equivalences \eqref{s4.3e11}, \eqref{s4.3e12} and 
\eqref{s4.3e13}, we can thus refine the result of Proposition~\ref{ftsym} as follows:
\begin{corollary}
\la{topch}
The derived character map 
$$\Tr(\Gamma)_\ast: \HC_\ast(k[\Gamma]) \,\xrightarrow{\iota^*}\, {\rm HS}_{\ast}(k[\Gamma]) 
\,\xrightarrow{\tilde{\Psi}_{\rm sym}^{\ast}}\, \HR_\ast(k[\Gamma]) $$
is induced on homology by a natural map of topological spaces in $ \Ho(\Top_*)$:
\begin{equation}  
ES^1 \times_{S^1} \cL(B\Gamma) \,\xrightarrow{\CS_{B \Gamma}}\,
\Omega \Omega^{\infty} \Sigma^{\infty} (B\Gamma)\, \xrightarrow{\SR_{B \Gamma}}\, 
\Omega {\rm SP}^{\infty}(B\Gamma) 
\la{s4.3e14} 
\end{equation}
\end{corollary}
In the next section, we will describe the maps $\CS $ and $\SR $ in topological terms in two ways: using the classical `little cubes' operads and the Goodwillie calculus of homotopy functors. 

\subsection{Generalization to monoids}
\la{S4.4}
All results of this section generalize naturally to (simplicial) monoids. 
We briefly outline this generalization as we will need it in Section~\ref{S5.3}.
Instead of $\ffgr$, we start with the category $\fM \subset {\rm Mon} $ whose objects are 
finitely generated free monoids\footnote{Abusing notation, we will use the same symbols to denote the objects of $\fM$ and $\ffgr$.} $ \langle n \rangle $, one for each $ n \ge 0 $.
In this case, the abelianization functor reads
$$
\underline{\N}:\, \fM \to \Set\,,\qquad \langle n \rangle \mapsto \N^n
$$
where $\N$ is the set of natural numbers, i.e. the underlying set of the free abelian monoid of rank one. 
The associated category of elements $\fM_{\N}\,:=\, \fM\! \smallint \underline{\N}$ has an explicit description similar to that of $\ffgr_{\Z}$: its objects are $ (\langle n \rangle; \, k_1,\ldots,k_n))$, where $\langle n \rangle$ is the free monoid on $n$ generators and $(k_1,\ldots,k_n) \in \N^n $. Any simplicial monoid
$M$ gives a functor $M: \fM^{\rm op} \to \sset $ that restricts to $ \fM_{\N}^{\rm op}$ via the canonical projection $p:\,\fM_{\N} \to \fM$. The analogue (generalization) of Theorem~\ref{omspinf} says:
\begin{proposition}\la{propmon}
For any simplicial monoid $M$, there is a weak equivalence in $\Top_*$:
\begin{equation}\la{spacemon}
|{\rm hocolim}_{\fM_{\N}^{\rm op}} (p^{\ast}M)| \simeq \Omega\,{\rm SP}^{\infty}(BM)\,,
\end{equation}
where $ BM $ is the classifying space of $M$.
\end{proposition}
\begin{proof}
The same argument as in the proof of Theorem \ref{omspinf} --- based on Proposition~\ref{ShLemma} --- shows
$$
{\rm hocolim}_{\fM_{\N}^{\rm op}} (p^{\ast}M) \simeq \L(M)_{\rm ab}
$$
where $\,\L(\,-\,)_{\rm ab}\,$ denotes the derived abelianization functor on simplicial monoids.
To compute this last functor, instead of Kan loop group, we will use the 2-sided (simplicial) bar resolution
\eqref{simpbar}: $\, B_*(\uC_1, \uC_1, M) \xrightarrow{\sim} M\,$ in $\sset_*$, where $ \uC_1 $ is the monad associated to the (simplicial analogue of) little $1$-cube operad (see \eqref{Comon}). Since $ (\uC_1(X))_{\rm ab} = \uC_0(X) $, we have
$$
|\L(M)_{\rm ab}| \simeq |B_*(\uC_1, \uC_1, M)_{\rm ab}| \simeq |B_*(\uC_0, \uC_1, M)| \simeq \Omega\,{\rm SP}^{\infty}(BM)\,,
$$
where the last equivalence is a result of Lemma~\ref{maybar} below (see \eqref{omspbar}).
\end{proof}
The relation between monoids and groups is determined by the canonical 
(group completion) functor $l: \fM \to  \ffgr $. This last functor extends naturally to a functor 
$\tilde{l}: \fM_{\N} \to \ffgr_{\Z} $, and the maps $ \Psi_{\rm sym}: \Delta S \to \ffgr^{\rm op} $ and $ \tilde{\Psi}_{\rm sym}:\ \Delta S \to \ffgr_{\Z}^{\rm op} $ defined by \eqref{s4.2e5} and \eqref{s4.3e7} factor through $l$ and $ \tilde{l}$ respectively, giving the commutative diagram
\begin{equation}
\la{diagmon}
\begin{diagram}[small]
 & & & & \fM_{\N}^{\rm op} & \rTo^{\tilde{l}} & \ffgr^{\rm op}_{\Z} \\
 & & & \ruTo^{\tilde{\Psi}_{\rm sym}} & \dTo^{p} & & \dTo_{p}\\
\Delta C^{\rm op} & \rTo^{\iota}& \Delta S &  \rTo^{\,\Psi_{\rm sym}\ } & \fM^{\rm op} & \rTo^{l} & \ffgr^{\rm op}
\end{diagram}
\end{equation}

As a consequence of Proposition~\ref{propmon}, we get
\begin{corollary}
\la{cor33}
For any homotopy simplicial group $ \Gamma \in \sGr^h $, there is a weak equivalence
$$
|{\rm hocolim}_{\fM_{\N}^{\rm op}} (p^{\ast}l^{\ast} \Gamma)|\,\simeq\,
\Omega\,{\rm SP}^{\infty}(B\Gamma)
$$
\end{corollary}
\begin{proof}
Apply Proposition~\ref{propmon} to the simplicial group $ \L K(\Gamma) $ viewed as a simplicial monoid.
\end{proof}
\begin{remark} \la{sgrsmon}
Corollary~\ref{cor33} can be also deduced from Theorem~\ref{omspinf} if we notice that
the natural map
$$ 
{\rm hocolim}_{\fM_{\N}^{\rm op}} (p^{\ast}l^{\ast} \Gamma) \,\xrightarrow{\sim}\, {\rm hocolim}_{\ffgr_{\Z}^{\rm op}} (p^{\ast}\Gamma) 
$$ 
is a weak equivalence for any $ \Gamma $. This last fact follows from Theorem~\ref{Cofinality}, the
assumptions of which hold thanks to the known properties of the group completion functor ({\it cf.} \cite[Lemma 3.2]{BRYI}).
\end{remark}

\section{Topological character maps via Goodwillie calculus and  operads}
\la{S55}
In this section, we will describe the maps $\CS$ and $\SR $ explicitly in topological terms, using
Goodwillie calculus and classical operads. The latter approach is based on ideas of Fiedorowicz \cite{F} that were developed by Ault in \cite{Au1}. The former is inspired by results of Biedermann and Dwyer that appeared in \cite{BD}. The interpretation in terms of Goodwillie derivatives leads to a natural nonlinear (polynomial) generalization of topological character maps that deserves a further study (see Section~\ref{S5.4}). 

\subsection{Goodwillie homotopy calculus}
Goodwillie calculus provides a universal approximation (``Taylor decomposition'') of basic
homotopy functors that arise in topology in terms of polynomial functors. This method --- introduced by T. Goodwillie in the series of papers \cite{CalcI, CalcII, CalcIII} --- has been studied extensively in recent years and has found many interesting applications (see, e.g., the survey papers \cite{ArCh} and \cite{Kuhn}).

Recall that by a homotopy functor we mean a functor on topological spaces that preserves
weak homotopy equivalences. A homotopy functor $ F: \Top_* \to \Top_* $ 
is called {\it $n$-excisive} (or {\it polynomial of degree} $ \le n $) if it takes any strongly coCartesian $(n+1)$-dimensional cubical diagram in $ \Top_* $
to a Cartesian diagram ({\it cf.} \cite[Definition 1.1.2]{ArCh}). 
For $n=0$, this simply means that $F$ is homotopically constant: i.e. $ F(X) \simeq F(\ast)$ for any $X \in {\rm Top}_\ast$. For $n=1$, this is the usual Mayer-Vietoris
property: a functor $F$ is $1$-excisive if and only if it
maps homotopy pushout squares to homotopy pullback squares in ${\rm Top}_\ast$
(see \cite[Example 1.1.4]{ArCh}). For $ n > 1 $, $F$ enjoys a higher dimensional version of the Mayer-Vietoris property that reduces to the usual one inductively in $n$.

The main construction of Goodwillie calculus can be described as follows ({\it cf.} \cite[Theorem 1.8]{CalcIII}).
\begin{theorem}[Goodwillie] \la{TaylorS}
For any homotopy functor $\,F:{\rm Top}_\ast \to {\rm Top}_{\ast}\,$ on pointed spaces, there exists a natural tower of functors $($fibrations$)$ under $F$:
\begin{equation} \la{s4.4e1} 
\begin{diagram}[small, nonhug]
 & & \vdots\\
 &  & \dTo_{p_3}\\
 & & P_2F(X)\\
 & \ruTo(2,4)^{\ \delta_2}\quad &\dTo_{p_2}\\
 &  & P_1F(X)\\
&\ \ruTo^{\ \delta_1}& \dTo_{p_1}\\
F(X)& \rTo^{\ \delta_0} & P_0F(X)\\
\end{diagram}
\end{equation}
satisfying the following properties: for all $ n\ge 0 $,
\begin{enumerate}
\item[$(1)$] $\,P_nF:\, {\rm Top}_\ast \to {\rm Top}_{\ast} $ is an $n$-excisive functor, 
\item[$(2)$] 
$\delta_n:\,F \to P_nF$ is the universal weak natural transformation to an $n$-excisive functor.
\end{enumerate}
\end{theorem}

The last property needs an explanation. By a {\it weak}\, natural transformation $\, \delta: F \to P $ one means a pair (`zig-zag') of natural transformations
$\,F \xrightarrow{\delta'} G \xleftarrow{\delta''} P\,$, where $ \delta'' $ is a natural
weak equivalence, i.e. $\, \delta_X'': G(X) \xleftarrow{\sim} P(X) \,$ is a weak homotopy equivalence for all spaces $ X \in \Top_* $. Note that if $F$ and $P$ are homotopy functors, a weak natural transformation $\, \delta: F \to P $ induces a well-defined natural tranformation between the corresponding functors on the homotopy category $ \Ho(\Top_*)$.
Property (2) of Theorem~\ref{TaylorS} then says that the weak natural transformation $\, \delta_n:\,F \to P_nF$ is homotopically initial among all natural transformations from $F$ to $n$-excisive functors.
 
Given a homotopy functor  $F: \Top_* \to \Top_* $, we define its {\it $n$-th  layer} to be the homotopy fibre 
\begin{equation} \la{s4.4e2} D_nF(X)\,:=\,{\rm hofib}\{P_nF(X) \xrightarrow{p_n} P_{n-1}F(X) \} \end{equation}
where $p_n$ is the canonical projection at the $n$-th stage of the Goodwillie tower \eqref{s4.4e1}. A remarkable fact discovered in \cite{CalcIII} (see \cite[Example 1.2.4]{ArCh}) is that, all layers of a homotopy functor $F$ are naturally infinite loop spaces. More precisely, for each $n \geqslant 0$, there is a spectrum $ \partial_n F $ equipped with a (na\"ive) action of the symmetric group $ S_n $ such that
\begin{equation}\la{GD1} 
D_nF(X) \,\simeq \,\Omega^{\infty}(\partial_n F \wedge (\Sigma^\infty X)^{\wedge n})_{hS_n}\,,
\end{equation}
where $\,(\Sigma^{\infty},\, \Omega^{\infty}) \,$ are the suspension spectrum and the infinite loop space functors, respectively. The spectrum $ \partial_n F $ is called the {\it $n$-th Goodwillie derivative of $F$} (at the basepoint $ \ast$).

\subsection{The map $\CS$}
\la{S5.2}
Recall that, by Corollary \ref{topch}, the derived character map $\Tr(\Gamma)_\ast $ is induced by the composition of natural maps in $ \Ho(\Top_*)$:
\begin{equation} \la{s4.4e3} ES^1 \times_{S^1} \cL(X) \,\xrightarrow{\CS_{X}}\,
\Omega \Omega^{\infty} \Sigma^{\infty} (X)\, \xrightarrow{\SR_{X}}\, 
\Omega {\rm SP}^{\infty}(X) \end{equation}
where $X=B\Gamma$. Since the classifying space functor on homotopy simplicial groups  induces an equivalence $\,\Ho(\sGr^h) \cong \Ho({\rm Top}_{0,\ast})\,$, the maps \eqref{s4.4e3} are defined on (the homotopy types of) all pointed connected spaces. To analyze these maps we introduce the notation:
$$\Theta(X) \,:= \, ES^1 \times_{S^1} \cL(X) \,=\, ES^1 \times_{S^1} {\rm Map}(S^1,X)$$
and define $\bar{\Theta}:\, {\rm Top}_\ast \to {\rm Top}_\ast$ by
\begin{equation} \la{s4.4e4} \bar{\Theta}(X)\,:=\, \Theta(X)/\Theta(\ast) \,\cong\, ES^1 \times_{S^1} \cL(X)/BS^1 \,\cong\, ES^1_{+} \wedge_{S^1} \cL(X) 
\end{equation}
Note that \eqref{s4.4e4} is a {\it reduced}\, homotopy functor, so that
$\, P_0\bar{\Theta}(X) \simeq \bar{\Theta}(\ast) = \{\ast\} \,$ and $\,P_1\bar{\Theta}(X) \cong D_1\bar{\Theta}(X) $ for any space $X \in \Top_*$ (see \eqref{s4.4e2}).

The next proposition shows that the natural transformation $\CS$  in \eqref{s4.4e3}, relating cyclic to symmetric homology, essentially coincides with the first Goodwillie layer of the functor \eqref{s4.4e4}. We deduce this from results of Carlsson and Cohen \cite{CC} by elaborating on a remark of Fiedorowicz \cite{F}.
\begin{proposition} 
\la{csx}
The map $\CS$ in \eqref{s4.4e3} is represented by  
 \begin{equation*}
 \begin{diagram}[small]
ES^1 \times_{S^1} \cL(X)& &\rTo^{\CS_{X}} & & \Omega\, \Omega^{\infty} \Sigma^{\infty} (X) \\
\Big| \Big| & &  &  &\dTo_{\wr}\\
\Theta(X) & \rOnto^{{\rm can}\qquad} & \bar{\Theta}(X) & \rTo^{\quad\delta_{1,X}} & D_1\bar{\Theta}(X)
 \end{diagram}
 \end{equation*}
where the right vertical arrow is a natural weak equivalence and
$\delta_1$ is the 1-st layer of the functor \eqref{s4.4e4}.
\end{proposition}
\begin{proof} As noticed in \cite[Remark 1.4]{F}),  the map $\CS_X$ factors in the homotopy category as
\begin{equation} \la{s4.4e5} 
ES^1 \times_{S^1} \cL(X) \xrightarrow{{\rm can}}  ES^1_{+} \wedge_{S^1} \cL(X) \xrightarrow{f_X} \Omega \Omega^{\infty} \Sigma^{\infty}(X) 
\end{equation}
where $ f_X $ is a certain natural map constructed in \cite{CC}. We review the construction of $ f_X $ and compare it to a well-known general formula for the first Goodwillie layer of a reduced homotopy functor.

First, we recall a standard stabilization construction 
due to Waldhausen \cite{Wald}. For a pointed space $ X $, denote by $ CX = X \wedge I $ and $ \Sigma X = X \wedge S^1 $ the reduced cone
and the reduced suspension of $X$, respectively. The latter can be obtained by glueing two copies of the former along a common base which is identified with $X$: this yields the natural pushout square in $ \Top_*$
 \begin{equation}
 \la{squa}
 \begin{diagram}[small]
 X  & \rTo & CX \\
\dTo &     & \dTo_{j_0}               \\
CX  &  \rTo_{j_1}    & \Sigma X
 \end{diagram}
 \end{equation}
Applying the given functor $F$ to \eqref{squa}
and taking the homotopy pullback along the maps $ j_0 $
and $j_1$ induces a natural map
\begin{equation}
\la{Whoc}
F(X) \to \hocolim\,[\,F(CX) \xrightarrow{} 
F(\Sigma X) \xleftarrow{} F(CX)\,]
\end{equation}
Since the functor $F$ is homotopic and reduced, we have $\,F(CX) \simeq F(\ast) \simeq \{\ast\}\,$, which implies that the homotopy colimit in  \eqref{Whoc} is equivalent to $\Omega F(\Sigma X)$. In this way, we get a natural map $ s: F(X) \to \Omega F(\Sigma X) $. This last map can be iterated any number of times: 
\begin{equation}
\la{snmap}
s_{n}:\ F(X)\,\to\, \Omega^n F(\Sigma^n X)  \ , \quad n \ge 0\ ,
\end{equation}
and eventually stabilized, defining the map
\begin{equation}
\la{stabmap}
s_{\infty}:\ F(X)\,\to\, \varinjlim_{n}\, \Omega^n F(\Sigma^n X)\, =\, \Omega^\infty F \Sigma^{\infty}(X)   
\end{equation}
In particular,  \eqref{stabmap} exists for our functor $ F = \bar\Theta $, see \eqref{s4.4e4}.

Next, for each $n \geqslant 0$, define $\,\Sigma^n X \to \bar{\Theta}\Sigma^n(\Sigma X)\,$ to be the composition of the following natural maps
$$\Sigma^n X \stackrel{\varepsilon}{\to} \Omega \Sigma(\Sigma^n X)= \Omega(\Sigma^{n+1}X) \hookrightarrow \cL(\Sigma^{n+1} X) \simeq ES^1 \times \cL(\Sigma^{n+1} X) \twoheadrightarrow ES^1_+ \wedge_{S^1} \cL(\Sigma^{n+1} X) = \bar{\Theta}\Sigma^n(\Sigma X)\,, 
$$
where $ \varepsilon: \id \to \Omega \Sigma $ is the adjunction unit of $(\Sigma, \Omega) $.
Looping $n$ times then yields an inductive system of maps
\begin{equation} \la{s4.4e6} i_n :\,
\Omega^n \Sigma^n X \to \Omega^n \bar{\Theta} \Sigma^n (\Sigma X) \ , 
\quad \forall\,n \ge 0\ ,
\end{equation}
which, by \cite[Lemma 4.1]{CC}, induce in the limit a homotopy equivalence 
\begin{equation} 
\la{s4.4e7} 
i_{\infty}:\ \Omega^{\infty}\Sigma^{\infty} X \,\xrightarrow{\sim}\, \Omega^\infty \bar{\Theta} \Sigma^\infty (\Sigma X) 
\end{equation}
Finally, we note the following canonical identifications
\begin{eqnarray}\la{ident}
\lefteqn{
\Omega^\infty \bar{\Theta} \Sigma^\infty (X)\, := \,\varinjlim_{n}\, \Omega^n \bar{\Theta} \Sigma^n (X) \,=\,  \varinjlim_{n}\, \Omega^{n+1} \bar{\Theta} \Sigma^{n+1} (X)\,=}  \nonumber\\
& &  =\, \varinjlim_{n}\, \Omega [\Omega^n \bar{\Theta} \Sigma^n ( \Sigma X)]
\cong  \Omega \varinjlim_{n}\, [\Omega^n \bar{\Theta} \Sigma^n ( \Sigma X)] 
= \Omega\,\Omega^{\infty}\bar{\Theta}\Sigma^{\infty}(\Sigma X)
\end{eqnarray}
The Carlsson-Cohen map $ f_X  $ that appears in \eqref{s4.4e5} can now be represented by the zig-zag of natural transformations
\begin{equation} 
\la{s4.4e8} 
\bar{\Theta}(X) \xrightarrow{s_\infty} 
\Omega^{\infty}\bar{\Theta} \Sigma^{\infty}(X) \stackrel{\eqref{ident}}{\cong} \Omega\,\Omega^{\infty}\bar{\Theta}\Sigma^{\infty}(\Sigma X)
\xleftarrow{\Omega i_{\infty}} \Omega\, \Omega^{\infty}\Sigma^{\infty}(X)\,,
\end{equation}
where the leftmost arrow is the Waldhausen stabilization map \eqref{stabmap} for  $\bar{\Theta}$ and the rightmost arrow is a natural weak equivalence induced by \eqref{s4.4e7}. To complete the proof it remains to note that $\, P_1(F) \simeq \Omega^\infty F \Sigma^\infty $ for any reduced homotopy functor $F$, and
the universal natural transformation $ \delta_{1}: F \to P_1 F = D_1 F $ coincides (up to homotopy) with the stabilization map \eqref{stabmap} (see, e.g. \cite[Example 5.3]{Kuhn}).
\end{proof}
\subsection{The map $ \SR $} 
\la{S5.3}
We now turn to the second map $ \SR_X $ in \eqref{s4.4e3} that relates symmetric homology to representation homology. In this section, we construct this map topologically by a method similar to that of Proposition~\ref{csx}; its relation to Goodwillie calculus will be discussed
in Section~\ref{S5.4}. Our starting point is the well-known fact that the Dold-Thom functor $ \SP^\infty:\, \Top_* \to \Top_* $ factors through the category of abelian topological monoids: in fact,
$ \SP^\infty(X)$ is the free abelian topological monoid generated by the space $X$ (see, e.g., \cite{McC69}). This implies that $ \SP^{\infty}$ is a linear (i.e., $1$-excisive) functor. The latter can be seen directly as follows. Consider the natural maps \eqref{snmap} for the functor $ F = \SP^\infty $ constructed in the proof of Proposition~\ref{csx}:
\begin{equation}
\la{s_n}
s_n:\ \SP^\infty(X)\,\to\,\Omega^n \SP^\infty \Sigma^n(X)\ ,\quad n \ge 0\ ,
\end{equation}
The maps \eqref{s_n} are all weak equivalences, which follows immediately from the commutative diagrams
 \begin{equation*}
 \begin{diagram}[small]
 \pi_i\,\SP^\infty(X)  & \rTo^{\ \pi_i(s_n)\ } & \pi_i\,\Omega^n \SP^\infty \Sigma^n(X) \\
 \dTo^{\wr} &     & \dTo_{\wr}               \\
\tilde{H}_i(X)  &  \rTo^{\sim}    
& \tilde{H}_{i+n}(\Sigma^n X)
 \end{diagram}
 \end{equation*}
where the vertical arrows are isomorphisms by the Dold-Thom Theorem. Thus, in the limit, we get 
\begin{equation}
\la{s_inf}
s_\infty:\ \SP^\infty(X)\,\sto\,\Omega^\infty \,\SP^\infty\, \Sigma^\infty(X)\ ,
\end{equation}
showing that $\,\SP^\infty \simeq P_1(\SP^\infty) \simeq D_1(\SP^\infty)\,$, whence the linearity of $\SP^\infty$.

On the other hand, for all $n\ge 0$, we have canonical maps $\, \Sigma^n X \to \SP^\infty(\Sigma^n X)\,$
inducing the Hurewicz homomorphisms, see \eqref{s4e8}). Applying loop functors to these maps 
yield an inductive system of maps
\begin{equation}
\la{i_n}
i_n:\ \Omega^n \Sigma^n(X)\,\to\,\Omega^n \SP^\infty \Sigma^n(X)\ ,\quad n \ge 0\ ,
\end{equation}
which in the limit, induces
\begin{equation}
\la{i_inf}
i_\infty:\ \Omega^\infty \Sigma^\infty (X)\,\to\,\Omega^\infty \SP^\infty \Sigma^\infty(X)
\end{equation}
Unlike the analogous map \eqref{s4.4e7} for the functor $ \bar\Theta$,
\eqref{i_inf} is not a weak equivalence in general. Nevertheless, looping it once
and combining with \eqref{s_inf}, we get the pair of natural transformations
\begin{equation}
\la{omis}
\Omega \Omega^\infty \Sigma^\infty(X) \,\xrightarrow{\Omega i_\infty}\,
\Omega \Omega^\infty \SP^\infty \Sigma^\infty(X) \,\xleftarrow{\Omega s_\infty}\,
\Omega \SP^\infty(X)
\end{equation}
where the rightmost one is a natural weak equivalence. Our goal is to prove
the following analogue of Proposition~\ref{csx}.
\begin{proposition}
\la{srx}
The map $ \SR $ is represented by the weak natural transformation \eqref{omis}. Thus, in the homotopy category, $ \SR_X$ is equivalent to the map
\begin{equation}
\la{omis1}
(\Omega s_{\infty})^{-1}\,(\Omega i_\infty): \ \Omega \Omega^\infty \Sigma^\infty(X)\ \to \ \Omega \SP^\infty(X)
\end{equation}
which is the $($looped once$)$ canonical natural transformation relating stable homotopy to $($reduced$)$ singular
homology of pointed spaces.
\end{proposition}
To prove this proposition we will reinterpret the map \eqref{omis1} in terms of (topological) operads. The standard reference for the background material that we need is \cite{May72}
(for a brief introduction, see also \cite[Chapter 12]{R20}). 
Recall that an {\it operad $\C $} in $\Top_*$ is a collection of pointed spaces $ \{\C(j)\}_{j \ge 0} $ with $ \C(0) := \{\ast\} $ such that each $ \C(j) $ carries a right $S_j$-action and there are composition laws $ \C(k) \times \C(j_1) \times \ldots \times \C(j_k) \to \C(j_1 + \ldots +j_k) $ satisfying natural associativity and unitality conditions. If $ \C $ is an operad, a {\it $\C$-space} is a pointed space $X$ equipped with an action of $\C$, which is given by a sequence of $S_j$-equivariant maps $\,\theta_j: \C(j) \times X^{j} \to X \,$, with 
$ \theta_0: \C(0)\into X $ being the basepoint inclusion, that satisfy associativity
and unitality conditions compatible with those of $\C$. Every operad $\C $
determines a monad $\, \underline{\C} $ on $ \Top_* $ (i.e., a monoid with respect to `$\circ$' in the category of endofunctors $ \Top_* \to \Top_* $) in such a way that the notion of a $\C$-space is equivalent to that of $ \underline{\C} $-algebra. Explicitly, 
given an operad $ \C $, the corresponding monad $ \underline{\C}: \Top_* \to \Top_* $ is defined by
\begin{equation}\la{monad}
\underline{\C}(X) := \coprod_{j \ge 0}\, \left(\C(j) \times_{S_j} X^j\right)/\sim
\end{equation}
where the equivalence relation is of the form
$$ 
(c, x_1, \ldots, x_{i-1},\ast, x_{i+1}, \ldots, x_j) \sim (\sigma_i(c), x_1, \ldots, x_{i-1}, x_{i+1}, \ldots, x_j) 
$$ 
for certain natural maps $ \sigma_i:\, \C(j) \to \C(j-1) $ (see \cite[Construction 2.4]{May72}). A {\it $\underline{\C}$-algebra} is then defined to be a space $A \in \Top_* $ with an action map $ \xi: \underline{\C}(A) \to A $ satisfying natural associativity and unitality conditions. Opposite to the notion of a $ \underline{\C} $-algebra is that of a {\it $ 
\uC$-functor}, which a functor $F$ on $\Top_* $ equipped a morphism $ F \circ \underline{\C} \to F $ defining a right action of $ \underline{\C}$ on $F$. Associated to a triple $(F,\, \uC, \, A)$, there is a {\it two-sided  bar construction} $\, B(F,\, \uC, \, A)\,$ defined as the geometric realization of a simplicial space $ B_\ast(F,\, \uC, \, A) \in s\Top_*$ with components
\begin{equation}
\la{simpbar}
B_n(F,\, \uC, \, A) := F \uC^n(A)\ ,\quad n\ge0\,, 
\end{equation}
where the faces $  d_i: B_{n} \to B_{n-1} $ and degeneracies $s_j: B_n \to B_{n+1} $ are 
determined by the structure maps of $A$ and $F$ (see \cite[Construction 9.6]{May72}).

Now, our main examples  will be the so-called {\it little cubes operads} $ \{\C_0,\,\C_1,\,\C_2, \ldots\} $ originally introduced by Boardman and Vogt (see \cite[Section 4]{May72}). The $ \C_0 $ and $ \C_1 $ 
are discrete operads\footnote{These operads are denoted in \cite{May72} by 
$ \mathfrak{N}$ and $\mathfrak{M}$, respectively.} defined by $ \C_0(j) := \{\ast\} $ and $ \C_1(j) := S_j $ for all $ j\ge 0$, with $S_j$-action being trivial in the former case and  induced by multiplication in $S_j$ in the latter. A $\C_0$-space is just an abelian monoid in $\Top_*$, and the monad associated to $\C_0 $ is precisely the Dold-Thom functor: 
\begin{equation}
\la{Cosp}
\uC_0(X) \,\cong\, \SP^{\infty}(X)
\end{equation}
A $\C_1$-space is just a monoid in $\Top_*$ (i.e., an associative $H$-space with $1$), and the monad associated to $\C_1 $ yields the classical James functor:
\begin{equation}\la{Comon}
\uC_1(X) \,\cong \, J(X)  
\end{equation}
where $\,  J(X) = (\amalg_{n \ge 0} X^n)/\!\sim \,$ is the free topological monoid generated by $X$.
For $ n\ge 2$, the operad $ \C_n $ is not discrete:  for $ j \ge 1 $, the space $ \C_n(j) $  can be represented by the $j$-tuples of `little $n$-cubes' (i.e. linear embeddings $ I^n \into I^n $ with parallel axes and disjoint interiors)  with natural (permutation) $S_j$-action. Thus, for $n\ge 2$, each  $ \C_n(j) $ is homotopy equivalent to ${\rm conf}_j({\mathbb R}^n) $,  the configuration space of unordered $j$-tuples of points in $ {\mathbb R}^n $
equipped with canonical free $S_j$-action. Natural inclusions of cubes $ I^n \into I^{n+1} $ induce the embeddings of spaces $ \C_n(j) \into \C_{n+1}(j) $, and hence the maps of operads $ \C_n \into \C_{n+1} $ for all $ n\ge 2$.
This allows one to define the operad $ \C_{\infty} := \varinjlim_{n} \C_n $. Since $\pi_i[\C_n(j)] \cong \pi_i[{\rm conf}_j({\mathbb R}^n)] = 0 $ for $ i \le n-2 $, each component $ \C_\infty(j) $ of $ \C_\infty $ is contractible, and as the $S_j$-action on $ \C_\infty(j) $ (induced from $ \C_n(j) $) is free,  $\, \C_\infty \,$ is an $E_\infty$-operad. Finally, we recall  May's Approximation Theorem (see \cite[Theorem 2.7]{May72}) that asserts that the natural map of monads 
$\, 
\alpha_n:\, \uC_n(X) \to \uC_n \Omega^n \Sigma^n(X) \to \Omega^n \Sigma^n(X) 
\,$ 
gives a homotopy equivalence
\begin{equation}
\la{Mayappr}
\uC_n(X)\, \simeq \, \Omega^n \Sigma^n(X)\ ,\quad
\forall\,n = 1,2,\ldots, \infty\,  ,
\end{equation}
whenever $X$ is connected. 

We can now state the following result which is probably well known to experts.

%
\begin{lemma}[{\it cf.} \cite{F}] 
\la{maybar}
For any topological monoid $M$, there are natural homotopy equivalences 
\begin{eqnarray} 
B(\uC_\infty, \,\uC_1,\, M) & \simeq & \Omega\,\Omega^\infty \Sigma^\infty(BM)\,, \la{ombar}\\
B(\uC_0, \,\uC_1,\, M) & \simeq & \Omega\,{\rm SP}^{\infty}(BM)\,,  \la{omspbar}
\end{eqnarray}
and the map \eqref{omis1} for $X = BM $ is equivalent to the map
\begin{equation}
\la{omis2}    
B(\uC_\infty, \,\uC_1,\, M)\,\to \,B(\uC_0, \,\uC_1,\, M)
\end{equation}
induced by the canonical  $($unique$)$ morphism of operads $\, \C_\infty \to \C_0\,$.
\end{lemma}
\begin{proof} The equivalence \eqref{ombar} was originally proved by Fiedorowicz (see \cite[Proposition 1.7]{F} and also \cite[Lemma 39]{Au1}); the proof of  \eqref{omspbar} is  similar. We describe these equivalences in both cases. First,
\begin{eqnarray*}
\lefteqn{B(\uC_\infty,\, \uC_1,\, M) 
\,\simeq\,  
B(\Omega^\infty \Sigma^\infty, \,\uC_1,\, M) \,\simeq }\\*[1ex]
&& B(\Omega\, \Omega^\infty \Sigma^\infty \Sigma, \,\uC_1,\, M)  \,\simeq\,
\Omega\, \Omega^\infty \Sigma^\infty 
B(\Sigma, \,\uC_1,\, M)
\,\simeq\,
\Omega\,\Omega^\infty \Sigma^\infty(BM)\, ,
\end{eqnarray*}
where the first equivalence is induced by  \eqref{Mayappr}, the second is obvious, the third is
 a formal property of the bar construction (see \cite[Lemma 9.7]{May72}), and the last one follows from a theorem of Fiedorowicz (see \cite[Corollary 9.7]{F84}) that yields $\, B(\Sigma,\,\uC_1,\,M) \simeq BM  \,$ for any topological monoid $M$. 
 Similarly,
\begin{equation*}
B(\uC_0,\, \uC_1,\, M) \,\cong\,  
B({\rm SP}^{\infty}, \,\uC_1,\, M) \,\simeq\,
B(\Omega\,{\rm SP}^{\infty}\Sigma, \,\uC_1, \,M)  \,\simeq\,
\Omega\,{\rm SP}^{\infty}B(\Sigma,\, \uC_1,\, M) \,\simeq\,
\Omega\,{\rm SP}^{\infty}(BM)\, ,
\end{equation*}
where the first identification follows from \eqref{Cosp}, the second is induced by the equivalence \eqref{s_n}, which is a consequence of the Dold-Theorem, the third follows from \cite[Lemma 9.7]{May72}, and the last one is  \cite[Corollary 9.7]{F84}. The last statement of the lemma is now deduced by comparing the above equivalences with the construction of the map \eqref{omis1} given in the beginning of Section~\ref{S5.3}. 
\end{proof}
%
%
%

\begin{proof}[Proof of Proposition \ref{srx}]
For any topological monoid $M$, consider the following diagram of spaces
 \begin{equation}\la{maindiag}
 \begin{diagram}[small]
  |{\rm hocolim}_{\Delta S_+}(B_{\rm sym}M)| & \rTo^{\ f_{\infty}\ } & B(\uC_{\infty}, \uC_1, M) & \rTo^{\eqref{ombar}} & \Omega\,\Omega^\infty \Sigma^\infty(BM) \\
  \dTo^{\tilde{\Psi}_{{\rm sym}}^*} & &  \dTo^{{\rm can}} & & \dTo_{\eqref{omis1}}  \\
  |{\rm hocolim}_{\fM_{\N}^{\rm op}}(p^\ast M)| & \rTo^{\ f_0\ } & B(\uC_{0}, \uC_1, M) & \rTo^{\eqref{omspbar}} & \Omega\,\SP^\infty (BM) 
  \end{diagram}
 \end{equation}
In this diagram all horizontal maps are natural weak equivalences:  $f_\infty$ is the equivalence
constructed by Fiedorowicz in \cite{F} (see \cite[Theorem 38]{Au1}), $\,f_0$ is the equivalence \eqref{spacemon} of Proposition~\ref{propmon}, and \eqref{ombar}, \eqref{omspbar} are the equivalences described in 
Lemma~\ref{maybar}. The map $ \tilde{\Psi}^{*}_{\rm sym}$ is induced by the functor 
$ \tilde{\Psi}_{\rm sym} $ defined in \eqref{diagmon}. To prove the proposition we need to show that 
the diagram \eqref{maindiag} commutes. By Lemma~\ref{maybar}, we already know that the rightmost square of \eqref{maindiag}  commutes; thus it suffices to prove the commutativity of the leftmost square.
For this, we shall describe the maps $ f_{\infty}$ and $ f_0 $ explicitly. 

The map $ f_{\infty}$ is explicitly constructed in the proof of \cite[Lemma 36]{Au1}. As in {\it loc. cit.}, we let $\cN:\Top_\ast \to \Top_\ast$ denote the functor defined as the coend
$$ \cN(X) := \int^{[n] \in \Delta S_+} N([n] \downarrow \Delta S_+) \times B_{\rm sym}J(X)[n]\ .$$
By \cite[Lemma~36]{Au1}, there is an equivalence of functors $\Theta:\cN \simeq \uC_\infty$, 
inducing an equivalence of bar constructions 
$B(\cN,\uC_1,M) \simeq B(\uC_\infty,\uC_1,M)$. The identification $|{\rm hocolim}_{\Delta S_+}(B_{\rm sym} M)| \simeq B(\cN,\uC_1,M) $ by \cite[Lemma 9.7]{May72} then yields $f_\infty$.

The map $f_0$ can be constructed in a similar way. Let $\cP:\Top_\ast \to \Top_\ast$ denote the functor 
$$\cP(X) := \int^{(\langle n \rangle; k_1,\ldots,k_n) \in \fM_{\N}^{\rm op}} N((\langle n \rangle; k_1,\ldots,k_n) \downarrow \fM_{\N}^{\rm op}) \times p^{\ast}J(X)(\langle n \rangle; k_1,\ldots,k_n) \ . $$
Identifying $J(X)(\langle n \rangle) = {\rm Hom}_{\rm Mon}(\langle n \rangle , J(X))$ and recalling that $\uC_0(X) = {\rm SP}^{\infty}(X)$ is the abelianization of $J(X)$, we note that the map 
$$ \coprod N((\langle n \rangle; k_1,\ldots,k_n) \downarrow \fM_{\N}^{\rm op}) \times p^{\ast}J(X)(\langle n \rangle; k_1,\ldots,k_n) \to \uC_0(X) = {\rm SP}^{\infty}(X)\,,\qquad y \times \varphi  \mapsto \varphi_{\rm ab}(k_1,\ldots, k_n) $$
descends to the coend to yield a natural equivalence 
$$ \Lambda: \cP(X) \simeq \uC_0(X)\,,$$
which, in turn, yields an equivalence of bar constructions $B(\cP,\uC_1,M) \simeq B(\uC_0,\uC_1,M)$. Composing this with the identification $|{\rm hocolim}_{\fM_{\N}^{\rm op}}(p^{\ast}M)| \simeq B(\cP,\uC_1,M)$ by \cite[Lemma9.7]{May72} yields the map $f_0$.

It can be easily verified that the following diagram commutes. 
$$ 
\begin{diagram}[small]
 \cN & \rTo^{\Theta} & \uC_{\infty}\\
  \dTo^{\tilde{\Psi}_{\rm sym}^{\ast}} & & \dTo_{{\rm can}}\\
  \cP & \rTo^{\Lambda} & \uC_0
\end{diagram}
$$
It follows that the first square in the diagram \eqref{maindiag} commutes. Finally, we note that in the case when $M=\Gamma$, a simplicial group the map $\tilde{\Psi}^{\ast}_{\rm sym}$ in the diagram \eqref{maindiag} may be identified with the corresponding map in \eqref{s4.3e10} by Corollary \ref{cor33} (also see Remark \ref{sgrsmon}). This completes the proof ofthe desired proposition.
\end{proof}

\begin{corollary}
\la{Cor22}
Let $ \Gamma $ be a $($homotopy$)$ simplicial group such that $ X = B \Gamma $ has homotopy type of a simply-connected CW complex, which is of $($locally$)$ finite rational type.
If $k$ is a field of characteristic zero, then the map $ \SR_X $ induces an isomorphism 
$$
\HS_\ast(k[\Gamma]) \cong \HR_*(k[\Gamma])\,.
$$
\end{corollary}
\begin{proof}
As mentioned above, the natural map $\, i_\infty:\, \Omega^\infty \Sigma^\infty(X) \to  \SP^{\infty}(X)\,$ (defined by composing \eqref{i_inf} with the inverse of \eqref{s_inf} in $ \Ho(\Top_*) $) is not an equivalence in general. However, it is known that for any connected CW complex $X$, 
this map induces an isomorphism of cohomology rings
\begin{equation}
\la{i_inf1}
i^*_{\infty}: \ H^*(\SP^{\infty}(X),\,k)\,\xrightarrow{\sim}\,  H^*(\Omega^\infty \Sigma^\infty(X),\,k)   
\end{equation}
provided the coefficients are taken in a field $k$ of characteristic zero (see, e.g., \cite[Section 7.3]{CM95}). 
Now, under our assumption on $X$, both $ \SP^{\infty}(X) $ and $\Omega^\infty \Sigma^\infty(X)$ are simply-connected spaces of finite rational type. Hence,  there is a natural (`Cotor') spectral sequence with $E_2$-term $ E^{\ast, \ast}_2(Z) = \Ext^{*}_{H^*(Z,\,k)}(k,\,k) $ that converges to $ H_*(\Omega Z,\,k)$ for any simply-connected space $Z$ (see, e.g., \cite[Section 5.5, (5.13)]{CM95}). By naturality, the map \eqref{i_inf1} induces an isomorhism 
$\, E^{\ast, \ast}_2(\Omega^\infty \Sigma^\infty X) \xrightarrow{\sim}  E^{\ast, \ast}_2(\SP^{\infty} X)\,$ of such spectral sequences for $ Z = \Omega^\infty \Sigma^\infty(X) $ and $ Z = \SP^{\infty}(X) $. This last isomorphism is compatible with the map $\, \Omega i_{\infty}: \,H_*(\Omega\, \Omega^\infty \Sigma^\infty(X),\, k) \,\to\,H_*(\Omega\,\SP^{\infty}(X),\,k) \,$ which, by Proposition~\ref{srx}, coincides with 
$ \SR_X: \HS_*(k[\Gamma]) \to \HR_*(k[\Gamma]) $ for $ X = B\Gamma$. Thus, by Comaprison Theorem for spectral sequences (see \cite[Theorem 5.2.12]{W}), we conclude that $ \SR_X $ is an isomorphism.
\end{proof}
\begin{remark}\la{Rem5.1}
We expect that the result of Corollary~\ref{Cor22} holds for {\it any} homotopy simplicial group $\Gamma$,
including the usual (discrete) groups, for which $B\Gamma$ is a $K(1, \Gamma)$-space, i.e. certainly not simply connected.
\end{remark}

\subsection{Polynomial extensions}
\la{S5.4}
There is a natural way to describe and generalize the map $\SR $ via Goodwillie calculus. As we have seen above, the Dold-Thom functor $ \SP^\infty $ is $1$-excisive, hence there is a canonical (up to homotopy) natural transformation $\,\beta_{1}:\, P_1(\id) \to \SP^\infty $, where $ P_1(\id) = D_1(\id) $ is the 1-st layer of the functor $\id$. The latter is known to be the stable homotopy functor $ P_1(\id) \simeq \Omega^\infty \Sigma^\infty $ and $\, \beta_{1} \simeq i_{\infty} $. Thus $\, \SR \simeq \Omega \beta_{1} \,$. It turns out that the map $ \beta_1 $ can be extended naturally to higher layers --- and in fact, to the entire Goodwillie tower of the  functor $\id $.   This is based on results of the paper \cite{BD} that compares the Goodwillie tower of the identity with the lower central series of the Kan loop group.

Recall that, for any connected space $X$, we can identify
$ \SP^\infty(X) \simeq B[{\mathbb A}(X)] $, where $ {\mathbb A}(X) :=  \lgr(X)_{\rm ab} $ is the abelianization of the Kan loop group $\lgr(X)$  of (a reduced simplicial set representing) $X$, see \eqref{s4e11}. Now, instead of just abelianization, consider the lower central series of $ \lgr(X) $:
$$
\ldots\, \to \,\lgr(X)/\lgr_{n+1}(X)\,\to\, \lgr(X)/\lgr_n(X) \,\to\,\ldots \,\to\, \lgr(X)/\lgr_2(X) = {\mathbb A}(X)
$$
where $ \lgr_n(X) $ are the simplicial subgroups of $ \lgr(X) $ defined inductively by
$$ 
\lgr_1(X) := \lgr(X) \quad  \mbox{and}\quad
\lgr_{n+1}(X) := [\lgr(X), \lgr_n(X)]\,,\ n\ge 1\,.
$$
It is shown in \cite{BD} that the functor $\,X \mapsto B[\lgr(X)/\lgr_{n+1}(X)]\,$ is
$n$-excisive for each $n \ge 1$, and there exists a canonical (up to homotopy) morphism of towers
\begin{equation}\la{GoodKan}
\begin{diagram}[small]
\ldots & \rTo & P_n(\id)(X) & \rTo & P_{n-1}(\id)(X)& \rTo & \ldots &\rTo & P_1(\id)(X)\\
&  & \dTo^{\beta_n} &  & \dTo^{\beta_{n-1}} &  & \ldots & & \dTo^{\beta_1}\\
\ldots & \rTo & B[\lgr(X)/\lgr_{n+1}(X)] & \rTo & B[\lgr(X)/\lgr_{n}(X)]& \rTo & \ldots & \rTo & B[{\mathbb A}(X)]
\end{diagram}    
\end{equation}
where the rightmost vertical arrow is precisely the map $\beta_1: P_1(\id) \to \SP^\infty $. 
This morphism induces natural maps on the layers of the Goodwillie tower
\begin{equation}
\la{layern} 
\beta_n:\ D_n(\id)(X) \,\to \,B[\lgr_n(X)/\lgr_{n+1}(X)]\ ,\ n \ge 1\ .
\end{equation}
that we can describe in explicit terms. First of all, by a theorem of 
B. Johnson \cite{J95}  ({\it cf.} \cite[Example 1.2.5]{ArCh}), all Goodwillie derivatives of the identity functor are known: for $ n \ge 1 $, the spectrum $ \partial_n(\id) $ is  equivalent to 
a wedge of $ (n-1)! $ copies of the $(1-n)$-sphere spectrum $ \bS^{1-n} = 
\Sigma^{1-n}(\bS^0) $. Hence, by formula \eqref{GD1}, we have
\begin{equation}
\la{bform}
D_n(\id)(X) \,\simeq\, \Omega^\infty\biggl(\,\Mvee_{(n-1)!} \Sigma^{1-n} \,
(\Sigma^{\infty} X)^{\wedge n} \biggr)_{hS_n}     
\end{equation}
On the other hand, the Kan simplicial group $ \lgr(X) $ is (degreewise) free for any $X$. Hence, by classic PBW Theorem (see, e.g., \cite[I.4.3]{Serre}),  for all $ n \ge 1 $, there are natural isomorphisms of simplicial abelian groups
\begin{equation}
\la{Lien}
\lgr_n(X)/\lgr_{n+1}(X)\,\cong\, \Lie_n[{\mathbb A}(X)]\,,
\end{equation}
where $\, \Lie_n \,$ denotes (the simplicial extension of) the degree $n$ graded component of the free graded Lie algebra functor $ \Lie_*(A) = \Moplus_{n \ge 1} \Lie_n(A) \,$ on abelian groups $A$. Thus, with identifications \eqref{bform} and \eqref{Lien}, 
the morphism of towers \eqref{GoodKan} (looped once) induces on layers natural maps
\begin{equation}
\la{srn}
\SR^{(n)}_X:\ \Omega\, 
\Omega^\infty\biggl(\,\Mvee_{(n-1)!} \Sigma^{1-n} \,
(\Sigma^{\infty} X)^{\wedge n} \biggr)_{hS_n} \to \  |\,\Lie_n({\mathbb A}X)\,|\ ,
\quad n \ge 1\,.
\end{equation}
These maps can be viewed as nonlinear (polynomial) extensions of the map \eqref{srx}.
In fact, for $n=1$, 
$$
\SR^{(1)}_X: \ \Omega\,\Omega^{\infty}\Sigma^{\infty}(X)\,\to\,
|\mathbb A(X)| \,\simeq\, \Omega\,\SP^\infty(X) 
$$
while for $n=2$,  \eqref{srn} becomes
$$
\SR^{(2)}_X:\ 
\Omega\,\Omega^\infty \Sigma^{-1}(\Sigma^\infty X \wedge  \Sigma^\infty X)_{h\Z_2}\,\to\,
|\,\Lie_2({\mathbb A} X)\,|
$$
since the action of $\Z_2 $ on the spectrum $\, \partial_2(\id) \simeq \bS^{-1} $ is known to be trivial (see \cite[Example 1.2.5]{ArCh}). 

\section{Stable character maps and derived Poisson brackets} \la{S6}
\la{S5}
In this section, we study the behavior of the derived character maps \eqref{Tr*} in the limit as $\, n \to \infty $.
We show that, on simply connected spaces, these maps stabilize, inducing an isomorphism between the graded
symmetric algebra generated by the $S^1$-equivariant homology of the free loop space of $ X = B \Gamma $
and the invariant part of the representation homology in the projective limit $\varprojlim \HR_*(\Gamma, \GL_n)^{\GL_n} $. This result is a topological counterpart of a stabilization theorem proved for representation homology of  algebras in \cite{BR}. In  case when $X$ represents a closed manifold, so that its $S^1$-equivariant homology  carries the Chas-Sullivan bracket, we show that the
stable character map is an isomorphism of Lie algebras, where the Lie bracket on representation homology
is induced by a natural derived Poisson structure on the Quillen model of $X$.

\subsection{Stabilization of derived character maps}
\la{S5.1}
For this section, let $k$ be a field of characteristic $0$. The (homotopy) group homomorphism $\Gamma \to \{1\}$ (resp., $\{1\} \to \Gamma$) induces a morphism of cyclic modules $k[B^{\rm cyc}]\,\Gamma \to k[B^{\rm cyc}\, \{1\}] = k$ (resp., $k=k[B^{\rm cyc}\,\{1\}] \to k[B^{\rm cyc}\,\Gamma]$). In this way, the trivial cyclic module $k$ is a direct summand of $k[B^{\rm cyc}\,\Gamma]$ yielding a direct sum decomposition
$$k[B^{\rm cyc} \,\Gamma] \,\cong\, k \oplus k[\overline{B^{\rm cyc}\,\Gamma}]\ . $$
The {\it reduced} cyclic homology $\rHC_\ast(k[\Gamma])$ is defined by
$$ \rHC_\ast(k[\Gamma])\,:=\,{\rm Tor}^{\Delta C}_\ast(k[\overline{B^{\rm cyc}\,\Gamma}],k)\,, $$
so that
$$ \HC_\ast(k[\Gamma]) \,\cong\, \HC_{\ast}(k) \oplus \rHC_\ast(k[\Gamma]) \ . $$
On the other hand, the homomorphism of group schemes ${\rm GL_n} \hookrightarrow {\rm GL}_{n+1}$ (given by padding with $1$ on the bottom right corner) induces a morphism of commutative Hopf algebras $\cO({\rm GL}_{n+1}) \to \cO({\rm GL}_n)$, and hence, a morphism of left $\ffgr$-modules $\underline{\cO}({\rm GL}_{n+1}) \to \underline{\cO}({\rm GL}_n)$. This induces a morphism on representation homologies 
\begin{equation} \la{s3.5e1} \mu_{n+1,n}:\ {\rm HR}_\ast(\Gamma, {\rm GL}_{n+1}) = {\rm Tor}^{\ffgr}_\ast(k[\Gamma], \underline{\cO}({\rm GL}_{n+1})) \to {\rm HR}_{\ast}(\Gamma, {\rm GL}_{n}) = {\rm Tor}^{\ffgr}_\ast(k[\Gamma], \underline{\cO}({\rm GL}_{n}))\ .\end{equation}
It is not difficult to verify that \eqref{s3.5e1} restricts to a morphism on the invariant part of the representation homologies
\begin{equation} \la{s3.5e2} \mu_{n+1,n}\,:\, {\rm HR}_\ast(\Gamma, {\rm GL}_{n+1})^{\rm GL_{n+1}}  \to {\rm HR}_{\ast}(\Gamma, {\rm GL}_{n})^{{\rm GL}_n} \ .\end{equation}
\begin{lemma} \la{invlimit}
The following diagram commutes for all $n \,$: 
$$ 
\begin{diagram}[small]
\rHC_\ast(k[\Gamma]) & \rTo^{\Tr_{n+1}(\Gamma)\ } & {\rm HR}_\ast(\Gamma, {\rm GL}_{n+1})^{{\rm GL}_{n+1}}\\
          & \rdTo_{\Tr_n(\Gamma)}  & \dTo_{\mu_{n+1,n}}\\
          & & {\rm HR}_\ast(\Gamma, {\rm GL}_{n})^{{\rm GL}_n} 
\end{diagram}
$$
\end{lemma}
\begin{proof}
Since any homotopy simplicial group is weakly equivalent to a cofibrant strict simplicial group, we may assume without loss of generality that $\Gamma$ is a cofibrant strict simplicial group. Continuing to denote the map $\overline{k[\Gamma]} \otimes_{\Delta C} k \to k[\Gamma] \otimes_{\ffgr} \underline{\cO}(\GL_n)$ induced by $\Delta_{{\rm GL}_n} {\rm tr}$ by $\Tr_n(\Gamma)$, we then need to verify that the following diagram commutes
\begin{equation} \la{s3.5e3}
\begin{diagram}[small]
\overline{k[\Gamma]} \otimes_{\Delta C} k  & \rTo^{\Tr_{n+1}(\Gamma)} & k[\Gamma] \otimes_{\ffgr} \underline{\cO}({\rm GL}_{n+1})\\
          & \rdTo_{\Tr_n(\Gamma)}  & \dTo_{\mu_{n+1,n}}\\
          & &  k[\Gamma] \otimes_{\ffgr} \underline{\cO}({\rm GL}_{n})
\end{diagram}
\end{equation}
By (the proof of) \cite[Theorem 4.1]{KP}, $\,\Tr_n(\Gamma_m)$ is induced (in each simplicial degree $m$) by the composite map 
$$\begin{diagram}[small] \Gamma_m &\rTo^{\rho_n} & {\rm GL}_n(\cO[{\rm Rep}_n(\Gamma_m)]) & \rInto & \mathbb{M}_n(\cO[{\rm Rep}_n(\Gamma_m)]) & \rTo^{\Tr} & \cO[{\rm Rep}_n(\Gamma_m)] \cong k[\Gamma_m] \otimes_{\ffgr} \underline{\cO}({\rm GL_n})\,, \end{diagram}$$
where $\rho_n$ denotes the universal $n$-dimensional representation. A similar argument shows that the following diagram commutes:
\begin{diagram}[small]
\Gamma_m  & \rTo^{\rho_{n+1}\qquad} & {\rm GL}_{n+1}(\cO[{\rm Rep}_{n+1}(\Gamma_m)])\\
    \dTo^{\rho_n}      &  & \dTo_{\mu_{n+1,n}}\\
 {\rm GL}_n(\cO[{\rm Rep}_n(\Gamma_m)])         & \rInto &   {\rm GL}_{n+1}(\cO[{\rm Rep}_n(\Gamma_m)])\\
\end{diagram}
Here, the lower horizontal arrow is given by padding by `$1$' on the bottom right. It follows that 
$$\Tr_n(\Gamma_m)(\langle \gamma \rangle -1) \,=\, \mu_{n+1,n} \circ \Tr_{n+1}(\langle \gamma \rangle-1) $$
for every conjugacy class $\langle \gamma \rangle$ in $\Gamma_m$. This shows commutativity of the diagram \eqref{s3.5e3} in every simplicial degree, proving the desired lemma.
\end{proof}
By Lemma \ref{invlimit}, the family of maps $\{\Tr_n(\Gamma)\}_{n \geqslant 1}$ yields a $k$-linear map
\begin{equation} \la{s3.5e4} \Tr_\infty(\Gamma): \rHC_\ast(k[\Gamma]) \to \HR_\ast(\Gamma,{\rm GL}_\infty)^{{\rm GL}_\infty}:=\, \varprojlim_n \HR_\ast(\Gamma,{\rm GL}_n)^{\GL_n}\,, \end{equation}
where the inverse limit is taken along the maps \eqref{s3.5e2}. The map $\Tr_\infty(\Gamma)$, which we call the {\it stable character map}, induces a morphism of graded commutative $k$-algebras
\begin{equation} \la{s3.5e5} \Sym \Tr_\infty(\Gamma)\,:\,\Sym_k[\rHC_\ast(k[\Gamma])] \to \HR_\ast(\Gamma,{\rm GL}_\infty)^{{\rm GL}_\infty} \ . \end{equation}

Next, recall that a simplicial group $\Gamma$ is said to be a {\it simplicial group model} of a pointed, connected topological space $X$ if $
\Gamma$ maps to $X$ under \eqref{s3e9}, i.e. $|\bar{W}(\Gamma)|$ is weakly equivalent to $X$. In this case, it is well known that  
\begin{equation} \la{s5.2e2} \HC_\ast(k[\Gamma]) \,\cong\, {\rm H}^{S^1}_\ast(\cL X; k) \end{equation}
where $\cL X$ is the free loop space of $X$, and the representation homology $\HR_\ast(\Gamma,G)$, which is an invariant of (the homotopy type of) $X$ by Lemma \ref{RHBadM} is denoted by $\HR_\ast(X,G)$. The isomorphism \eqref{s5.2e2} restricts to an isomorphism of graded $k$-modules 
\begin{equation} \la{s5.3e3} \rHC_\ast(k[\Gamma]) \,\cong\, \overline{{\rm H}}^{S^1}_\ast(\cL X; k)\ . \end{equation}
Here, $\rH_{\ast}^{S^1}(\cL X; k)$ stands for the {\it reduced} $S^1$-equivariant homology of $\cL X$, i.e. 
$$ {\rH}^{S^1}_\ast(\cL_X; k) \,:=\,{\rm Ker}[\,\pi_\ast:{\rm H}^{S^1}_\ast(\cL X) \to \H^{S^1}_\ast({\rm pt})\,] \ .$$
The map $\pi_\ast$ is induced on $S^1$-equivariant homology by the map $\cL X \to {\rm pt}$. The derived character map $\Tr_n(X):=\Tr_n(\Gamma)$ is thus morphism of graded $k$-vector spaces
\begin{equation} \la{s5.2e4} \Tr_n(X)\,:\, \rH^{S^1}_{\ast}(\cL X;k) \to \HR_\ast(X,\GL_n)^{\GL_n}\,,\end{equation}
and the stable character map becomes
\begin{equation} \la{s5.2e5} \Tr_\infty(X)\,:\, \rH^{S^1}_{\ast}(\cL X;k) \to \HR_\ast(X,\GL_\infty)^{\GL_\infty} \ .\end{equation}
The following theorem is the main result of this section.
\begin{theorem} 
\la{SRHs}
Let $X$ be a simply connected space of finite $($rational$)$ type. The stable character map \eqref{s5.2e5} induces an isomorphism of graded commutative algebras
$$\Sym \Tr_{\infty}(X)\,:\,\Sym_k[\,\rH^{S^1}_{\ast}(\cL X;k)] \,\sto \, \HR_\ast(X,\GL_\infty)^{\GL_\infty}\ .$$
\end{theorem}

If, moreover, $X$ is a simply connected manifold of dimension $d$ then $\rH^{S^1}_\ast(\cL X; k)$ is equipped with the {\it Chas-Sullivan bracket} (also called {\it string topology bracket}), a graded Lie bracket of (homological) degree $2-d$. This Lie bracket arises out of a derived Poisson structure (in the sense on \cite[Sec. 3.1]{BRZ}) on an algebra weakly equivalent to $k[\Gamma]$. On the other hand, the representation homologies $\HR_\ast(X,\GL_n)^{\GL_n}$ are equipped with graded, ($(2-d)$-shifted) Poisson structures arising from the Poincar\'{e} duality pairing on the cohomology of $X$. Passing to the inverse limit, one obtains a graded ($(2-d)$-shifted) Poisson structure on $\HR_\ast(X,\GL_\infty)^{\GL_\infty}$. As an application of Theorem \ref{SRHs}, we obtain the following corollary which allows us to express the Chas-Sullivan bracket in terms of a graded Poisson bracket.
\begin{corollary} \la{manifold}
The map
$$\Sym \Tr_{\infty}(X)\,:\,\Sym_k[\,\rH^{S^1}_{\ast}(\cL X;k)\,] \sto  \HR_\ast(X,\GL_\infty)^{\GL_\infty}
$$
is an isomorphism of graded $(2-d)$-shifted Poisson algebras.
\end{corollary}

\subsection{Proofs of Theorem \ref{SRHs} and Corollary \ref{manifold}}
The shortest way to prove Theorem \ref{SRHs} and Corollary \ref{manifold} 
is to apply the results of the paper \cite{BR} that deals with stabilization of 
representation homology and derived character maps for (augmented) associative algebras. 
These results being applicable in our case follows from Remark \ref{RemAlg}. In what follows 
we outline key steps and necessary modifications of the arguments of \cite{BR}, leaving 
details for interested readers.


\begin{proof}[Sketch of proof of Theorem \ref{SRHs}]
Let $L_X$ denote a (cofibrant) Quillen model of $X$. Since $X$ is of finite rational type, $L_X$ may be chosen to be semi-free, and finitely generated in each homological degree. By Remark \ref{RemAlg}, if suffices to prove the assertions of this theorem working with $UL_X$ instead of $k[\Gamma]$. Further, since $X$ is simply connected, the generators of $L_X$ are in positive homological degree. Theorem \ref{SRHs} follows from (a minor modification of the proof of) \cite[Theorem 7.8]{BR}. Indeed, since $R=UL_X$ is freely generated by finitely many generators in each homological degree, and since all its generators are in positive homological degree, the arguments of \cite[Section 7.4]{BR} go through to show that for each $k>0$, the map
\begin{equation} \la{s5.2e6}\tilde{\mu}_{n+1,n}:R_{n+1}^{\GL, \leqslant k}  \to  R_{n}^{\GL, \leqslant k} \end{equation}
is an isomorphism for $n$ sufficiently large (i.e., for all $n>N(k)$ for some $N(k)$ which possibly depends on $k$). Here $R_n^{\GL}$ is the representation DG algebra as in \cite[formula (2.10)]{BR},  whose homology is isomorphic to $\HR_\ast(R,n)^{\GL_n} \cong \HR_\ast(X,\GL_n)^{\GL_n}$ and $\,R_n^{\GL,\leqslant k}$ stands for the (brutal) truncation of $R_n^{\GL}$ to homological degrees $\leqslant k$.
The map \eqref{s5.2e6} is defined as in \cite[Section 4]{BR} (where it is denoted by $\mu_{n+1,n}$). On homologies, \eqref{s5.2e6} induces the map $\mu_{n+1,n}:\HR_\ast(X,\GL_{n+1})^{\GL_{n+1}} \to \HR_\ast(X,\GL_n)^{\GL_n}$. As in the proof of \cite[Theorem 7.8]{BR} (see also Proposition 7.5 of \cite{BR}, which is the crux thereof), it then follows that the map  
$$\Sym \Tr_{\infty}(X)\,:\,\Sym_k [\rH^{S^1}_{\ast}(\cL X;k)] \to \H_\ast[R_\infty^{\GL}] $$
is an isomorphism of graded commutative algebras where $R_{\infty}^{\GL} = \varprojlim_n R_n^{\GL}$. The desired verification is thus complete once we check that $\H_{\ast}[R_{\infty}^{\GL}] \cong \varprojlim_n \H_\ast[R_n^{\GL}]$. By \eqref{s5.2e6}, the inverse system $\{R_n^{\GL}\}$ is Mittag-Leffler. \eqref{s5.2e6} further implies that for each $k$, the inverse system $\{\H_{k+1}(R_n^{\GL})\}$ stabilizes, i.e. becomes constant for large $n$, and is thus Mittag-Leffler. It follows that $\lim^1_n \H_{k+1}(R_n^{\GL})=0$. That $\H_{\ast}[R_{\infty}^{\GL}] \cong \varprojlim_n \H_\ast[R_n^{\GL}]$, as desired, then follows from \cite[Theorem 3.5.8]{W}. 
This outlines the proof of Theorem \ref{SRHs}.
\end{proof}

\begin{proof}[Sketch of proof of Corollary \ref{manifold}]
Moreover (see \cite[Section 4.2]{BRZ} for example), $L_X$ may be chosen so that its universal enveloping algebra $UL_X$ is equipped with a derived Poisson structure inducing the Chas-Sullivan bracket on its (reduced) cyclic homology (which is isomorphic to $\rH^{S^1}_\ast(\cL X;k)$). More precisely, $L_X$ may be chosen to be Koszul dual to the (graded linear dual of) the Lambrechts-Stanley model of $X$ (see \cite{LaSt}), which is equipped with a cyclic pairing. Now, if $\Gamma$ is a simplicial group model of $X$, then $k[\Gamma]$ is weakly equivalent to $UL_X$. By Remark \ref{RemAlg}, if suffices to prove the assertions of this theorem working with $UL_X$ instead of $k[\Gamma]$. In this setting, it follows immediately from \cite[Theorem 5.1]{BRZ} (also see \cite[Theorem 2]{BCER} and {\it loc. cit.}, Theorem 3.1) that the cyclic pairing on (the graded linear dual of) the Lambrechts-Stanley model of $X$ yields a graded ($(2-d)$-shifted) Poisson structure on $\HR_\ast(X,\GL_n)^{\GL_n}$ such that the derived character map $\Tr_n: \rH^{S^1}_\ast(\cL X;k) \to \HR_\ast(X,\GL_n)^{\GL_n}$ is a homomorphism of graded Lie algebras. Moreover, the maps $\mu_{n+1,n}:\HR_\ast(UL_X,n+1)^{\GL_{n+1}} \to \HR_\ast(UL_X,n)^{\GL_n}$ are easily seen to be homomorphisms of graded Poisson algebras in the setting of \cite[Section 5]{BRZ}. Hence, $\HR_\ast(X,\GL_\infty)^{\GL} \cong \HR_\ast(UL_X,\infty)^{\GL}$ acquires the structure of a graded Poisson algebra. It follows that $\Tr_\infty(X): \rH^{S^1}_{\ast}(\cL X;k) \to \HR_\ast(X,\GL_\infty)^{\GL_\infty}$ is a homomorphism of grade Lie algebras, which implies that $\Sym \Tr_\infty(X): \Sym_k[\rH^{S^1}_{\ast}(\cL X;k)] \to  \HR_\ast(X,\GL_\infty)^{\GL_\infty}$ is a homomorphism of graded Poisson algebras, where the Poisson structure in the left-hand side is obtained by extending the Chas-Sullivan bracket using the Leibniz rule.  That it is an {\it isomorphism} of graded Poisson algebras then follows from Theorem \ref{SRHs}.
\end{proof}
\bibliographystyle{plain}
\bibliography{derivedchar_bibtex}{}

\begin{thebibliography}{10}

\bibitem{ArCh}
Gregory Arone and Michael Ching.
\newblock Goodwillie calculus.
\newblock In {\em Handbook of homotopy theory}, CRC Press/Chapman Hall Handb.
  Math. Ser., pages 1--38. CRC Press, Boca Raton, FL, [2020] \copyright 2020.

\bibitem{Au1}
Shaun~V. Ault.
\newblock Symmetric homology of algebras.
\newblock {\em Algebr. Geom. Topol.}, 10(4):2343--2408, 2010.

\bibitem{Au2}
Shaun~V. Ault.
\newblock Homology operations in symmetric homology.
\newblock {\em Homology Homotopy Appl.}, 16(2):239--261, 2014.

\bibitem{Ba02}
Bernard Badzioch.
\newblock Algebraic theories in homotopy theory.
\newblock {\em Ann. of Math. (2)}, 155(3):895--913, 2002.

\bibitem{Bau}
Hans~Joachim Baues.
\newblock Homotopy types.
\newblock In {\em Handbook of algebraic topology}, pages 1--72. North-Holland,
  Amsterdam, 1995.

\bibitem{BCER}
Yuri Berest, Xiaojun Chen, Farkhod Eshmatov, and Ajay Ramadoss.
\newblock Noncommutative {P}oisson structures, derived representation schemes
  and {C}alabi-{Y}au algebras.
\newblock In {\em Mathematical aspects of quantization}, volume 583 of {\em
  Contemp. Math.}, pages 219--246. Amer. Math. Soc., Providence, RI, 2012.

\bibitem{BFPRW2}
Yuri Berest, Giovanni Felder, Sasha Patotski, Ajay~C. Ramadoss, and Thomas
  Willwacher.
\newblock Chern-{S}imons forms and higher character maps of {L}ie
  representations.
\newblock {\em Int. Math. Res. Not. IMRN}, (1):158--212, 2017.

\bibitem{BFPRW}
Yuri Berest, Giovanni Felder, Sasha Patotski, Ajay~C. Ramadoss, and Thomas
  Willwacher.
\newblock Representation homology, {L}ie algebra cohomology and the derived
  {H}arish-{C}handra homomorphism.
\newblock {\em J. Eur. Math. Soc. (JEMS)}, 19(9):2811--2893, 2017.

\bibitem{BKR}
Yuri Berest, George Khachatryan, and Ajay Ramadoss.
\newblock Derived representation schemes and cyclic homology.
\newblock {\em Adv. Math.}, 245:625--689, 2013.

\bibitem{BR}
Yuri Berest and Ajay Ramadoss.
\newblock Stable representation homology and {K}oszul duality.
\newblock {\em J. Reine Angew. Math.}, 715:143--187, 2016.

\bibitem{BR22b}
Yuri Berest and Ajay~C. Ramadoss.
\newblock Homotopy algebras and functor homology, (2022).
\newblock in preparation.

\bibitem{BRYIII}
Yuri Berest, Ajay~C. Ramadoss, and Wai-kit Yeung.
\newblock Vanishing theorems for representation homology and the derived
  cotangent complex.
\newblock {\em Algebr. Geom. Topol.}, 19(1):281--339, 2019.

\bibitem{BRYII}
Yuri Berest, Ajay~C. Ramadoss, and Wai-Kit Yeung.
\newblock Representation homology of simply connected spaces.
\newblock {\em J. Topol.}, 15(2):692--744, 2022.

\bibitem{BRYI}
Yuri Berest, Ajay~C. Ramadoss, and Wai-Kit Yeung.
\newblock Representation homology of topological spaces.
\newblock {\em Int. Math. Res. Not. IMRN}, (6):4093--4180, 2022.

\bibitem{BRZ}
Yuri Berest, Ajay~C. Ramadoss, and Yining Zhang.
\newblock Dual {H}odge decompositions and derived {P}oisson brackets.
\newblock {\em Selecta Math. (N.S.)}, 23(3):2029--2070, 2017.

\bibitem{BRZ21}
Yuri Berest, Ajay~C. Ramadoss, and Yining Zhang.
\newblock Hodge decomposition of string topology.
\newblock {\em Forum Math. Sigma}, 9:Paper No. e33, 31, 2021.

\bibitem{Berg06}
Julia~E. Bergner.
\newblock Rigidification of algebras over multi-sorted theories.
\newblock {\em Algebr. Geom. Topol.}, 6:1925--1955, 2006.

\bibitem{BD}
Georg Biedermann and William~G. Dwyer.
\newblock Homotopy nilpotent groups.
\newblock {\em Algebr. Geom. Topol.}, 10(1):33--61, 2010.

\bibitem{BK72}
A.~K. Bousfield and D.~M. Kan.
\newblock {\em Homotopy limits, completions and localizations}.
\newblock Lecture Notes in Mathematics, Vol. 304. Springer-Verlag, Berlin-New
  York, 1972.

\bibitem{BF}
D.~Burghelea and Z.~Fiedorowicz.
\newblock Cyclic homology and algebraic {$K$}-theory of spaces. {II}.
\newblock {\em Topology}, 25(3):303--317, 1986.

\bibitem{CC}
G.~E. Carlsson and R.~L. Cohen.
\newblock The cyclic groups and the free loop space.
\newblock {\em Comment. Math. Helv.}, 62(3):423--449, 1987.

\bibitem{CM95}
Gunnar Carlsson and R.~James Milgram.
\newblock Stable homotopy and iterated loop spaces.
\newblock In {\em Handbook of algebraic topology}, pages 505--583.
  North-Holland, Amsterdam, 1995.

\bibitem{CS}
Wojciech Chach\'{o}lski and J\'{e}r\^{o}me Scherer.
\newblock Homotopy theory of diagrams.
\newblock {\em Mem. Amer. Math. Soc.}, 155(736):x+90, 2002.

\bibitem{Cis}
Denis-Charles Cisinski.
\newblock Images directes cohomologiques dans les cat\'{e}gories de mod\`eles.
\newblock {\em Ann. Math. Blaise Pascal}, 10(2):195--244, 2003.

\bibitem{Dja19}
Aur\'{e}lien Djament.
\newblock D\'{e}composition de {H}odge pour l'homologie stable des groupes
  d'automorphismes des groupes libres.
\newblock {\em Compos. Math.}, 155(9):1794--1844, 2019.

\bibitem{F84}
Z.~Fiedorowicz.
\newblock Classifying spaces of topological monoids and categories.
\newblock {\em Amer. J. Math.}, 106(2):301--350, 1984.

\bibitem{F}
Zbigniew Fiedorowicz.
\newblock Symmetric bar construction, (1991).
\newblock available at
  \url{http://www.math.ohio-state.edu/~fiedorow/symbar.ps.gz}.

\bibitem{LF}
Zbigniew Fiedorowicz and Jean-Louis Loday.
\newblock Crossed simplicial groups and their associated homology.
\newblock {\em Trans. Amer. Math. Soc.}, 326(1):57--87, 1991.

\bibitem{GJ}
Paul~G. Goerss and John~F. Jardine.
\newblock {\em Simplicial homotopy theory}.
\newblock Modern Birkh\"{a}user Classics. Birkh\"{a}user Verlag, Basel, 2009.

\bibitem{Go}
Thomas~G. Goodwillie.
\newblock Cyclic homology, derivations, and the free loopspace.
\newblock {\em Topology}, 24(2):187--215, 1985.

\bibitem{CalcI}
Thomas~G. Goodwillie.
\newblock Calculus. {I}. {T}he first derivative of pseudoisotopy theory.
\newblock {\em $K$-Theory}, 4(1):1--27, 1990.

\bibitem{CalcII}
Thomas~G. Goodwillie.
\newblock Calculus. {II}. {A}nalytic functors.
\newblock {\em $K$-Theory}, 5(4):295--332, 1991/92.

\bibitem{CalcIII}
Thomas~G. Goodwillie.
\newblock Calculus. {III}. {T}aylor series.
\newblock {\em Geom. Topol.}, 7:645--711, 2003.

\bibitem{Gr}
Daniel Graves.
\newblock Hyperoctahedral homology for involutive algebras.
\newblock {\em Homology Homotopy Appl.}, 24(1):1--26, 2022.

\bibitem{Hat}
Allen Hatcher.
\newblock {\em Algebraic topology}.
\newblock Cambridge University Press, Cambridge, 2002.

\bibitem{Hir}
Philip~S. Hirschhorn.
\newblock {\em Model categories and their localizations}, volume~99 of {\em
  Mathematical Surveys and Monographs}.
\newblock American Mathematical Society, Providence, RI, 2003.

\bibitem{Hov}
Mark Hovey.
\newblock {\em Model categories}, volume~63 of {\em Mathematical Surveys and
  Monographs}.
\newblock American Mathematical Society, Providence, RI, 1999.

\bibitem{J95}
Brenda Johnson.
\newblock The derivatives of homotopy theory.
\newblock {\em Trans. Amer. Math. Soc.}, 347(4):1295--1321, 1995.

\bibitem{Kan1}
Daniel~M. Kan.
\newblock On homotopy theory and c.s.s. groups.
\newblock {\em Ann. of Math. (2)}, 68:38--53, 1958.

\bibitem{K01}
M.~Kapranov.
\newblock Injective resolutions of {$BG$} and derived moduli spaces of local
  systems.
\newblock {\em J. Pure Appl. Algebra}, 155(2-3):167--179, 2001.

\bibitem{KP}
Martin Kassabov and Sasha Patotski.
\newblock Character varieties as a tensor product.
\newblock {\em J. Algebra}, 500:569--588, 2018.

\bibitem{Kuhn}
Nicholas~J. Kuhn.
\newblock Goodwillie towers and chromatic homotopy: an overview.
\newblock In {\em Proceedings of the {N}ishida {F}est ({K}inosaki 2003)},
  volume~10 of {\em Geom. Topol. Monogr.}, pages 245--279. Geom. Topol. Publ.,
  Coventry, 2007.

\bibitem{LaSt}
Pascal Lambrechts and Don Stanley.
\newblock Poincar\'{e} duality and commutative differential graded algebras.
\newblock {\em Ann. Sci. \'{E}c. Norm. Sup\'{e}r. (4)}, 41(4):495--509, 2008.

\bibitem{L}
Jean-Louis Loday.
\newblock {\em Cyclic homology}, volume 301 of {\em Grundlehren der
  mathematischen Wissenschaften [Fundamental Principles of Mathematical
  Sciences]}.
\newblock Springer-Verlag, Berlin, second edition, 1998.
\newblock Appendix E by Mar\'{\i}a O. Ronco, Chapter 13 by the author in
  collaboration with Teimuraz Pirashvili.

\bibitem{L2}
Jean-Louis Loday.
\newblock Free loop space and homology.
\newblock In {\em Free loop spaces in geometry and topology}, volume~24 of {\em
  IRMA Lect. Math. Theor. Phys.}, pages 137--156. Eur. Math. Soc., Z\"{u}rich,
  2015.

\bibitem{LQ84}
Jean-Louis Loday and Daniel Quillen.
\newblock Cyclic homology and the {L}ie algebra homology of matrices.
\newblock {\em Comment. Math. Helv.}, 59(4):569--591, 1984.

\bibitem{LM85}
Alexander Lubotzky and Andy~R. Magid.
\newblock Varieties of representations of finitely generated groups.
\newblock {\em Mem. Amer. Math. Soc.}, 58(336):xi+117, 1985.

\bibitem{May72}
J.~P. May.
\newblock {\em The geometry of iterated loop spaces}.
\newblock Lecture Notes in Mathematics, Vol. 271. Springer-Verlag, Berlin-New
  York, 1972.

\bibitem{McC69}
M.~C. McCord.
\newblock Classifying spaces and infinite symmetric products.
\newblock {\em Trans. Amer. Math. Soc.}, 146:273--298, 1969.

\bibitem{TP21}
Tony Pantev and Bertrand To\"{e}n.
\newblock Poisson geometry of the moduli of local systems on smooth varieties.
\newblock {\em Publ. Res. Inst. Math. Sci.}, 57([3-4]):959--991, 2021.

\bibitem{PTVV}
Tony Pantev, Bertrand To\"{e}n, Michel Vaqui\'{e}, and Gabriele Vezzosi.
\newblock Shifted symplectic structures.
\newblock {\em Publ. Math. Inst. Hautes \'{E}tudes Sci.}, 117:271--328, 2013.

\bibitem{PR02}
T.~Pirashvili and B.~Richter.
\newblock Hochschild and cyclic homology via functor homology.
\newblock {\em $K$-Theory}, 25(1):39--49, 2002.

\bibitem{P2}
Teimuraz Pirashvili.
\newblock On the {PROP} corresponding to bialgebras.
\newblock {\em Cah. Topol. G\'{e}om. Diff\'{e}r. Cat\'{e}g.}, 43(3):221--239,
  2002.

\bibitem{Prid3}
Jonathan~P. Pridham.
\newblock Derived moduli of schemes and sheaves.
\newblock {\em J. K-Theory}, 10(1):41--85, 2012.

\bibitem{Prid2}
Jonathan~P. Pridham.
\newblock Constructing derived moduli stacks.
\newblock {\em Geom. Topol.}, 17(3):1417--1495, 2013.

\bibitem{Prid1}
Jonathan~P. Pridham.
\newblock Presenting higher stacks as simplicial schemes.
\newblock {\em Adv. Math.}, 238:184--245, 2013.

\bibitem{Pr}
C.~Procesi.
\newblock The invariant theory of {$n\times n$} matrices.
\newblock {\em Advances in Math.}, 19(3):306--381, 1976.

\bibitem{Q}
Daniel Quillen.
\newblock Rational homotopy theory.
\newblock {\em Ann. of Math. (2)}, 90:205--295, 1969.

\bibitem{Q2}
Daniel Quillen.
\newblock Higher algebraic {$K$}-theory. {I}.
\newblock In {\em Algebraic {$K$}-theory, {I}: {H}igher {$K$}-theories ({P}roc.
  {C}onf., {B}attelle {M}emorial {I}nst., {S}eattle, {W}ash., 1972)}, pages
  85--147. Lecture Notes in Math., Vol. 341, 1973.

\bibitem{R20}
Birgit Richter.
\newblock {\em From categories to homotopy theory}, volume 188 of {\em
  Cambridge Studies in Advanced Mathematics}.
\newblock Cambridge University Press, Cambridge, 2020.

\bibitem{Ros15}
J.~Rosick\'{y}.
\newblock Rigidification of algebras over essentially algebraic theories.
\newblock {\em Appl. Categ. Structures}, 23(2):159--175, 2015.

\bibitem{SS03}
Stefan Schwede and Brooke Shipley.
\newblock Equivalences of monoidal model categories.
\newblock {\em Algebr. Geom. Topol.}, 3:287--334, 2003.

\bibitem{Serre}
Jean-Pierre Serre.
\newblock {\em Lie algebras and {L}ie groups}, volume 1500 of {\em Lecture
  Notes in Mathematics}.
\newblock Springer-Verlag, Berlin, second edition, 1992.
\newblock 1964 lectures given at Harvard University.

\bibitem{Sik}
Adam~S. Sikora.
\newblock Character varieties.
\newblock {\em Trans. Amer. Math. Soc.}, 364(10):5173--5208, 2012.

\bibitem{T}
R.~W. Thomason.
\newblock Homotopy colimits in the category of small categories.
\newblock {\em Math. Proc. Cambridge Philos. Soc.}, 85(1):91--109, 1979.

\bibitem{TV08}
Bertrand To\"{e}n and Gabriele Vezzosi.
\newblock Homotopical algebraic geometry. {II}. {G}eometric stacks and
  applications.
\newblock {\em Mem. Amer. Math. Soc.}, 193(902):x+224, 2008.

\bibitem{T83}
B.~L. Tsygan.
\newblock Homology of matrix {L}ie algebras over rings and the {H}ochschild
  homology.
\newblock {\em Uspekhi Mat. Nauk}, 38(2(230)):217--218, 1983.

\bibitem{Wald}
Friedhelm Waldhausen.
\newblock Algebraic {$K$}-theory of topological spaces. {I}.
\newblock In {\em Algebraic and geometric topology ({P}roc. {S}ympos. {P}ure
  {M}ath., {S}tanford {U}niv., {S}tanford, {C}alif., 1976), {P}art 1}, Proc.
  Sympos. Pure Math., XXXII, pages 35--60. Amer. Math. Soc., Providence, R.I.,
  1978.

\bibitem{W}
Charles~A. Weibel.
\newblock {\em An introduction to homological algebra}, volume~38 of {\em
  Cambridge Studies in Advanced Mathematics}.
\newblock Cambridge University Press, Cambridge, 1994.

\bibitem{Wer}
Kay Werndli.
\newblock Double homotopy (co)limits for relative categories.
\newblock volume 708 of {\em Contemp. Math.}, pages 275--308. Amer. Math. Soc.,
  RI, 2018.

\end{thebibliography}
\end{document}